\documentclass{article}
\usepackage[margin=1in]{geometry}
\usepackage{hyperref}
\hypersetup{hidelinks}
\usepackage{stdpreamble}

\usepackage[toc]{appendix}
\usepackage[backend=biber,style=numeric-comp,sorting=nyt,sortcites=true,giveninits=true,maxbibnames=99]{biblatex}
\numberwithin{equation}{section}

\newcommand{\Nt}{\tilde{N}}

\DeclareMathOperator{\Ev}{Ev}

\newcommand{\G}{\mathcal{G}} 

\newcommand{\tF}{\widetilde{F}}
\newcommand{\tE}{\widetilde{E}}

\newcommand{\Lhat}{\widehat{L}}
\newcommand{\Khat}{\widehat{K}}

\newcommand{\Lcal}{\mathcal{L}}

\newcommand{\Fhat}{\widehat{F}}
\newcommand{\Ehat}{\widehat{E}}
\newcommand{\Uhat}{\widehat{U}}

\newcommand{\Hplus}{H^+}
\newcommand{\Hminus}{H^-}
\newcommand{\Hpm}{H^{\pm}}
\newcommand{\Htplus}{\widetilde{H}^+}

\newcommand{\Bplus}{B^+}
\newcommand{\Bminus}{B^-}

\newcommand{\CS}{C_{S}}

\newcommand{\ES}{E^{(S)}}
\newcommand{\FS}{F^{(S)}}
\newcommand{\US}{U^{(S)}}

\newcommand{\Fs}{F^{(s)}}
\newcommand{\Us}{U^{(s)}}

\newcommand{\EsOne}{E^{(s+1)}}
\newcommand{\FsOne}{F^{(s+1)}}

\newcommand{\ESp}{E^{(S)'}}
\newcommand{\FSp}{F^{(S)'}}
\newcommand{\USp}{U^{(S)'}}

\DeclareMathOperator{\SVS}{SVS}

\newcommand{\tC}{\tilde{C}}

\newcommand{\Mcal}{\mathcal{M}}

\newcommand{\Ccal}{\mathcal{C}}
\newcommand{\Ical}{\mathcal{I}}

\newcommand{\ann}[1]{{#1}^{\circ}} 
\newcommand{\preann}[1]{\prescript{\circ}{}{#1}} 

\DeclareMathOperator{\adj}{adj} 

\DeclareMathOperator{\vol}{vol}

\newcommand{\detLL}{\widetilde{\det}}
\newcommand{\VLL}{\widetilde{V}}

\newcommand{\measurable}{\operatorname{(M)}}

\newcommand{\phiq}{\vphi^{(q)}}

\DeclareMathOperator{\OSE}{OSE}

\newcommand{\nt}{\tilde{n}}

\title{Effective Multiplicative Ergodic Theorems for Banach Space Cocycles}
\author{Adam Zaidan\thanks{School of Mathematics, Georgia Institute of Technology, Atlanta, GA 30332, USA. Email: \texttt{adam.zaidan@gatech.edu}}}
\date{October 22, 2026}

\begin{document}
\maketitle

\begin{abstract}
    The multiplicative ergodic theorem (MET) describes the asymptotic growth rates of $q$-volumes under the iterates of an operator \textit{cocycle}, i.e., compositions of operators driven by a stationary process, organizing the underlying space into equivariant filtrations or splittings into so called Oseledets subspaces. For operator cocycles in finite dimensions, and more generally in Hilbert spaces, one can control the geometry of the Oseledets subspaces given by the MET via the convergence rates of finite-time volume growth data. This makes the MET \textit{effective} for matrix and Hilbert-space operator cocycles. We extend this effective theory to cocycles of Banach-space operators, obtaining quantitative angle estimates that are new even in the finite-dimensional setting. We take a sequential approach \`{a} la Raghunathan and Ruelle, allowing us to separate the geometric argument of the MET from measure-theoretic considerations. This permits weaker cocycle measurability hypotheses for the non-invertible and semi-invertible METs on a broad class of Banach spaces, including nonseparable ones. As an application, we establish the MET for transfer operator cocycles on $\BV[0,1]$ associated with jointly measurable random Lasota-Yorke maps.
\end{abstract}

\section{Introduction}\label{sec:intro}
Let $(\Omega,\F, \p)$ be a probability space and let $\sigma:\Omega\to \Omega$ be an invertible ergodic measure preserving transformation (mpt). Let $(V,\norm{\cdot})$ be a finite-dimensional normed space and suppose $A:\Omega\to L(V)$ is a measurable map into the space of linear operators. For $\omega\in \Omega$, the Lyapunov exponent of the cocycle $n \mapsto A^n_\omega = A_{\sigma^{n-1}\omega}\cdots A_{\omega}$ along $v\in V$ is the exponential rate
\[\lambda(\omega,v) = \limsup_{n\to\infty} \frac{1}{n}\log \norm{A^n_\omega v}.\]
Under a mild integrability condition on $\norm{A_\omega}$, the (semi-invertible) MET provides for $\p$-typical $\omega\in \Omega$ an a.s. constant finite Lyapunov spectrum $\lambda_1 > \cdots > \lambda_r \geq -\infty$, $1\leq r\leq \dim V$, and a splitting of $V$ into Oseledets spaces
\[V = U_1(\omega)\oplus\cdots\oplus U_r(\omega)\]
playing the role of generalized eigenspaces: they are equivariant (i.e., $A_\omega U_s(\omega) \subseteq U_s(\sigma\omega)$), and if $1\leq s\leq r$ and $v\in U_s(\omega)\setminus\innerp{0}$, then $\lambda(\omega,v) = \lambda_s$; in fact the $\limsup$ in $\lambda(\omega, v)$ is a convergent limit \cite{Oseledets1968,Raghunathan1979,Ruelle1979,Walters1993,FroylandLloydQuas2010}.

For cocycles of bounded linear operators on infinite-dimensional Hilbert \cite{Ruelle1982,SchaumloffelFlandoli1991} or Banach spaces \cite{Mane1983, Thieullen1987, LianLu2010, FroylandLloydQuas2013,GonzalezTokmanQuas2015,GonzalezTokmanQuas2014,Blumenthal2016,Lee2024}, we may extend the MET under similar integrability assumptions. However, in the infinite-dimensional setting, the measurability hypotheses required on the operator cocycle can be quite stringent, and there is an interplay between separability and cocycle measurability (see \autoref{subsec:litReview}). Nevertheless, if the measurability and integrability requirements are satisfied, the Lyapunov spectrum splits into an exceptional part and a residual portion, much like the spectrum of a single operator, whose exceptional spectrum is the set of eigenvalues above the essential spectral radius in magnitude. The exceptional Lyapunov spectrum is the (at most countable) set of exponential rates resolved by the MET, giving us a splitting of $V$ into the resolved $(U_s)$'s and, unlike the finite-dimensional case, an unresolved component $F_\alpha$, for which growth rates are dominated by the exceptional spectrum.

A central application of the MET is to derivative cocycles of smooth systems; it is foundational for the theory of non-uniformly hyperbolic systems in finite \cite{BarreiraPesin2001,Ruelle1979,GouezelStoyanov2017} and infinite dimensions \cite{Ruelle1982,LianLu2010,BlumenthalYoung2017}. Another fundamental application, pertaining specifically to infinite-dimensional Banach spaces, is the extension of spectral theory for transfer operators of deterministic maps \cite{LasotaYorke1973,Baladi2000} to random dynamical systems \cite{FroylandLloydQuas2010,FroylandLloydQuas2013,horan2020spectral,DS21quenchedLimThms,GonzalezTokmanQuas2022,Lee2023}. Indeed, for deterministic systems, the spectral theory of transfer operators can be used to obtain invariant densities and long-time statistical behaviors such as correlation decay. The analogous object to study for random systems would be quenched transfer operator cocycles \cite{Buzzi1999, DS21quenchedLimThms}, with the natural spectral-theory-analogue for cocycles being the MET.

\subsubsection*{Main Contributions of this Paper} 
In a Hilbert space, if $s_1(A)\geq s_2(A)\geq \cdots$ denote the singular values of an operator $A$, then the (exceptional) Lyapunov spectrum, counted with multiplicity, is given a.s. by the limits
\begin{equation}\label{eq:singularValueLimsForLE}
    \lim_{n\to\infty} \frac{1}{n} \log s_i(A^n_\omega) \text{ for $i \geq 1$.}
\end{equation}
Once these limits are obtained, a standard method to prove the MET is the Ruelle-Raghunathan strategy: for a fixed $\p$-typical $\omega$, one proves a sequential version of the MET utilizing singular value decompositions on the cocycle \cite{Raghunathan1979,Ruelle1979,Ruelle1982}. The primary advantage of this approach is its effective nature.
\begin{defn}\label{defn:modulusOfConvergence}
    For a metric space $(X,d_X)$ and a convergent sequence $x_n\to x$ in $X$, the map $N:(0,\infty)\to \N$ is a \textit{modulus of convergence} for $(x_n)_n$ if $d_X(x_n,x)<\e$ for all $n\geq N(\e)$, where $\e > 0$.
\end{defn}

Roughly speaking, an estimate is effective in a convergent sequence $x_n\to x$ if its defect can be expressed in terms of a modulus of convergence of $(x_n)_n$ (we discuss this idea more precisely in \autoref{remark:effectiveEstimates}). Thus, in this sense, the Ruelle-Raghunathan strategy produces growth estimates on the Oseledets subspaces effective in the limits given in \eqref{eq:singularValueLimsForLE}. Unfortunately, this strategy makes great use of the structures afforded by a Hilbert space, and as a result, all known treatments of the MET for Banach space cocycles do not provide effective growth estimates. Indeed, to circumvent the lack of an inner product structure, such treatments require an inductive procedure that peels off one Oseledets space at a time \cite{Mane1983,Thieullen1987,LianLu2010,FroylandLloydQuas2013,GonzalezTokmanQuas2014,GonzalezTokmanQuas2015,Blumenthal2016,Lee2024}. The key idea is that, once the top $s-1$ fast spaces are peeled off, the remaining portion should be a finite-codimensional Oseledets slow space $F_{s}(\omega)$ for which the exceptional Lyapunov spectrum is dominated by $\lambda_{s}$. Through this procedure, estimates are at most effective in the maximum growth of the cocycle restricted to $F_{s}(\omega)$, i.e., in
\begin{equation}\label{eq:}
    \frac{1}{n}\log \norm{A^n_\omega|_{F_{s}(\omega)}}\to \lambda_s
\end{equation}
rather than in growth information provided by the unrestricted cocycle as in \eqref{eq:singularValueLimsForLE}. As such, these current approaches do not permit direct control of the finite-time Lyapunov-Oseledets geometry via finite-time growth data of the cocycle. This is quite restrictive, as the ability to control the geometry via such finite-time data is pivotal in, e.g., the quantitative theory of hyperbolic dynamics \cite{GouezelStoyanov2017}. 

There is another disadvantage to the inductive approach that becomes apparent when extending the spectral theory of transfer operators to random systems, and, more generally, applying the MET to cocycles on nonseparable spaces. These inductive peeling procedures require the use of measure-theoretic and ergodic-theoretic tools at each inductive step in a way that does not yield a direct adaptation to a ``fixed $\omega$" sequential theorem. On nonseparable Banach spaces, the measurability requirements needed on the cocycle are often too strong. This is the case for transfer operators of random piecewise expanding interval maps, whose natural spectral-theoretic home is the nonseparable space of functions of bounded variation \cite{Buzzi1999,DS21quenchedLimThms}. Indeed, for such random maps $\omega\mapsto f_\omega:[0,1]\to[0,1]$, measurability requirements for previously known Banach space METs force $\omega\mapsto f_\omega$ to take on an at most countable collection of interval maps if one would like to apply the MET to the associated transfer operator cocycle on $\BV$ (see \cite[Remark 4.5.7.]{horan2020spectral} for an example).

The purpose of this paper is twofold:
\begin{enumerate}[label=(\Alph*)]
    \item bridge the effective treatment of the Hilbert space MET to Banach space cocycles; and
    \item establish a flexible measurability framework that allows for, among other things, applicability of the MET to random transfer operators on $\BV$ associated with general jointly measurable RDS $f:\Omega\times[0,1]\to [0,1]$ (subject of course to the requisite integrability conditions).
\end{enumerate}
We shall accomplish both of these goals by directly extending the Ruelle-Raghunathan framework to Banach space cocycles and deriving estimates in this framework beyond those provided by the original treatments. These estimates demonstrate a direct relationship between the Oseledets geometry of a cocycle and its finite-time growth data, accomplishing goal (A). Goal (B) shall follow as a by-product: our effective treatment allows us to separate the geometric properties of the Lyapunov spectrum from measure theoretic considerations, avoiding the repeated use of ergodic-theoretic arguments that force stringent measurability conditions on cocycles. Our measurability framework should extend well beyond transfer operators, with potential applications including cocycles induced by linearized PDE and SPDE on nonseparable function spaces.

\subsection{Statement of Results}\label{subsec:statementOfResults}
Our first two theorems will be results for one-sided sequences of bounded linear operators; these will form the basis of our effective theory. Our third theorem will showcase our extended measurability framework and multiplicative ergodic theorems through an application to transfer operator cocycles on $\BV[0,1]$ under joint measurability of the base random dynamical system. Four theorems, consisting of two sequential results for two-sided sequences and the non-invertible and semi-invertible METs, shall be stated in their relevant sections (see \autoref{subsec:paperStructure}).

\subsubsection*{The Sequential Theorems}
The sequential theorems state the qualitative consequences of our effective treatment; we discuss the effective forms of these theorems in \autoref{remark:effectiveEstimates}. 

We must define a few preliminary notions. For a Banach space $\B$ let $B_\B$ be the closed unit ball centered at the origin and write $B_E = E\cap B_\B$ for a linear subspace $E\subseteq \B$. For two linear subspaces $E,F\subset\B$, the \textit{minimal angle} $\theta(E,F)\in \bracket{0,\frac{\pi}{2}}$ from $E$ to $F$ is defined by
\[\sin\theta(E,F) = \inf \curly{d(e,F):e\in E \text{ and } \norm{e} = 1}.\]
We will also use the following notions of volume and determinant due to Blumenthal \cite{Blumenthal2016} (Lian and Lu formulate another notion of volume in \cite{LianLu2010}). Let $E\subset \B$ be a $q$-dimensional subspace for some $q\in \N = \{1,2,3,\ldots\}$. We define the \textit{induced volume} $m_E$ to be the unique Haar measure on $E$ for which $m_E(B_E) = \omega_{q}$, where $\omega_{q}$ is the volume of the Euclidean unit ball in $\R^{q}$.

\begin{defn}\label{defn:determinants}
    Let $\B$ and $\B'$ be Banach spaces. Let $A\in L(\B,\B')$ be a bounded linear operator and suppose $E\subset \B$ is a finite-dimensional subspace. We define the \textit{determinant} $\det(A|E)$ of $A$ on $E$ by
    \[\det(A|E) = \begin{cases}
        \frac{m_{AE}(AB_{E})}{m_E(B_E)} & A|_E \text{ is injective,} \\
        0 & \text{otherwise.}
    \end{cases}\]
\end{defn}
\begin{defn}\label{defn:BlumenthalMaximalVolGrowth}
    Let $\B$ and $\B'$ be Banach spaces. Let $q\in \N$ and $A\in L(\B,\B')$. We define the function $V_q:L(\B,\B')\to [0,\infty)$ by
    \[V_q(A) = \sup \curly{\det(A|E):\dim E = q}.\]
\end{defn}

We introduce the conventions that $V_0(A) = 1$ and $V_q(A) = 0$ when $q > \dim \B$. We now give the statement of our first ``sequential MET." Let $\{\B_n\}_{n\geq 0}$ be a collection of Banach spaces and $A_n:\B_n\to\B_{n+1}$ a sequence of bounded linear operators. 

\begin{assumption}\label{assump:oneSidedSequences}
    Write $A^n = A_{n-1}\cdots A_0$.
    \begin{enumerate}[label = (\roman*)]
        \item $\displaystyle \limsup_{n\to\infty} \frac{1}{n}\log \norm{A_n}_{\B_n\to \B_{n+1}}\leq 0$; and
        \item $\displaystyle l_q = \lim_{n\to\infty} \frac{1}{n} \log V_q(A^n)$ exists for all $q\geq 1$ and $l_1 < \infty$.
    \end{enumerate}
\end{assumption}

\begin{thm}\label{thm:sequentialMET_oneSided}
    Let $(A_n)_{n\geq 0}$ be a sequence of operators subject to \autoref{assump:oneSidedSequences}. Write $\chi_1 = l_1$ and $\chi_q = l_q-l_{q-1}$ for $q>1$ with $l_q > -\infty$, and $\chi_q = -\infty$ if $l_q = -\infty$ so that $(\chi_q)_{q\geq 1}$ is a decreasing sequence and
    \[l_\alpha = \inf_{q\in \N} \frac{1}{q} l_q = \inf_{q\in\N} \chi_q \geq -\infty\]
    exists. If $\chi_1>l_\alpha$, denote by $\lambda_1>\lambda_2>\cdots > -\infty$ the distinct values of $(\chi_q)_q$ with finite multiplicity $m_s$. Let $1\leq r\leq \infty$ be the number of distinct values with finite multiplicity and write $M_s = m_1+\cdots+m_s$ for $1\leq s \leq r$. Then, there is a filtration of linear subspaces
        \[\B_0 = F_1\supsetneq F_2 \cdots \supsetneq F_{\alpha} \supseteq \innerp{0}\]
    such that the following hold.
    \begin{enumerate}[label = (\arabic*)]
        \item For $1\leq s < r+1$, $F_s$ is closed and $M_{s-1}$-codimensional, and $F_\alpha$ is closed and $M_r$-codimensional. If $r = \infty$, then
        \[F_\alpha = \bigcap_{t=1}^\infty F_t.\]
        \item For $1\leq s < r+1$, the maximal exponential growth rate of $A^n|_{F_s}$ is $\lambda_s$ in the sense that
        \[\lim_{n\to\infty}\frac{1}{n}\log\sup_{v\in F_s\setminus\{0\}}\frac{\norm{A^n v}_{\B_n}}{\norm{v}_{\B_0}} = \lambda_s.\]
        Additionally, $\limsup_{n\to\infty} \frac{1}{n}\log\norm{A^n|_{F_\alpha}}\leq l_\alpha$, with equality attained if $r<\infty$.
        \item Writing $\Ccal_{s}(\eta) \coloneqq \{v\in F_{s}:d(v,F_{s+1})\geq \eta\norm{v}\}$ for $\eta\in (0,1]$ and $1 \leq s < r$, if $\Ccal_s(\eta)\neq \{0\}$, we have
        \[\lim_{n\to\infty} \frac{1}{n} \log \inf_{v\in \Ccal_s(\eta)\setminus\{0\}}\frac{\norm{A^n v}_{\B_n}}{\norm{v}_{\B_0}} = \lambda_{s}.\]
        If $r<\infty$, then the analogous statement holds writing $s = r$ and identifying $F_{r+1} = F_\alpha$.
        \item Writing $\G\Ccal_s(\eta) \coloneqq \{E\subset \B_0:\dim E = M_s \text{ and } \sin\theta(E,F_{s+1})\geq \eta\}$ for $\eta \in (0,1]$ and $1\leq s < r$, if $\G\Ccal_s(\eta)\neq \varnothing$, we have
        \[\lim_{n\to\infty}\frac{1}{n}\log \inf_{E\in \G\Ccal_s(\eta)}\det(A^n|E) = m_1\lambda_1+\cdots+m_{s}\lambda_{s} = l_{M_s}.\]
        If $r<\infty$, then the analogous statement holds writing $s = r$ and identifying $F_{r+1} = F_\alpha$.
    \end{enumerate}
\end{thm}
\begin{remark}\label{remark:limsupDefnOfSlowSpaces}
    If $\chi_1 = l_\alpha$, we shall also define $r = 0$. Sequences of operators for which $\chi_1 > l_\alpha$ are called \textit{quasicompact}. In the quasicompact case, the values $(\lambda_s)_{1\leq s < r+1}$ are called the \textit{exceptional} Lyapunov exponents. Moreover, parts (1)-(3) imply
    \[F_{s} = \curly{v\in \B_0: \limsup_{n\to\infty}\frac{1}{n} \log \norm{A^n v}_{\B_n}\leq \lambda_s}\]
    for $1\leq s < r+1$, with the analogous statement holding if one replaces $F_s$ by $F_\alpha$ and $\lambda_s$ by $l_\alpha$, even if $r = \infty$.
\end{remark}

We may further refine the above theorem if we have slightly stronger assumptions on our sequence of operators that are exponentially uniform in base space.
\begin{assumption}\label{assump:refined_oneidedSequences}
    For $m\geq 0$, write $A_m^n = A^n(m) = A_{m+n-1}\cdots A_{m}$.
    \begin{enumerate}[label = (\roman*)]
        \item For all $m\geq 0$, the sequence $(A_m,A_{m+1},A_{m+2},\ldots)$ satisfies \autoref{assump:oneSidedSequences} and the Lyapunov exponents are constant w.r.t. $m$.
        \item For $q\in \N$ such that $l_q > -\infty$, we have
        \[\lim_{n,k\to\infty}\frac{1}{n+k} \abs{\log V_{q}\paren{A^k_{n}}-kl_{q}} = 0.\]
        \item If $q\in \N$ is the smallest integer such that $l_q = -\infty$, then for all $\Delta \in \R$,
        \[\limsup_{n,k\to\infty} \frac{1}{n+k}\paren{\log V_q(A^k_n) + k\Delta} \leq 0\]
    \end{enumerate}
\end{assumption}
We shall say a sequence of linear subspaces $V_{m}\subset \B_m$ is \textit{equivariant} if $A_m V_m\subseteq V_{m+1}$. For an operator $A:\B\to \B'$, where $\B$ and $\B'$ are Banach spaces, and a subspace $E\subseteq \B$, we define the \textit{minimum norm} by
\[m(A|_E) = \inf\curly{\norm{Ae}:e\in E\text{ and }\norm{e} = 1}.\]
We also define the \textit{annihilator} by
\[\ann{E} = \curly{\ell\in \B^*: \innerp{v,\ell} = 0\ \forall v\in E}.\]
\begin{thm}\label{thm:refined_sequentialMET_oneSided}
    Let $(A_n)_{n\geq 0}$ be a sequence of bounded linear operators subject to \autoref{assump:refined_oneidedSequences}, and assume the sequence is quasicompact. For $m\geq 0$, let
    \[\B_m = F_1(m)\supsetneq F_2(m)\supsetneq\cdots\supsetneq F_\alpha(m)\supseteq \innerp{0}\]
    be the corresponding filtration of slow spaces for $(A_{m+n})_{n\geq 0}$. Then the filtration is equivariant. Moreover, we have the following.
    \begin{enumerate}[label = (\arabic*)]
        \item For $1\leq s < r$ and $\eta\in (0,1]$, if $\G\Ccal_s(\eta)\neq \varnothing$, we have
        \[\lim_{n\to\infty} \frac{1}{n} \log \inf_{E\in \G\Ccal_s(\eta)} \sin\theta(A^n_0 E, F_{s+1}(n)) = 0.\]
        If $r<\infty$, then the analogous statement holds writing $s = r$ and identifying $F_{r+1}(n) = F_\alpha(n)$.
        \item For $1\leq s < r$, writing $\Ehat_s(n) = \ann{F_{s+1}(n)}$ we have
        \[\lim_{n\to\infty} \frac{1}{n}\log m\paren{(A^n_0)^*|_{\Ehat_s(n)}} = \lambda_s\]
        and
        \[\lim_{n\to\infty} \frac{1}{n} \log \det((A^n_0)^*|\Ehat_s(n)) = \sum_{t=1}^s m_t \lambda_t.\]
        If $r<\infty$, then the analogous statement holds writing $s = r$ and identifying $\Ehat_{r}(n) = \ann{F_\alpha(n)}$.
    \end{enumerate}
    The corresponding conclusions for the subspaces of $\B_m$ follow by applying the theorem to the sequence $(A_{m+n})_{n\geq 0}$.
\end{thm}
We use the term \textit{logarithmic temperedness} for sublinear growth of logarithms of geometric quantities, a property usually referred to simply as \textit{temperedness} in the literature. Item (1) is a sequential analogue of the usual notion of temperedness for angles between Oseledets spaces.

\begin{remark}[Effective estimates]\label{remark:effectiveEstimates}
Though we describe the Ruelle-Raghunathan strategy as effective, we note that neither Raghunathan \cite{Raghunathan1979} nor Ruelle \cite{Ruelle1979,Ruelle1982} provides an explicit connection between the convergence rates of singular values of the cocycle and growth properties of the cocycle on Oseledets spaces. Following Gou\"{o}zel and Stoyanov \cite{GouezelStoyanov2017}, we shall base our effective theory on large deviations bounds for finite-time volume growth rates. This is best illustrated through an example. 

Consider the case of \autoref{thm:sequentialMET_oneSided}. For the sequence $(A_n:\B_n\to\B_{n+1})_{n\geq 0}$, and for $q\in \N$ and $\delta > 0$, define the deviations-type quantities
\begin{itemize}
    \item $\Bplus(0,\delta) = \sup_{n\geq 0} \log \norm{A_{n}} - n \delta$; and
    \item $\Hplus_q(0,\delta) = \sup_{n\geq 0} \abs{\log V_q(A^n) - nl_q} - n\delta$.
\end{itemize}
In proving item (2), we obtain the following estimate: let $\delta>0$ be small enough (dependent on the gaps in the Lyapunov spectrum), and fix $B\geq \Bplus(0,\delta)$ and $H\geq \Hplus_1(0,\delta),\ldots,\Hplus_{1+M_{s-1}}(0,\delta)$. Then there exist constants $C\geq 1$ and $N\in \N$, both depending only on $\delta, H, B$ and $\lambda_1,\ldots,\lambda_s$ (including multiplicities) such that
\[\norm{A^n|_{F_s}}\leq C\exp\paren{n\paren{\lambda_s + 7(s-1)^2 \delta}}\]
for all $n\geq N$ (see \eqref{eq:effectiveFsMaxGrowthRate}). The above is what we mean by an effective estimate, as from the convergence data of the limits in \autoref{assump:oneSidedSequences}, we obtain a direct relationship to finite-time growth properties of the operator sequence on its slow spaces. Our estimates shall take this form throughout the paper. 

The effective treatment of \autoref{thm:sequentialMET_oneSided} rests on the quantitative convergence and gap estimates for approximate slow spaces given by Lemmas \ref{lem:cauchySeqSpaces}, \ref{lem:forwardGapSepEstimate}, and \ref{lem:backwardGapSepEstimate}. For \autoref{thm:refined_sequentialMET_oneSided}, an effective lower bound on angles is given by \eqref{eq:effectiveForwardAngleTempering}. For the two-sided theory, an effective lower bound on the separation between the equivariant fast and slow spaces is given by \eqref{eq:equivariantSlowFastSplittingAngles}.
\end{remark}

\subsubsection*{Transfer Operators for Random Lasota-Yorke Maps}
For our principal application, we consider random Lasota-Yorke interval maps in the setting of \cite{Buzzi1999}. These are jointly measurable random interval maps that are piecewise monotonic and nonsingular (p.m.n.s.), and satisfy certain covering, expansion, and integrability conditions. Precise hypotheses on $f$, and the precise statement of the following theorem, are given in \autoref{subsec:randomLYmaps}.

Now, let $(\Omega,\F,\p)$ be a probability space, and let $\sigma:\Omega\to\Omega$ be an ergodic measure-preserving transformation (mpt) with measurable inverse. Take $f:\Omega\times [0,1]\to [0,1]$ to be a jointly measurable random Lasota-Yorke map and note that it induces a random dynamical system $(\omega,x)\mapsto f^n_\omega(x) = f_{\sigma^{n-1}\omega}\circ\cdots\circ f_\omega(x)$. Denote by $\Lcal_\omega^n$ the Perron-Frobenius operator associated with $f^n_\omega$, which satisfies the cocycle relation 
\[\Lcal^n_\omega = \Lcal_{\sigma^{n-1}\omega} \circ \cdots \circ \Lcal_{\omega}.\]
For a random Lasota-Yorke map, $\Lcal_\omega$ is a bounded linear operator from $\BV[0,1]\to \BV[0,1]$, the space of $L^1$-functions on $[0,1]$ of bounded variation; this is a nonseparable Banach space in its standard norm. We shall view $\omega\mapsto \Lcal_\omega$ as a cocycle on $\BV[0,1]$.

\begin{thm}[Summary]\label{summaryThm:transferOps_cocycle}
    The cocycle $\omega\mapsto \Lcal_\omega$ is quasicompact, and its Lyapunov spectrum is a.s. constant. The top Lyapunov exponent is $\lambda_1 = 0$, and it is simple. The top Oseledets fast space is spanned by the a.s. unique invariant density of the RDS induced by $f$, and the codimension-one slow space is a.s. equal to $\BV_0[0,1] \coloneqq \curly{u\in \BV[0,1]:\int_0^1 u\,dx = 0}$. Moreover, the Oseledets fast spaces are strongly measurable w.r.t. the $L^1$-Grassmannian, while the Oseledets slow spaces are strongly measurable w.r.t. the $\BV[0,1]$-Grassmannian.
\end{thm}

\subsection{Comparison with Previous Work}\label{subsec:litReview}
\paragraph{On Effectiveness} 
We note that the works of Blumenthal \cite{Blumenthal2016}, Gonz\'{a}lez-Tokman and Quas \cite{GonzalezTokmanQuas2015}, and Lee \cite{Lee2024} yield partially effective estimates. In particular, effective estimates on the convergence of approximate slow spaces to the Oseledets slow spaces may be obtained by the methods in \cite{Blumenthal2016} and \cite{GonzalezTokmanQuas2015}, while the analogous estimates for fast spaces may be obtained by the methods in \cite{Lee2024}. However, as previously noted, their arguments are ultimately inductive and thus do not yield growth estimates effective in the unrestricted cocycle's finite-time volume growth data. Our approach avoids the inductive peeling method, but it takes inspiration from these previous works: we build upon the volume-based methods of \cite{Blumenthal2016} and \cite{GonzalezTokmanQuas2015} and make use of the elegant ``midpoint" construction for fast spaces developed in \cite{Lee2024}.

In terms of effectiveness for Banach space cocycles, the most closely aligned work is that of Quas, Thieullen, and Zarrabi \cite{QuasThieullenZarrabi2019}, extending earlier work of Blumenthal and Morris \cite{BlumenthalMorris2015} and Bochi and Gourmelon \cite{BochiGourmelon2009}. The result of Bochi and Gourmelon may be thought of as a topological analogue of the MET for dominated splittings in finite dimensions. Blumenthal and Morris provide a constructive generalization to Banach space cocycles, but their results are only effective in finite dimensions. Those of Quas, Thieullen, and Zarrabi are both constructive and effective, and much like our approach, they abstract away the need for a base dynamical system by considering (two-sided) sequences of operators subject to uniform two-sided analogues of Assumptions \ref{assump:oneSidedSequences} and \ref{assump:refined_oneidedSequences} (at a fixed index). The key difference with ours, however, is that their goal is to construct a dominated splitting, and this is reflected in the strong uniformity present in their assumptions. 

In particular, our \autoref{thm:sequentialMET_oneSided} only requires existence of one-sided volume growth limits at a fixed initial time, without requiring any uniformity of convergence across time-shifts. Such estimates, formalized as (SVG) and (FI), are pivotal to the approach of \cite{QuasThieullenZarrabi2019}. On the other hand, our approach to \autoref{thm:sequentialMET_oneSided} requires full knowledge of the Lyapunov spectrum up to a fixed index (as in \cite{Ruelle1982}), whereas \cite{QuasThieullenZarrabi2019} requires a spectral gap at only one index.

Our \autoref{thm:refined_sequentialMET_oneSided} is closer in spirit to \cite{QuasThieullenZarrabi2019}. Like \cite{QuasThieullenZarrabi2019}, we also construct ``moving" approximate slow-spaces. However, we only require subexponential control of volume growth defects along time, rather than the uniform requirements of (SVG) and (FI). As a consequence, we obtain effective estimates that measure the at most subexponential decay of angles between fast and slow directions under non-uniform hypotheses. To the best of our knowledge, these estimates are new even in the case of finite-dimensional semi-invertible cocycles. All previously known proofs of the semi-invertible MET (even in finite dimensions) establish logarithmic temperedness of angles via ergodic-theoretic arguments, and as such, do not derive explicit angle estimates based on finite-time volume growth data.

\paragraph{On Measurability Issues} Ruelle \cite{Ruelle1982} was the first to extend the MET to the infinite-dimensional setting, treating strongly measurable cocycles generated by compact operators or by invertible unitary-plus-compact operators on a separable Hilbert space. Subsequently, Ma\~{n}\'{e} \cite{Mane1983} generalized the MET to compact injective Banach space operators under a continuity condition w.r.t. the underlying probability space. Thieullen \cite{Thieullen1987} weakened Ma\~{n}\'{e}'s compactness condition to \textit{quasicompactness}. Schauml\"offel and Flandoli \cite{SchaumloffelFlandoli1991} then adapted Thieullen's proof to strongly measurable quasicompact cocycles on a separable Hilbert space, weakening Ruelle's compactness condition. Blumenthal \cite{Blumenthal2016} weakened Thieullen's continuity condition to the still rather strong condition of Bochner measurability for quasicompact cocycles on a nonseparable Banach space. Lian and Lu \cite{LianLu2010} consider injective cocycles on separable spaces, for which they show strong measurability is sufficient. We lastly note that Froyland, Lloyd, and Quas \cite{FroylandLloydQuas2013} extend their semi-invertible MET from finite-dimensions to Banach space cocycles under Ma\~{n}\'{e}'s continuity condition, while Gonz\'{a}lez-Tokman and Quas \cite{GonzalezTokmanQuas2014} weaken the requirement to strong measurability if the Banach space is separable.

As previously mentioned, the measurability requirements from prior treatments are a serious obstruction to the study of random transfer operators on $\BV$ \cite{horan2020spectral, DS21quenchedLimThms}. Gonz\'{a}lez-Tokman and Quas \cite{GonzalezTokmanQuas2022} recover partial Lyapunov-Oseledets data for a class of random transfer operator cocycles on $\BV$ by applying the MET to the cocycle on an auxiliary separable Banach space, yielding a partial MET interpretation for the $\BV$ cocycle associated with their family of systems. This is not a direct application of the MET on $\BV$. In contrast, our measurability framework permits direct application of the MET to a broad class of Banach spaces and cocycles, including random transfer operator cocycles on $\BV$ for jointly measurable random Lasota-Yorke maps, as in \cite{Buzzi1999}.

\subsection{Structure of the Paper}\label{subsec:paperStructure}
In \autoref{sec:geometry}, we review some preliminary results on geometry in a Banach space. We shall review the Grassmannian, angles, notions of volume and volume growth, and a generalization of the singular value decomposition due to Blumenthal and Morris \cite{BlumenthalMorris2015}. We shall also make slight modifications to some of these results.

In \autoref{sec:sequentialMET_oneSided}, we prove Theorems \ref{thm:sequentialMET_oneSided} and \ref{thm:refined_sequentialMET_oneSided}. In \autoref{sec:sequentialMET_twoSided}, we generalize the previous one-sided sequential theorems to two-sided sequences of operators, resulting in \autoref{thm:sequentialMET_twoSided}, which generalizes the first theorem, and \autoref{thm:refined_sequentialMET_twoSided}, which generalizes the second. The treatment of the one-sided and two-sided sequential theorems is effective. Both \autoref{thm:refined_sequentialMET_oneSided} and \autoref{thm:refined_sequentialMET_twoSided} are new in the case of finite-dimensional one-sided and two-sided matrix sequences, respectively.

In \autoref{sec:MET}, we formulate an existential measurability framework for cocycles. Then, under this weak cocycle measurability framework and the usual integrability assumption, we prove our extensions of the non-invertible MET (\autoref{thm:non-invMET}), and for invertible base dynamics, the semi-invertible MET (\autoref{thm:semi-invMET}). Using a strengthening of Kingman's subadditive ergodic theorem (\autoref{appThm:temperedSET}), for non-invertible and semi-invertible cocycles, we show the assumptions of the one-sided and two-sided sequential theorems, respectively, hold almost surely. In \autoref{sec:applications}, we provide checkable conditions which can be used to verify if a Banach space and cocycle pair satisfies our existential measurability framework. We then study some of the functional analytic properties of $\BV[0,1]$ and transfer operators for random Lasota-Yorke maps and show that joint measurability of the random system is enough to satisfy our measurability framework. As a consequence, we obtain an MET interpretation of \cite{Buzzi1999}, given in \autoref{thm:MET_applied_to_transferOps}, which is a more detailed version of \autoref{summaryThm:transferOps_cocycle}.

\subsection*{Notations and Conventions}
\begin{enumerate}[label = (\arabic*)]
    \item If $f,g:D\to \R\cup\{\pm\infty\}$ are functions from some set $D$ and $P$ is a collection of parameters, we write $f\lesssim_P g$ if there exists a constant $C_P>0$ depending on the parameters $P$ such that $f(x)\leq C_P g(x)$ for all $x\in D$. We write $f\approx_P g$ if $f\lesssim_P g$ and $g\lesssim_P f$. The subscript $P$ may be dropped if dependence is clear from context.
    \item $\N = \{1,2,3,\ldots\}$
    \item Throughout this paper, all vector spaces are assumed to be real, unless otherwise specified.
    \item For topological vector spaces (TVS) $V_1,V_2$, we denote the space of continuous linear operators $V_1\to V_2$ by $L(V_1,V_2)$. If $V_1=V_2=V$, then we write $L(V)$. The continuous dual of $V$ is denoted by $V^*$.
    \item The calligraphic $\B$ is reserved for Banach spaces.
    \item Double bars ($\norm{\cdot}$) are reserved for Banach space norms, while single bars ($\abs{\cdot}$) are reserved for Hilbert space norms. Moreover, $(\cdot,\cdot)$ will be used for inner products.
    \item If $v$ is a vector and $f$ a linear functional, then we will write $f(v) = v[f] = \innerp{v,f}$ for evaluations.
    \item If $v_1,\ldots,v_n$ are vectors, their linear span will be denoted by $\innerp{v_1,\ldots,v_n}$ or $\vspan(v_1,\ldots,v_n)$.
    \item For a subspace $E\subset \B$, we write
    \[B_E \coloneqq \{e\in E:\norm{e} \leq 1\}\text{ and } S_E\coloneqq \{e\in \E:\norm{e}=1\}.\]
    \item For a linear operator $A\in L(\B)$ and a subspace $E\subset \B$, we write $A|_E$ for the restriction of $A$ to $E$.
    \item The symbol $\delta_{ij}$ will denote the Kronecker delta.
    \item If $(X,\tau)$ is a topological space, and $S\subset X$, then $\tau|_S$ will denote the $\tau$-subspace topology on $S$.
\end{enumerate}

\subsubsection*{Acknowledgments}
I am deeply grateful to Alex Blumenthal for introducing me to this problem, for many fruitful discussions, and for his excellent advice on writing papers.

This material is based on work supported by the National Science Foundation Graduate Research Fellowship Program under Grant No. DGE-2545526 and by Alex Blumenthal's National Science Foundation CAREER Grant under Grant No. DMS-2237360. Any opinions, findings, conclusions, or recommendations expressed in this material are those of the author and do not necessarily reflect the views of the National Science Foundation.

\section{Volume and Geometry in Banach Spaces}\label{sec:geometry}
\subsection{Subspaces and Volume}
Let $(\B,\norm{\cdot})$ be a Banach space. If $F\subset \B$ is a linear subspace of $\B$, we define its annihilator $\ann{F}$ by
\[\ann{F} = \{\ell\in \B^*:\innerp{v,\ell} = 0\ \forall v\in F\}.\]
Similarly, if $L\subset \B^*$, we define its pre-annihilator by
\[\preann{L} = \{v\in \B:\innerp{v,\ell} = 0\ \forall \ell\in L\}.\]
It is easy to see that, if $F\subset \B$ is a closed subspace, then $\preann{\paren{\ann{F}}} = F$. Let $q\geq 0$. We define the codimension of the closed subspace $F$ by $\codim F\coloneqq \dim \ann{F}$. We will continue using the Banach space $(\B,\norm{\cdot})$ for the rest of this subsection.

\begin{remark}
    We note that every finite-dimensional subspace of $\B^*$ is the annihilator of some $q$-codimensional subspace of $\B$. In fact, if $L\subset \B^*$ is finite-dimensional, then $\ann{\paren{\preann{L}}} = L$. This fact continues to hold for the broader class of weak-* closed subspaces $L\subset \B^*$.
\end{remark}

\subsubsection{The Grassmannian}
We denote the Grassmannian of closed subspaces of $\B$ by
\[\G(\B) = \curly{E\subseteq \B: E \text{ is a closed linear subspace of } \B }.\]
If $q \geq 0$, then $\G_q(\B)\subset \G(\B)$ will denote the Grassmannian of $q$-dimensional subspaces, while $\G^q(\B)\subset \G(\B)$ will denote the Grassmannian of closed $q$-codimensional subspaces.

We shall assign to the Grassmannian the Hausdorff metric on the unit sphere, defined for $E,F\in\G(\B)\setminus \curly{\innerp{0}}$ by
\[d_H(E,F) = \max\curly{\sup_{v\in S_E} d(v,S_F),\sup_{v\in S_F} d(v,S_E)},\]
and if $E = \innerp{0}$, then $d_H(\innerp{0},\innerp{0}) = 0$ and $d_H(\innerp{0},F) \coloneqq 1$ if $F\neq \innerp{0}$. It is well known that, with the topology induced by this metric (also called the $\Gap$-topology), $\G(\B)$ is complete and $\G_q(\B)$ and $\G^q(\B)$ are closed subsets of $\G(\B)$ (see, e.g., Chapter IV, \S 2.1 of \cite{Kato1995}). 
\begin{lem}\label{lem:continuityOfDistToSpaces}
    Let $E, E'\subseteq \B$ be closed subspaces. For all $v\in \B$,
    \[\abs{d(v,E) - d(v,E')}\leq 4d_H(E,E')\norm{v}\]
\end{lem}
\begin{proof}
    The proof is straightforward and details are left to the reader.
\end{proof}
Another notion of distance that is nice to work with is the (directed) gap function, defined for $E,F\in \G(\B)\setminus \curly{\innerp{0}}$ by
\[\Gap(E,F) = \sup_{v\in S_E} d(v,F),\]
with the obvious redefinitions for $E$ or $F$ equal to $\innerp{0}$.
\begin{remark}
    The gap function may be referred to as the aperture function in other literature.
\end{remark}
The undirected gap function $(E,F)\mapsto \max\curly{\Gap(E,F),\Gap(F,E)}$ does not satisfy the triangle inequality, but the notion of ``distance" defined by it is equivalent to that of the Hausdorff metric.
\begin{lem}\label{lem:equivOfGap_and_Hausdorff}
Let $E,F\in \G(\B)$.
\begin{enumerate}[label = (\arabic*)]
    \item We have
    \[\Gap(E,F)\leq \sup_{v\in S_E} d(v,S_{F})\leq 2 \Gap(E,F),\]
    \[\max\curly{\Gap(E,F),\Gap(F,E)}\leq d_H(E,F)\leq 2 \max\curly{\Gap(E,F),\Gap(F,E)}.\]
    \item The map $E\mapsto \ann{E}$ is an ``isometry" in the sense that $\Gap(E,F) = \Gap(\ann{F},\ann{E})$.
    \item (Lemma 2.6 in \cite{Blumenthal2016}) Let $q\in \N$. If $E,F\in \G_q(\B)$ or $E,F\in \G^q(\B)$, and $\Gap(E,F) < \frac{1}{q}$, then
    \[\Gap(F,E)\leq \frac{q\Gap(E,F)}{1-q\Gap(E,F)}.\]
\end{enumerate}
\end{lem}
Items (1) and (2) of \autoref{lem:equivOfGap_and_Hausdorff} may be found in \S 2.1 and \S 2.3, respectively, of Chapter IV of \cite{Kato1995}. We will find another result in the same vein as item (2) of the above useful.
\begin{lem}[Chapter IV, \S 2.3 of \cite{Kato1995}]\label{lem:distance_vs_orthogonalProj}
    Let $F\in \G(\B)$ be such that $\innerp{0}\neq F\neq \B$.
    \begin{enumerate}[label = (\arabic*)]
        \item If $v\in \B$, then
        \[d(v,F) = \sup_{\ell\in S_{\ann{F}}} \innerp{v,\ell}.\]
        \item If $\ell\in \B^*$, then
        \[d(\ell,\ann{F}) = \sup_{v\in S_{F}} \innerp{v,\ell}.\]
    \end{enumerate}
\end{lem}

\subsubsection{Angles, Complements, and Projectors}
Let $E,F\subset \B$ be linear subspaces. We say that $E$ and $F$ are \textit{algebraic complements} if $\B = E \oplus F$. This means that, given $v\in \B$, there are unique vectors $e\in E$ and $f\in F$ such that $v = e + f$. In this case, we will write
\[\pi_{E//F} v = e \text{ and } \pi_{F//E} v = f.\]
The map $v\mapsto \pi_{E//F} v$ is the projection operator onto $E$ parallel to $F$, and one should easily see that $\pi_{F//E} = \id - \pi_{E//F}$. If the projection operator $\pi_{E//F}$ is bounded, then we say that $E$ and $F$ are \textit{topological} complements and that the splitting $\B = E \oplus F$ is \textit{topological}. By Lemma 2.2 in \cite{Blumenthal2016}, the splitting is topological if and only if $\theta(E,F)>0$, in which case
\[\sin\theta(E,F) = \norm{\pi_{E//F}}^{-1}.\]
We note that the closed graph theorem immediately implies that an algebraic splitting of closed subspaces is a topological splitting. It is also easy to see that $\pi_{E//F}^* = \pi_{\ann{F}//\ann{E}}$.

\begin{lem}\label{lem:closed_q_dim_complemented}
    Let $q\in \N$.
    \begin{enumerate}[label = (\arabic*)]
        \item (III.B.10 in \cite{Wojtaszczyk1991}) Suppose $E\subset \B$ is a $q$-dimensional subspace of $\B$. Then there is a closed $q$-codimensional subspace $F\subset \B$ such that $\sin\theta(E,F)\geq q^{-\frac{1}{2}}$.
        \item (III.B.11 in \cite{Wojtaszczyk1991}) Suppose $F\subset \B$ is a closed $q$-codimensional subspace of $\B$. Then there is a $q$-dimensional subspace $E\subset \B$ such that $\sin\theta(F,E)\geq (1+\sqrt{q})^{-1}$.
    \end{enumerate}
\end{lem}

\begin{prop}[Lemmas 3.3 and A.1 in \cite{BlumenthalYoung2017}]\label{prop:complementationOpen}
    Suppose $\B = E\oplus F$ is a topological splitting. For any $E'\in \G(\B)$ such that $d_H(E,E') < \sin \theta(E,F)$, we have that $\B = E'\oplus F$ is a topological splitting with 
    \[\sin\theta(E',F) \geq \sin \theta(E,F)-d_H(E,E').\]
\end{prop}
Note, in a similar vein, for $q\in \N$ we have by Corollary 5.5 of \cite{LianLu2010} that the map $(v_1,\ldots,v_q)\mapsto \innerp{v_1,\ldots,v_q}$ is continuous on the open set of $q$-tuples of linearly independent vectors in $\B$, viewed as a subspace of $\B^q$. In particular, this implies the map $(v_1,\ldots,v_q)\mapsto \innerp{v_1,\ldots,v_q}$ from $\B^q\to \G_0(\B)\cup\G_1(B)\cup\cdots\cup\G_q(\B)$ is measurable.

We will use \autoref{prop:complementationOpen} frequently. Another inequality we will use frequently is the following.

\begin{prop}\label{prop:gapProjEquiv}
    Suppose $E,F\subset \B$ are closed subspaces such that $\B = E\oplus F$ is a topological splitting. Assume $F'\subset\B$ is another closed subspace. Then
    \[\norm{\pi_{E//F}|_{F'}}\sin \theta(E,F)\leq \Gap(F',F)\leq \norm{\pi_{E//F}|_{F'}}.\]
\end{prop}
\begin{proof}
    The proof is straightforward and details are left to the reader.
\end{proof}

\subsubsection{Finite-Dimensional Volumes}
We shall use two notions of finite-dimensional volume in this paper. We have already introduced induced volumes in \autoref{subsec:statementOfResults}. The greatest utility of induced volumes on finite-dimensional subspaces is that they are generated by an inner product on $E$ that approximates the norm-structure of $\norm{\cdot}$ on $E$ uniformly well based only on $q=\dim E$. Having access to such a ``local" inner product structure is highly useful.
\begin{prop}[Proposition 2.14 in \cite{Blumenthal2016}]\label{prop:JohnsTheorem}
    Let $q\in \N$ and let $E\subset \B$ be a $q$-dimensional subspace. Then, there exists an inner product $\paren{\cdot,\cdot}_E$ on $E$ for which the following hold.
    \begin{enumerate}[label = (\arabic*)]
        \item The norm $\abs{\cdot}_E$ induced by $(\cdot,\cdot)_E$ satisfies the inequality
        \[\frac{1}{\sqrt{q}}\norm{v}\leq \abs{v}_E \leq \sqrt{q}\norm{v}\]
        for all $v\in E$.
        \item The induced volume measure $m_E$ corresponds to the Lebesgue measure on $E$ arising from the inner product $(\cdot,\cdot)_E$.
    \end{enumerate}
\end{prop}
\begin{remark}
    John's theorem (see Theorem 15 in Chapter 4 of \cite{Bollobas1999}) assigns a canonical inner product satisfying \autoref{prop:JohnsTheorem} to each finite-dimensional subspace.
\end{remark}
\begin{defn}
    Let $q\in \N$ and $v_1,\ldots,v_q\in \B$ be linearly independent. Then we say that $(v_1,\ldots,v_q)$ is \textit{almost orthonormal} if, writing $E = \innerp{v_1,\ldots,v_q}$, we have that $(v_1,\ldots,v_q)$ is an orthonormal basis of $E$ w.r.t. $\paren{\cdot,\cdot}_E$. This means that, for $1\leq i,j\leq q$, we have $\paren{v_i,v_j}_E = \delta_{ij}$.
\end{defn}

We shall also use a notion introduced by Lian and Lu \cite{LianLu2010}, defined for $q\in \N$ and $v_1,\ldots, v_q\in \B$ by
\[\vol_q(v_1,\ldots,v_q) \coloneqq \norm{v_q}d(v_{q-1},\innerp{v_q})\cdots d(v_1,\innerp{v_2,\ldots,v_q}).\]
To give a shorthand, we shall call $\vol_q$ the $q$-dimensional parallelepiped volume. This is because if $E$ is a $q$-dimensional space containing $v_1,\ldots,v_q$ then
\begin{equation}\label{eq:parallelepipedVolumeIdentity}
    q^{-q/2}m_E P[v_1,\ldots,v_q]\leq \vol_q(v_1,\ldots,v_q) \leq q^{q/2}m_E P[v_1,\ldots,v_q],
\end{equation}
where $P[v_1,\ldots,v_q] \coloneqq \curly{t_1v_1+\cdots+t_qv_q : t_1,\ldots,t_q\in [0,1]}$ (see Lemma 1.2 of \cite{Blumenthal2016}). By Lemma 4.2 of \cite{LianLu2010}, the map $\vol_q:\B^q\to [0,\infty)$ is continuous. The main utility of $\vol_q$ is its simplicity, from which we will greatly benefit in the discussions of \autoref{sec:MET} pertaining to measurable volume functions.

\subsection{The Geometry of Linear Operators}
For the rest of this subsection, we shall assume $\paren{\B,\norm{\cdot}_{\B}}$, $\paren{\B',\norm{\cdot}_{\B'}}$, and $(\B'',\norm{\cdot}_{\B''})$ are Banach spaces.

\subsubsection{Determinants, Volume Functions, and Singular Values}
Consider the determinant given in \autoref{defn:determinants}. If $E\subset \B$ is a finite-dimensional subspace and $A:\B\to \B'$ is a bounded linear operator that is injective on $E$, then it follows from uniqueness of Haar measures that
\begin{equation}\label{eq:determinantHaarUniqueness}
    \det(A|E) = \frac{m_{F}(A\mathcal{O})}{m_E(\mathcal{O})},
\end{equation}
for any Borel set $\mathcal{O}\subset E$ of non-zero and finite measure. Moreover the determinant is multiplicative, i.e., if $B:\B'\to\B''$ is another bounded linear operator, then $\det(BA|E) = \det(B|AE)\det(A|E)$. Multiplicativity of $\det$ implies submultiplicativity of $A\mapsto V_q(A)$, i.e., we have $V_q(BA)\leq V_q(B) V_q(A)$. We have the following further properties of determinants.
\begin{lem}[Lemma 3.11 in \cite{BlumenthalYoung2017}]
    Let $q\in \N$ and $A\in L(\B,\B')$. Suppose $V\subset \B$ and $V'\subset \B'$ are $q$-dimensional subspaces such that $AV\subset V'$.
    \begin{enumerate}[label = (\arabic*)]
        \item (Hadamard bound) If $v_1,\ldots,v_q\in V$ form an almost orthonormal basis, then
        \[\det(A|V) \leq \prod_{k = 1}^q \abs{Av_k}_{V'} \leq q^{q/2}\prod_{k = 1}^q \norm{Av_k}_{\B'}.\]
        \item (Singular Value Decomposition) If $v_1,\ldots,v_q\in V$ are almost orthonormal right singular vectors of $A|_V:V\to V'$ w.r.t. $(\cdot,\cdot)_V$ and $(\cdot,\cdot)_{V'}$, then $\det(A|V) = \abs{Av_1}_{V}\cdots\abs{Av_q}_{V}$ and
        \[q^{-q/2}\prod_{k = 1}^q \norm{Av_k}_{\B'}\leq\det(A|V) \leq q^{q/2}\prod_{k = 1}^q \norm{Av_k}_{\B'}.\]
    \end{enumerate}
\end{lem}
\begin{lem}[Lemma 2.15 in \cite{Blumenthal2016}]\label{lem:determinantProdDecomp}
    Let $q\in \N$ and $A\in L(\B,\B')$. Suppose $V\subset \B$ has dimension $q$ and assume $A|_V$ is injective. Let $V = E\oplus F$ be a splitting of $V$ with $\dim E = k<q$, and write $E' = AE$ and $F' = AF$. There is a constant $C_q$ dependent only on $q$ such that
    \[C_q^{-1} \paren{\sin \theta(E',F')}^k\leq \frac{\det(A|V)}{\det(A|E)\det(A|F)}\leq \frac{C_q}{\paren{\sin \theta(E,F)}^k}\]
\end{lem}

We can also define an analogous notion of determinant and volume function based on parallelepiped volumes.
\begin{defn}\label{defn:Lian-LuOperatorVol}
    Let $q\in \N$ and $A\in L(\B,\B')$.
    \begin{enumerate}[label = (\roman*)]
        \item For linearly independent vectors $v_1,\ldots,v_q\in \B$, define
        \[\detLL(A|v_1,\ldots,v_q) = \frac{\vol_q(Av_1,\ldots,Av_q)}{\vol_q(v_1,\ldots,v_q)}.\]
        \item Define
        \[\VLL_q(A) = \sup_{\vol_q(v_1,\ldots,v_q) \neq 0} \detLL(A|v_1,\ldots,v_q).\]
        We take the convention that $\VLL_0(A) = 1$ and $\VLL_q(A) = 0$ when $q > \dim \B$.
    \end{enumerate}
\end{defn}
\begin{remark}
    Notice that, since $\vol_q$ is not permutation-invariant, the order of the basis written in $\detLL$ is important.
\end{remark}
By Lemma 4.4 of \cite{LianLu2010}, $\VLL_q$ can be written in an easier form:
\begin{equation}\label{eq:formsOfLianLuVolNorm}
    \VLL_q(A) = \sup_{\vol_q(v_1,\ldots,v_q) = 1} \vol_q(Av_1,\ldots,Av_q) = \sup_{\norm{v_1},\ldots,\norm{v_q}\leq 1}\vol_q(Av_1,\ldots,Av_q).
\end{equation}
Additionally, like $V_q$, $A\mapsto \VLL_q(A)$ is submultiplicative. It is also immediately clear $\rank(A)<q$ iff $V_q(A) = 0$ iff $\VLL_q(A) = 0$. In fact, we have the following.
\begin{lem}\label{lem:equivalenceOfVolNorms}
    Let $q\in \N$ and $A\in L(\B,\B')$.
    \begin{enumerate}[label = (\arabic*)]
        \item Let $E\subset \B$ be a $q$-dimensional subspace and suppose $v_1,\ldots,v_q\in E$ form a basis of $E$. Then
        \[q^{-q} \det(A|E) \leq \detLL(A|v_1,\ldots,v_q)\leq q^q \det(A|E).\]
        \item $q^{-q} V_q(A)\leq \VLL_q(A)\leq q^q V_q(A)$.
    \end{enumerate}
\end{lem}
\begin{proof}
    Item (1) follows immediately from \eqref{eq:parallelepipedVolumeIdentity} and \eqref{eq:determinantHaarUniqueness} applied to $\mathcal{O} = P[v_1,\ldots,v_q]$. Item (2) is a simple corollary of item (1).
\end{proof}

In a Hilbert space, another notion intertwined with determinants and volume functions is the singular value sequence. This notion has been generalized by Pietsch to Banach space operators via the so called $s$-number sequences (see \cite{Pietsch1987,Ullrich2024}). We shall use one particular realization called the Gelfand number sequence. For $q\in \N$ and $A\in L(\B,\B')$, the $q$th Gelfand number is defined by
\[c_q(A) \coloneqq \inf\curly{\norm{A|_F}:F\in\G^{q-1}(\B)},\]
with the convention that $c_q(A) = 0$ when $q > \dim \B$.

\begin{lem}[Lemma 2.21 in \cite{Blumenthal2016}]\label{lem:s-number_volGrowth_relation}
    Let $q\in \N$ and $A\in L(\B,\B')$. Then there is a constant $C_q\geq 1$, dependent only on $q$, such that
    \[C_q^{-1} V_q(A)\leq \prod_{k=1}^q c_k(A) \leq C_q V_q(A).\]
\end{lem}

\begin{cor}\label{cor:equiv_volGrowth_and_dual}
    For all $q\in \N$, there is a constant $C_q\geq 1$ such that, for all operators $A\in L(\B,\B')$, we have
    \begin{enumerate}[label = (\arabic*)]
        \item $C_q^{-1}c_q(A)\leq c_q(A^*)\leq C_q c_q(A)$; and
        \item $C_q^{-1}V_q(A)\leq V_q(A^*)\leq C_q V_q(A)$.
    \end{enumerate}
\end{cor}
\begin{proof}[Proof sketch]
    It follows immediately from the definition of $s$-number sequences (see \cite{Ullrich2024}) that the dual Gelfand numbers $A\mapsto c_q^*(A) \coloneqq c_q(A^*)$ form an $s$-number sequence. By a result of Ullrich \cite{Ullrich2024}, all $s$-numbers are equivalent up to a constant dependent only on $q$, which means there is some $C_q\geq 1$ for which item (1) is satisfied. Item (2) follows immediately from \autoref{lem:s-number_volGrowth_relation} and item (1).
\end{proof}

\subsubsection{Singular Value Splittings for Banach Space Operators}
A key technical tool for this paper will be a generalization of the singular value decomposition (SVD) to Banach space operators developed by Blumenthal and Morris \cite{BlumenthalMorris2015}. We first begin by defining a useful bookkeeping tool. For a linear operator $A\in L(\B,\B')$ and a subspace $E\subset \B$, write
\[m(A|_E) = \inf\curly{\norm{Av}_{\B'}:v\in S_{E}}.\]

\begin{defn}\label{defn:fast-slowSplit}
    $A:\B\to\B'$ is a bounded linear operator. Take $1\leq q < \infty$ and $C\geq 1$. We say $(E,E',F,F')$ is a $q$\textit{-dimensional singular value splitting of order} $C$ (for $A$) if
    \begin{enumerate}[label = (\alph*)]
    \item $F$ and $F'$ are closed and $\dim E = \dim E'  = \codim F = \codim F' = q$;
    \item $\B = E \oplus F$ and $\B' = E'\oplus F'$ are topological splittings;
    \item $E' = AE$ and $F' \supseteq A F$; and
    \item the following volume growth conditions are satisfied
    \begin{enumerate}[label = (\roman*)]
        \item $m(A|_E) \geq C^{-1} c_{q}(A)$
        \item $\norm{A|_F}\leq C c_{q+1}(A)$
        \item $\norm{\pi_{E//F}},\norm{\pi_{E'//F'}}\leq C$
    \end{enumerate}
\end{enumerate}
The set of such quadruples of subspaces is denoted by $\SVS_q(C;A)$.
\end{defn}

\begin{defn}
    Let $q\in \N$ and $A\in L(\B,\B')$ be such that $V_q(A)\neq 0$. Let $E\in \G_q(\B)$ be such that $\det(A|E) > 0$ and write $E' = AE$. Then we define the quantities
    \[\sigma_1(A|E)\geq\cdots\geq \sigma_q(A|E)>0\]
    to be the $q$ singular values of $A|_E\in L(E,E')$ w.r.t. the induced inner products $\paren{\cdot,\cdot}_E$ and $\paren{\cdot,\cdot}_{E'}$.
\end{defn}
\begin{remark}
    Notice that $\det(A|E) = \sigma_1(A|E)\cdots\sigma_q(A|E)$.
\end{remark}
\begin{prop}[Singular Value Decomposition for Banach Space Operators]\label{prop:SVD}
    Let $\alpha\in (0,1)$ and $q\geq 1$. Then there is a constant $D = D(\alpha,q)\geq 1$ with the following property. Let $\B$ and $\B'$ be Banach spaces and let $A:\B\to\B'$ be a bounded linear operator whose rank is at least $q$. Suppose $E\subset\B$ is a $q$-dimensional subspace such that $\det(A|E)\geq \alpha V_q(A)$.
    \begin{enumerate}[label = (\arabic*)]
        \item (Proposition 18 in \cite{BlumenthalMorris2015}) There exist topological complements $F\in\G^q(\B)$ and $F'\in \G^q(\B')$ of $E$ and $E' = AE$, respectively, such that $(E,E',F,F')\in \SVS_q(D;A)$.
        \item There is an almost orthonormal basis $(v_1,\ldots,v_q)$ of $E$ such that $\paren{u_k = Av_k/\abs{Av_k}_{E'}}_{k=1}^q$ is an almost orthonormal basis of $E'$, and for all $1\leq k\leq q$
        \[\frac{1}{D} c_k(A)\leq \frac{\norm{Av_k}}{\norm{v_k}}\leq Dc_k(A),\]
        and
        \[\sin\theta(\innerp{v_k},\innerp{v_i:i\neq k}),\ \sin\theta(\innerp{u_k},\innerp{u_i:i\neq k})\geq \frac{1}{q}.\]
        We shall call $(v_1,\ldots,v_q)$ right singular vectors of $A|_E$, and $(u_1,\ldots,u_q)$ the left singular vectors of $A|_E$.
        \item For $1\leq k\leq q-1$, write
        \[E_k = \innerp{v_1,\ldots, v_k},\ F_k = \innerp{v_{k+1},\ldots,v_q}+F\]
        and
        \[E_k' = \innerp{u_1,\ldots, u_k},\ F_k' = \innerp{u_{k+1},\ldots,u_q}+F'.\]
        Then $(E_k,E_k',F_k,F_k')\in \SVS_k\paren{D; A}$.
    \end{enumerate}
\end{prop}
\begin{remark}
    We note that Blumenthal and Morris do develop a notion of left and right ``singular vectors" in the proof of their proposition 18 \cite{BlumenthalMorris2015}. For our purposes, however, there are some conveniences that come with ensuring the bases are almost orthonormal, which is not the case in the original proof of the proposition.
\end{remark}
\begin{remark}
    Item (2) implies that
    \[\sin\theta(\innerp{v_k},\innerp{v_i:i\neq k}+F),\ \sin\theta(\innerp{u_k},\innerp{u_i:i\neq k}+F')\geq \frac{1}{qD}.\]
\end{remark}
\begin{proof}
    Item (1) is just the statement of Proposition 18 in \cite{BlumenthalMorris2015}. For item (2), note that by \autoref{lem:s-number_volGrowth_relation},
    \[\alpha\prod_{k=1}^q c_k(A)\lesssim_q \prod_{k=1}^q \sigma_k(A|E)\lesssim_q \prod_{k=1}^q c_k(A).\]
    The variational characterization of singular values gives
    \[\sigma_k(A|E) = \inf\curly{\sup_{v\in V,\abs{v}_E\leq 1} \abs{Av}_{E'}:V\subseteq E,\ \dim V = q - k + 1 },\]
    which means $c_k(A|_E) \approx_q \sigma_k(A|E)$ by \autoref{prop:JohnsTheorem}. Since $c_k(A|_E)\leq c_k(A)$, \autoref{lem:s-number_volGrowth_relation} gives us $\alpha c_k(A)\lesssim_q \sigma_k(A|E)\lesssim_q c_k(A)$. Denoting by $v_1,\ldots,v_q \in E$ and $u_1,\cdots,u_q\in E'$ the right and left singular vectors of $A|_E\in L(E,E')$, respectively, w.r.t. the induced inner product structures on $E$ and $E'$, we get $\norm{Av_k}/\norm{v_k}\approx_{q} \sigma_k(A|E) \approx_{q,\alpha} c_k(A)$. Moreover, it is clear that for all $v\in \B$,
    \[\pi_{\innerp{v_k}//\innerp{v_i:i\neq k}+F}v = \paren{\pi_{E//F}v,v_k}_E v_k.\]
    If $v\in E$, then $\abs{(v,v_k)_E}\leq \abs{v}_E \leq \sqrt{q}\norm{v}$. Since $\norm{v_k}\leq \sqrt{q}\abs{v_k}_E = \sqrt{q}$, we get the operator norm bound in item (2). The same argument holds for $(u_k)_{k}$, thus item (2) is proven.

    For item (3), take $1\leq k\leq q$ and write $E_k, F_k, E_k',\text{ and } F_k'$ as in item (3). Items (a)-(c) and (d)(i) in \autoref{defn:fast-slowSplit} are trivially verified. We can also see that
    \[\norm{\pi_{E_k//F_k}}\leq \norm{\pi_{E_k//\innerp{v_i:k<i}}|_E}\norm{\pi_{E//F}}\lesssim_{q,\alpha} 1.\]
    This takes care of item (d)(iii) since the same argument holds for $E'_k$ and $F'_k$. Now take $v\in F_k$ and write $v = f_k + f$ where $f_k \in \innerp{v_i:k < i}$ and $f \in F$. We know $\norm{f_k},\norm{f}\lesssim_{q,\alpha} \norm{v}$, and thus
    \[\norm{A v}\leq \norm{Af_k} + \norm{A f}\lesssim (\sigma_{k+1}(A|E)+c_{q+1}(A))\norm{v}\lesssim c_{k+1}(A)\norm{v}.\]
    This provides item (d)(ii) of \autoref{defn:fast-slowSplit}, and thus the rest of the proposition.
\end{proof}

We are interested in relating singular value decompositions of adjoints to the original operator. Every Banach space canonically embeds isometrically into its bidual via the pointwise evaluation map. More precisely, for a Banach space $\B$, the pointwise evaluation map $\Ev:\B\to \B^{**}$ is defined for $v\in \B$ by $f\mapsto \Ev(v)f = \innerp{v,f}$ for all $f\in \B^*$. It is clear that $\Ev$ is a linear isometry. We shall use the notation $\Ev$ to denote the pointwise evaluation map for any Banach space as the ambiguity should be cleared from context. Under this isometry, we identify $\B$ and $\B'$ as closed subspaces of $\B^{**}$ and $\B'^{**}$, respectively.
\begin{lem}\label{lem:dual_SVD}
    Let $\alpha \in (0,1)$ and $q\in \N$. Then there is a constant $C = C(\alpha,q)\geq 1$ with the following property. Let $\B$ and $\B'$ be Banach spaces and let $A:\B\to\B'$ be a bounded linear operator whose rank is at least $q$. Suppose $E\subset \B$ is a $q$-dimensional subspace such that
    \[\det(A|E)\geq \alpha V_q(A)\]
    and write $E' = AE$. Let $F$ and $F'$ be as guaranteed in \autoref{prop:SVD} so that $(E,E',F,F')\in \SVS_q(D; A)$ for some $D\geq 1$ and write $L = \ann{F'}$, $K = \ann{E'}$, $\Lhat = \ann{F}$, and $\Khat = \ann{E}$. Then
    \[\det(A^*|L)\geq C^{-1} V_q(A^*)\text{ and }(L,\Lhat,K,\Khat)\in \SVS_q(C; A^*).\]
\end{lem}
\begin{proof}
    Let $(v_1,\ldots,v_q)$ and $(u_1,\ldots,u_q)$ be the right and left singular vectors of $A|_E$, respectively, and denote the singular values by $\sigma_k = \sigma_k(A|E)$. Since $\pi_{L//K} = (\pi_{E'//F'})^*$ and $\pi_{\Lhat//\Khat} = (\pi_{E//F})^*$, it follows that $\B'^{*} = L \oplus K$ and $\B^{*} = \Lhat\oplus \Khat$ are topological splittings for which $\sin\theta(L,K),\ \sin\theta(\Lhat,\Khat)\gtrsim_{q,\alpha} 1$.

    Given $1\leq k\leq q$, consider the linear functionals $f_k\in \B'^{*}$ and $g_k\in \B^{*}$ defined by
    \[f_k(u) \coloneqq \paren{\pi_{E'//F'} u, u_k}_{E'} \text{ and } g_k(v) \coloneqq \paren{\pi_{E//F} v, v_k}_{E}\]
    for $u\in \B'$ and $v\in \B$. Clearly, $f_k\in L=\ann{F'}$ and $g_k\in \Lhat=\ann{F}$. We would like to treat $(f_k)$ and $(g_k)$ as pseudo-singular vectors for $A^{*}|_L:L\to\Lhat$, and so there are a few things to verify. One can easily show that both $(f_1,\ldots,f_q)$ and $(g_1,\ldots,g_q)$ are linearly independent, and as such, form bases for $L$ and $\Lhat$, respectively. Furthermore, we see by definition that for all $v\in \B$, 
    \[A^{*}f_k(v) = \frac{1}{\sigma_k}\paren{A\pi_{E//F}v ,Av_k} = \frac{1}{\sigma_k}\paren{\pi_{E//F} v,T^\dagger Tv_k} = \sigma_k g_k(v),\]
    where $T = A|_E\in L(E,E')$ and $T^\dagger$ denotes the inner product adjoint. Since a function is determined by its action on all elements of its domain, it follows that $A^{*}f_k = \sigma_k g_k$.

    Notice that
    \[\norm{f_k} = \sup_{u\in S_{\B'}} \paren{\pi_{E'//F'} u, u_k}_{E'}\lesssim_{q,\alpha} 1.\]
    Furthermore, if $g\in \innerp{g_i:i\neq k}$, then
    \[\norm{g-g_k}\gtrsim_{q} \abs{(g-g_k)(v_k)} = g_k(v_k) \geq q^{-1/2},\]
    so, taking an infimum over $g\in \innerp{g_i:i\neq k}$ gives us $d(g_k,\innerp{g_i:i\neq k})\gtrsim_q 1$. This implies
    \begin{equation}\label{eq:svdCor_VolDistortion}
        \detLL(A^*|f_{1},\ldots,f_q) = \sigma_1\cdots\sigma_q\frac{\vol_q(g_1,\ldots,g_q)}{\vol_q(f_1,\ldots,f_q)} \gtrsim_{q,\alpha} \sigma_1\cdots\sigma_q.
    \end{equation}

    It is obvious that $K$ is a closed $q$-codimensional subspace, as it is the annihilator of a $q$-dimensional subspace of $\B'\subseteq \B'^{**}$. The same reasoning holds for $\Khat$, and clearly $A^*K\subseteq \Khat$. For the maximum growth estimate, let $h\in K$ be of unit norm. Following the proof of \autoref{prop:gapProjEquiv}, $\norm{A^* h}\leq d(A^*h,\Lhat) /\sin\theta(\Khat,\Lhat)\lesssim_{q,\alpha} d(A^*h,\Lhat)$. As such, $\norm{A^*h}\detLL(A^*|f_1,\ldots,f_q)\lesssim \detLL(A^*|h,f_1,\ldots,f_q)\lesssim c_1(A)\cdots c_q(A) c_{q+1}(A)$ by \autoref{cor:equiv_volGrowth_and_dual} and \autoref{lem:equivalenceOfVolNorms}. Hence,
    \[\norm{A^*h}\lesssim \frac{c_1(A)\cdots c_q(A)}{\sigma_1\cdots\sigma_q} c_{q+1}(A)\lesssim c_{q+1}(A).\]
    For the minimum growth estimate, we have by \eqref{eq:svdCor_VolDistortion} that
    \[m(A^*|_L)\gtrsim \frac{\det(A^*|L)}{\sigma_1\cdots\sigma_{q-1}}\gtrsim \sigma_q\gtrsim c_q(A).\]
\end{proof}

The following corollary will be useful for establishing measurability properties of the Oseledets filtration.
\begin{cor}\label{cor:reverse_dual_SVD}
    Let $\alpha \in (0,1)$ and $q\in \N$. Then there is a constant $C = C(\alpha,q)\geq 1$ with the following property. Let $\B$ and $\B'$ be Banach spaces and let $A:\B\to\B'$ be a bounded linear operator whose rank is at least $q$. Suppose $L\subset \B'^*$ is a $q$-dimensional subspace of linear functionals such that 
    \[\det(A^*|L)\geq \alpha V_q(A^*),\]
    naturally viewing $A^*$ as a bounded linear operator between the Banach spaces $\B'^*\to \B^*$. Then we have that $\preann{(A^*L)}\subset \B$ is a closed $q$-codimensional subspace of $\B$ and
    \[\norm{A|_{\preann{(A^*L)}}}\leq C(\alpha,q) c_{q+1}(A).\]
\end{cor}
\begin{proof}
    Apply \autoref{prop:SVD} to $A^*$ on $L$ so that we get spaces $(L,\Lhat,K,\Khat)\in \SVS_q(D; A^*)$ for some $D$. Then apply \autoref{lem:dual_SVD} to get spaces $(E,E',F,F')\in \SVS_q(C;A^{**})$ for some $C\geq 1$. Notice that $F = \ann{\Lhat} \subset \B^{**}$, so under the identification via the pointwise evaluation map, $F\cap\B = \preann{(A^*L)}\subset \B$. Since $A^{**}|_\B = A$, we can see by \autoref{cor:equiv_volGrowth_and_dual} that
    \[\norm{A|_{\preann{(A^*L)}}} \leq \norm{A^{**}|_F}\leq Cc_{q+1}(A^{**})\lesssim c_{q+1}(A).\]
    Since $A^*|_L$ is injective, $\codim \preann{(A^*L)} = \dim L = q$.
\end{proof}
\begin{remark}\label{remark:formOfSVDconsts}
    In \autoref{prop:SVD}, \autoref{lem:dual_SVD}, and \autoref{cor:reverse_dual_SVD}, the multiplicative constants that depend on $\alpha$ and $q$ are all of the form
    \[G_1(q)\alpha^{-G_2(q)},\]
    where $G_1,G_2:\N\to\N$ are some maps. We may assume $G_1 = G_2 \eqqcolon G$ and that they are the same (as non constant functions of $q$) across the three results by taking a finite number of maximums.
\end{remark}

\section{Growth Theorems for Products of Operators}\label{sec:sequentialMET_oneSided}
In this section, we will consider a one-sided sequence of bounded linear Banach space operators $(A_n:\B_n\to\B_{n+1})_{n\geq 0}$ subject to one of the two assumptions listed in the introduction. To simplify our proofs and avoid some tedious case-work, we shall assume $l_\alpha > -\infty$ in the case that there are a finite number of Lyapunov exponents, i.e. $r<\infty$. It will be notationally convenient to write $\lambda_{r+1} \coloneqq l_\alpha$ in this case. The case that $l_\alpha = -\infty$ when $r < \infty$ only requires a few modifications, and if $r = \infty$ then the concerns are irrelevant. Necessary modifications for the case $l_\alpha = -\infty$ are described at the end of this section.

We shall define the following quantities for $q\in \N$, $m\geq 0$, $\delta > 0$, and $N\geq 0$.
\begin{itemize}
    \item $\Bplus(m,\delta) = \sup_{n\geq 0} \log \norm{A_{m+n}} - n \delta$
    \item $\Hplus_q(m,\delta) = \sup_{n\geq 0} \abs{\log V_q(A^n_{m}) - nl_q} - n\delta$
    \item $\Htplus_q(m,\delta,N) = \sup_{n,k\geq N} \abs{\log V_q(A^k_{m+n}) -  kl_q} - (n+k)\delta$
\end{itemize}
For the sake of convenience, we define these functions to be 0 at $q = 0$. The relevant assumptions determine which of these quantities are finite. The inspiration for these functions derives from \cite{GouezelStoyanov2017}, and they are the main drivers behind our effective treatment of the MET. One can think of these quantities as a non-uniformity scale for convergence of growth rates. For instance, the slower $\Hplus_q(m,\delta)$ grows as $\delta\to 0^+$, the more uniform the rate of convergence of $(V_{q}(A_m^n))^{1/n}$ to $e^{l_q}$. Thus, an effective treatment of the MET reflects the non-uniformity of volume growth quantitatively.

We would like to state the following trivial claim, which will allow us to compare the $H$-functions to non-uniformity in the convergence rates of the Gelfand numbers. 

\begin{claim}\label{claim:equiv_of_vol_gelfand_pesin_funcs}
    Let $q\in \N$. Then, there is a constant $D_q\geq 1$, dependent only on $q$, such that the following holds. For all Banach space bounded linear operators $A$ of rank at least $q$, $\delta > 0$, and $i,j\geq 0$,
    \[\abs{\log c_{q}(A) - i \chi_q} - 2j\delta \leq D_q + \paren{\abs{\log V_{q}(A) - i l_q} - j\delta} +  \paren{\abs{\log V_{q-1}(A) - i l_{q-1}} - j\delta},\]
    writing $V_0(A) \coloneqq 1$ and $l_0 \coloneqq 0$.
\end{claim}
\begin{proof}
    This is a straightforward consequence of \autoref{lem:s-number_volGrowth_relation}.
\end{proof}

We introduce the following bookkeeping tool before we move on to proving the four spectral theorems.
\begin{defn}\label{defn:OseledetsSpectrumData}
    For $1\leq s < r+1$, let $\OSE_s = (\lambda_t;m_t)_{1\leq t\leq s}$ denote the \textit{Oseledets spectrum up to} $\lambda_s$. If $r$ is finite, let $\OSE_{r+1} = \OSE_r\cup\{l_\alpha\}$.
\end{defn}

\begin{remark}
    All constants in the estimates in this section and \autoref{sec:sequentialMET_twoSided} can be made explicit by tracing the proofs. We record their dependence on the relevant parameters but make no attempt to optimize their precise form.
\end{remark}

\subsection{Proof of \autoref{thm:sequentialMET_oneSided}}
In this section, we prove \autoref{thm:sequentialMET_oneSided}. Let $(A_{n}:\B_{n}\to\B_{n+1})_{n\geq 0}$ be a sequence of operators satisfying \autoref{assump:oneSidedSequences}. We shall subdivide the proof into three key lemmas. Note that \autoref{assump:oneSidedSequences} implies that $\Hplus_q(0,\delta)$ is finite for all $q,\delta$.

\begin{lem}\label{lem:cauchySeqSpaces}
    Let $2\leq s < r + 2$ and $C_s\geq 1$. Suppose we have a sequence of singular value splittings $(E_{n,s-1},E_{n,s-1}',F_{n,s},F_{n,s}')\in \SVS_{M_{s-1}}(C_s;A^n)$ for all $n$. Then, given $0<\delta<\paren{\lambda_{s-1}-\lambda_{s}}/5$, $H\geq \max_{-1\leq j\leq 1} \Hplus_{j+M_{s-1}}(0,\delta)$, and $B\geq B^+(0,\delta)$, there is an integer $N\in\N$ and constant $C\geq 1$, both only dependent on $C_s,\delta,H,B,\lambda_s,\lambda_{s-1},\text{ and } M_{s-1}$, such that for all $n\geq N$ and $k\geq 0$ we have
    \begin{equation}\label{eq:cauchyInequality}
        d_H(F_{n,s},F_{n+k,s})\leq C\exp(-n(\lambda_{s-1}-\lambda_{s}-5\delta))
    \end{equation}
    In particular, $(F_{n,s})_n$ is a Cauchy sequence of subspaces in $\G^{M_{s-1}}(\B_0)$. Moreover, the slow spaces in $\B_0$ of any such sequence of singular value splittings, over all $C_s\geq 1$, have the same limit.
\end{lem}

Now, consider an arbitrary $2\leq S < r+2$ and a sequence of spaces $\ES_{n,S-1}\in \G_{M_{S-1}}(\B_0)$ such that
\[\frac{\det\paren{A^n|\ES_{n,S-1}}}{V_{M_{S-1}}(A^n)}\geq \alpha_{S} > 0\]
for all $n\geq 1$. Let $\delta>0$. By \autoref{prop:SVD} and \autoref{claim:equiv_of_vol_gelfand_pesin_funcs}, there is a constant $\CS = C(M_{S-1},\alpha_S)\geq 1$, a sequence of spaces $\paren{\ES_{n,S-1},\ESp_{n,S-1},\FS_{n,S},\FSp_{n,S}}\in \SVS_{M_{S-1}}\paren{\CS;A^n}$, and almost orthonormal splittings of \textit{singular spaces}
\[\ES_{n,S-1} = \US_{n,1}\oplus \cdots \oplus \US_{n,S-1}\]
and
\[\ESp_{n,S-1} = \USp_{n,1}\oplus \cdots \oplus \USp_{n,S-1}\]
with the following properties.
\begin{itemize}
        \item $\USp_{n,t} = A^n \US_{n,t}$ and $\dim \US_{n,t} = \dim \USp_{n,t} = m_t$.
        \item  If $\delta>0$ and $H\geq \max_{0\leq q\leq 1+M_{S-1}} \Hplus_{q}(0,\delta)$ then, for all $n\geq 0$, 
        \begin{equation}\label{eq:SVDtime-nGrowth}
            (e^{2H}\CS)^{-1}e^{n(\lambda_{t}-2\delta)}\norm{u}\leq \norm{A^n u}\leq e^{2H}\CS e^{n(\lambda_{t}+2\delta)}\norm{u}
        \end{equation}
        for all $u\in \US_{n,t}\setminus\innerp{0}$.
        \item Writing (for $1\leq t\leq S$)
        \[\ES_{n,t-1} \coloneqq \US_{n,1}\oplus\cdots\oplus \US_{n,t-1}\]
        and
        \[\FS_{n,t} \coloneqq \US_{n,t}\oplus \cdots \oplus \US_{n,S-1}\oplus \FS_{n,S},\]
        and the analogous spaces $\ESp_{n,t-1}$ and $\FSp_{n,t}$ in $\B_n$, we have
        \begin{equation}\label{eq:SVDprojNorm}
            \norm{\pi_{\ES_{n,t-1}//\FS_{n,t}}},\norm{\pi_{\ESp_{n,t-1}//\FSp_{n,t}}}\leq \CS.
        \end{equation}
\end{itemize}

When the context is clear, we will drop the $(S)$ superscript. For convenience, we write $\ES_{n,0} = \innerp{0}$, $\ESp_{n,0} = \innerp{0}$, $\FS_{n,1} = \B_0$, and $\FSp_{n,1} = \B_n$.
\begin{lem}\label{lem:forwardGapSepEstimate}
        Suppose $2\leq s'\leq s \leq S$. Given $0<\delta < \frac{1}{5}\min_{2\leq t\leq S} \lambda_{t-1}-\lambda_{t}$, $H \geq \max_{1\leq q\leq 1+M_{S-1}} \Hplus_q(0,\delta)$, and $B > B^{+}(0,\delta)$, there is an integer $N\in\N$ and constant $C\geq 1$, dependent on $\alpha_S,\delta,H,B,\text{ and}$ $\OSE_S$, such that for all $n\geq N$ and $k\geq 0$ we have
    \begin{equation}\label{eq:forwardGapSepEstimate}
        \Gap\paren{\FS_{n,s},\FS_{n+k,s'}}\leq C\exp(-n(\lambda_{s'-1}-\lambda_{s}-5(1+s-s')\delta)).
    \end{equation}
\end{lem}

\begin{lem}\label{lem:backwardGapSepEstimate}
    Suppose $2\leq s'\leq s \leq S$. Given $0<\delta < \frac{1}{5(S-1)^2}\min_{2\leq t\leq S} \lambda_{t-1}-\lambda_{t}$, $H \geq \max_{1\leq q\leq 1+M_{S-1}} \Hplus_q(0,\delta)$, and $B > B^{+}(0,\delta)$, there is an integer $N\in\N$ and constant $C\geq 1$, dependent on $\alpha_S,\delta,H,B,\text{ and}$ $\OSE_S$, such that for all $n\geq N$ and $k\geq 0$ we have
    \begin{equation}\label{eq:backwardGapSepEstimate}
        \Gap(\FS_{n+k,s},\FS_{n,s'})\leq C\exp(-n(\lambda_{s'-1}-\lambda_{s}-5(S-1)(1+s-s')\delta))
    \end{equation}
\end{lem}

\autoref{lem:cauchySeqSpaces} and \autoref{lem:backwardGapSepEstimate} will be directly used in the proof of \autoref{thm:sequentialMET_oneSided}, while \autoref{lem:forwardGapSepEstimate} is used to prove \autoref{lem:backwardGapSepEstimate}. 

We postpone the proofs of these three lemmas to the sequel and instead provide the proof of \autoref{thm:sequentialMET_oneSided}. For the sake of brevity, we shall suppress subscript notation from norms and operator norms, as the relevant spaces should be clear regardless.

\begin{remark}
    The proof of item (4) is more involved than the other three items, so we shall treat it separately.
\end{remark}

\begin{proof}[Proof of items (1)-(3)]
\noindent \textbf{Item (1):} For each $2\leq s < r+2$, we will define
\[F_s = \lim_{n\to\infty} F^{(s)}_{n,s}\in \G^{M_{s-1}}(\B_0)\]
which exists by \autoref{lem:cauchySeqSpaces}. In the case that $r$ is finite, this means that $F_\alpha\coloneqq F_{r+1}$. Otherwise, $F_{\alpha}\coloneqq F_1\cap F_2\cap\cdots$. For $1\leq s < r+1$, the fact that $F_{s+1}\subset F_s$ follows from the fact that $\FsOne_{n,s+1}\to F_{s+1}$, $\FsOne_{n,s}\to F_s$ and $\FsOne_{n,s+1}\subset \FsOne_{n,s}$ for all $n$.

Now, take $1\leq s < r+2$. We shall utilize the spaces constructed under \autoref{lem:cauchySeqSpaces}, and we arbitrarily fix $\alpha_S = \frac{1}{2}$ for all $2\leq S < r+2$. Throughout the following proofs, the constant in $\lesssim$ shall depend on $\delta,H,B,\text{ and}$ $\OSE_S$, where $\delta, H, B$ and $S$ are to be specified. In the partial proof of item (3), it shall also depend on $\eta>0$.

\noindent \textbf{Upper bound for item (2):} Since $F_1 = \B_0$, item (2) is clear for $s = 1$, so we may assume $s\geq 2$. Let $\delta>0, H,\text{ and } B$ be as in \autoref{lem:backwardGapSepEstimate} for $S = s$. Consider an arbitrary vector $v\in F_{s}$. We may then write
\[v = v_n + \sum_{t=1}^{s-1} u_{n,t}\]
where $v_n \in \Fs_{n,s}$ and $u_{n,t} \in \Us_{n,t}$. \autoref{lem:cauchySeqSpaces} establishes that $\Fs_{n,s'}\to F_{s'}$ for each $2\leq s'\leq s$, and so \autoref{lem:backwardGapSepEstimate} gives us that
\[\Gap\paren{F_{s},\Fs_{n,s'}} = \lim_{k\to\infty} \Gap\paren{\Fs_{n+k,s},\Fs_{n,s'}}\lesssim \exp\paren{-n(\lambda_{s'-1}-\lambda_s-5(s-1)(1+s-s')\delta)}.\]
In particular, this implies for each $1\leq t\leq s-1$ that
\[\norm{u_{n,t}}\lesssim \exp\paren{-n(\lambda_{t}-\lambda_s-5(s-1)(s-t)\delta)}\norm{v},\]
from which it follows
\[\norm{A^n u_{n,t}}\lesssim \exp(n(\lambda_t+2\delta))\norm{u_{n,t}} \lesssim \exp(n(\lambda_s+7(s-1)(s-t)\delta))\norm{v}.\]
In addition,
\[\norm{A^n v_{n}}\lesssim \exp(n(\lambda_s+2\delta))\norm{v_n}\lesssim \exp(n(\lambda_s+2\delta))\norm{v},\]
and so
\begin{equation}\label{eq:effectiveFsMaxGrowthRate}
    \norm{A^nv}\lesssim \exp(n(\lambda_s+7(s-1)^2\delta))\norm{v}.
\end{equation}
Therefore,
\begin{equation}\label{eq:FsMaxGrowthRate}
    \limsup_{n\to\infty} \frac{1}{n}\log\sup_{v\in F_s\setminus\{0\}} \frac{\norm{A^nv}}{\norm{v}}\leq \lambda_s + 7(s-1)^2\delta
\end{equation}
Since $\delta>0$ was arbitrarily small, this shows the $\limsup$ is bounded above by $\lambda_s$.

\noindent \textbf{Lower bound for item (3):} Now, assume $1 \leq s < r+1$. Let $\delta>0, H,\text{ and } B$ be as in \autoref{lem:cauchySeqSpaces} for $S = s+1$. Let $0 < \eta \leq 1$ and take an arbitrary $v\in \Ccal_s(\eta)\setminus\innerp{0}$. It follows by \autoref{lem:continuityOfDistToSpaces} that, for $n$ large enough,
\begin{equation}\label{eq:item(3)AngleEstimate}
    \norm{v_n}\geq d\paren{v_n,\FsOne_{n,s+1}} = d\paren{v,\FsOne_{n,s+1}}\geq \paren{\eta-4d_H\paren{F_{s+1},\FsOne_{n,s+1}}}\norm{v} \geq \frac{\eta}{2}\norm{v}
\end{equation}
where $v_n = \pi_{\EsOne_{n,s}//\FsOne_{n,s+1}} v$. Writing $w_n = \pi_{\FsOne_{n,s+1}//\EsOne_{n,s}} v$, we get
\begin{align*}
    \norm{A^n v}&\geq \norm{A^n v_n} - \norm{A^n w_n} \\
    &\geq (e^{2H}C_{s+1})^{-1}\frac{\eta}{2}\exp(n(\lambda_s - 2\delta))\norm{v}-e^{2H}C_{s+1}^2\exp(n(\lambda_{s+1}+2\delta))\norm{v} \\
    &\geq \paren{(e^{2H}C_{s+1})^{-1}\frac{\eta}{2}-e^{2H}C_{s+1}^2\exp(-n(\lambda_{s}-\lambda_{s+1}-4\delta))}\exp(n(\lambda_s-2\delta))\norm{v} \\
    &\geq \frac{\eta}{4e^{2H}C_{s+1}}\exp(n(\lambda_s-2\delta))\norm{v}.
\end{align*}
for $n$ large enough. As such,
\begin{equation}\label{eq:unifGrowthOverConeFs}
    \liminf_{n\to\infty} \frac{1}{n} \log \inf\curly{\frac{\norm{A^n v}}{\norm{v}}:d\paren{v,F_{s+1}}\geq \eta \norm{v}>0}\geq \lambda_s - 2\delta.
\end{equation}
Since $\delta>0$ was arbitrarily small, the $\liminf$ is bounded below by $\lambda_s$. 

\begin{remark}
    The proof of the lower bound for item (3) holds for all $v\in \B_0\setminus F_{s+1}$ with $d(v,F_{s+1}) \geq \eta \norm{v}$.
\end{remark}

\noindent \textbf{Finishing Items (2) and (3):} Collecting both \eqref{eq:FsMaxGrowthRate} and \eqref{eq:unifGrowthOverConeFs} gives us the rest of items (2) and (3) simultaneously except for the lower bound on $\liminf_{n\to\infty} \frac{1}{n}\log\norm{A^n|_{F_\alpha}}$ when $r< \infty$ or the upper bound on the $\limsup$ when $r = \infty$. In the case $r<\infty$, it follows easily from the definition of the Gelfand numbers and \autoref{assump:oneSidedSequences} that
\[l_\alpha = \lim_{n\to\infty}\frac{1}{n}\log c_{1+M_r}(A^n)\leq \liminf_{n\to\infty} \frac{1}{n}\log\norm{A^n|_{F_\alpha}}.\]
As item (2) provides the upper bound of $l_\alpha$ on the $\limsup$, we are done.

When $r = \infty$, we then get that
\[\limsup_{n\to\infty} \frac{1}{n}\log \norm{A^n|_{F_\alpha}}\leq \inf_{s\geq 1} \lim_{n\to\infty} \frac{1}{n}\log \norm{A^n|_{F_s}} = \inf_{s\geq 1}\lambda_s = l_\alpha.\]
\end{proof}

For the proof of item (4), we shall require a preliminary result. The idea for the following claim originates from Lemma 12 of \cite{Lee2024}.

\begin{claim}\label{claim:ImagesOfFastSpacesSqueeze}
    Let $1\leq s < S < r + 2$. Given $\eta\in (0,1]$, $0<\delta < \frac{1}{9S^2+4}\min_{2\leq t\leq S} \lambda_{t-1}-\lambda_{t}$, $H \geq \max_{1\leq q\leq 1+M_{S-1}} \Hplus_q(0,\delta)$, and $B > B^{+}(0,\delta)$, there is an integer $N\in\N$ and constant $C\geq 1$, dependent on $\alpha_S,\eta,\delta,H,B,\text{ and}$ $\OSE_S$, such that for all $n\geq N$ and $E\in \G\Ccal_{s}(\eta)$, one has
    \[d_{H}\paren{A^n E,A^n \ES_{n,s}}\leq C\exp(-n(\lambda_{s}-\lambda_{s+1}-4\delta)).\]
\end{claim}
\begin{remark}
    We do not need $\delta$ to be so small for this particular proof, but it will be helpful for item (4).
\end{remark}
\begin{proof}
    Let $n$ be large enough. Take $E\in \G\Ccal_s(\eta)$. The proof of item (3) shows that
    \[\frac{\eta}{4e^{2H}C_{s+1}}\exp(n(\lambda_s-2\delta))\leq m(A^n|_{E})\]
    for all $n$ large enough. As such, if $v\in E$ is such that $\norm{A^n v} = 1$, then
    \[\norm{v}\leq \frac{1}{m(A^n|_{E})}\lesssim \exp(n(-\lambda_s+2\delta)).\]
    Now, since
    \[\pi_{F_{n,s+1}'//E_{n,s}'} A^n v = A^n \pi_{F_{n,s+1}//E_{n,s}} v,\]
    it follows that
    \begin{align*}
        \norm{\pi_{F_{n,s+1}'//E_{n,s}'} A^n v} &\lesssim \norm{A^n|_{F_{n,s+1}}}\norm{v}\\
        &\lesssim \exp(n(\lambda_{s+1}+2\delta))\exp(n(-\lambda_s+2\delta)) \\
        &= \exp(-n(\lambda_s-\lambda_{s+1}-4\delta))
    \end{align*}
    for $n$ large enough. Since $v$ was arbitrary, this means
    \[\Gap(A^nE,A^nE_{n,s})\leq \norm{\pi_{F_{n,s+1}'//E_{n,s}'}|_{A^nE}}\lesssim\exp(-n(\lambda_s-\lambda_{s+1}-4\delta)).\]
    The rest follows by \autoref{lem:equivOfGap_and_Hausdorff}.
\end{proof}

\begin{proof}[Proof of item (4)]
    Let $1\leq s < r+1$. Take $0 < \eta \leq 1$ and $E\in\G\Ccal_s(\eta)$, and let $\delta, H, \text{ and } B$ be as in \autoref{claim:ImagesOfFastSpacesSqueeze} for $S = s+1$. Lastly, take $n$ large enough.

    \noindent\textbf{Step one:} The first step is to define an inductive procedure that divides $E$ into an Oseledets-like splitting of ``fast spaces." Let $U_s \coloneqq E \cap F_{s}$, which is clearly an $m_s$-dimensional subspace. By definition of $\G\Ccal_s(\eta)$ and minimal angle, it follows that 
    \[\sin\theta(U_s,F_{s+1})\geq \sin\theta(E,F_{s+1}) \geq \eta.\]
    By \autoref{lem:closed_q_dim_complemented}, there is a subspace $E_{s-1}\subset E$ of dimension $M_{s-1}$ such that
    \[\sin\theta(U_{s},E_{s-1})\geq \frac{1}{\sqrt{m_s}}.\]
    Let $U_{s - 1}\coloneqq E_{s-1}\cap F_{s-1}$, which, similarly to before, has dimension $m_{s-1}$. Now, since
    \[\pi_{E_{s-1}//F_{s}} = \paren{\pi_{E_{s-1}//U_s}|_{E_s}} \pi_{E_{s}//F_{s+1}},\]
    it follows that
    \[\sin\theta(U_{s-1},F_s)\geq\sin\theta(E_{s-1},F_s)\geq \frac{\eta}{2\sqrt{m_s}}.\]
    Keeping this process going inductively, we get for $1\leq t\leq s$ a collection of spaces
    \[E_t = U_1\oplus \cdots \oplus U_t\subseteq E\]
    such that
    \begin{equation}\label{eq:OseledetsLikeSplittingAngles}
        \sin\theta(U_t,F_{t+1})\geq\sin\theta(E_t,F_{t+1})\geq \frac{\eta}{2^{s-t}\sqrt{m_s\cdots m_{t}}} \eqqcolon \eta_t.
    \end{equation}
    Following the proofs of items (2) and (3), we get that, for all $\delta > 0$ and $n\in\N$,
    \[\exp(n(\lambda_{t}-2\delta))\lesssim m(A^n|_{E_t})\leq m(A^n|_{U_t})\leq \norm{A^n|_{U_t}}\lesssim \exp(n(\lambda_{t}+7t^2\delta)).\]

    \noindent\textbf{Step two:} The next step is to compute angle estimates. Let $2\leq t\leq s$ and take $u\in U_t\setminus\innerp{0}$. Then,
    \begin{equation}\label{eq:InitialAngleOfImagesEstimate}
        d(A^n u, A^n E_{t-1}) \geq d(A^n u, A^nE_{n,t-1}) - 4\norm{A^n u}d_H(A^n E_{t-1},A^n E_{n,t-1}).
    \end{equation}
    Of course,
    \[d(A^n u, A^nE_{n,t-1}) = d(\pi_{F'_{n,t}//E'_{n,t-1}}A^n u, E_{n,t-1}')\geq \frac{1}{2}\norm{A^n \pi_{F_{n,t}//E_{n,t-1}} u}\sin\theta(E_{n,t-1}',F_{n,t}'),\]
    and $\sin\theta(E_{n,t-1}',F_{n,t}')\geq C_{s+1}^{-1} > 0$ by construction. Now, we may write
    \[\pi_{F_{n,t}//E_{n,t-1}} u = u_{n,t} + w_{n,t}\]
    where $u_{n,t}\in U_{n,t}$ and $w_{n,t} \in F_{n,t+1}$. Take $v_{n,t} = \pi_{E_{n,t-1}//F_{n,t}} u$ so that
    \[e_{n,t}\coloneqq \pi_{E_{n,t}//F_{n,t+1}} u = v_{n,t} + u_{n,t}.\]
    By \eqref{eq:item(3)AngleEstimate}, it follows that $\norm{e_{n,t}}\geq \frac{\eta_t}{2}\norm{u}$ for $n$ large enough. However, by \autoref{lem:cauchySeqSpaces} and the fact that $u \in F_{t}$, it follows that
    \[\norm{v_{n,t}}\lesssim \exp(-n(\lambda_{t-1}-\lambda_{t}-5\delta))\norm{u}\leq \frac{\eta_t}{4}\norm{u}\]
    for $n$ large enough. Therefore,
    \[\norm{u_{n,t}}\geq\norm{e_{n,t}}-\norm{v_{n,t}}\geq \frac{\eta_t}{4}\norm{u},\]
    which means
    \begin{align*}
        \norm{A^n\pi_{F_{n,t}//E_{n,t-1}} u} &\geq \norm{A^n u_{n,t}}-\norm{A^n w_{n,t}} \\
        &\geq m\paren{A^n|_{U_{n,t}}}\norm{u_{n,t}} - \norm{A^n|_{F_{n,t+1}}}\norm{w_{n,t}} \\
        &\geq \exp(n(\lambda_{t}-2\delta))\frac{\eta_t}{4e^{2H} C_{s+1}}\norm{u} - e^{2H}C_{s+1}^2\exp(n(\lambda_{t+1}+2\delta))\norm{u} \\
        &\gtrsim e^{-9s^2\delta n}\norm{A^n|_{U_t}}\norm{u}
    \end{align*}
    for $n$ large enough. Substituting the above and \autoref{claim:ImagesOfFastSpacesSqueeze} into \eqref{eq:InitialAngleOfImagesEstimate} and using the fact that $\lambda_{t-1}-\lambda_t - 4\delta > 9s^2 \delta$, we get
        \[d(A^n u, A^n E_{t-1})\geq (2C_{s+1})^{-1}\norm{A^n\pi_{F_{n,t}//E_{n,t-1}} u}-d_H(A^n E_{t-1},A^n E_{n,t-1})\norm{A^nu} \gtrsim e^{-9s^2\delta n}\norm{A^n|_{U_t}}\norm{u}\]
    for all $n$ large enough. Therefore,
    \begin{equation}\label{eq:FinalAngleOfImagesEstimate}
        \sin\theta(A^n U_{t},A^nE_{t-1})\geq \inf_{u\in U_{t}\setminus\innerp{0}}\frac{1}{\norm{A^n|_{U_{t}}}\norm{u}}d(A^nu,A^n E_{t-1})\gtrsim e^{-9s^2\delta n}
    \end{equation}
    for $n$ large enough.

    \noindent\textbf{Step three:} The final step is to combine steps one and two to compute volume estimates. By \autoref{lem:determinantProdDecomp} and \eqref{eq:FinalAngleOfImagesEstimate}, it follows that
    \[\det(A^n|E)\gtrsim \det(A^n|U_1)\cdots\det(A^n|U_s) e^{-9s^2M_sn\delta}.\]
    Now,
    \[\det(A^n|U_t)\geq (m_t)^{-m_t} m(A^n|U_t)^{m_t}\gtrsim\exp(nm_t(\lambda_t-2\delta)).\]
    \[\inf_{E\in \G\Ccal_s(\eta)}\det(A^n|E)\gtrsim \exp(n(m_1\lambda_1+\cdots+m_s\lambda_s-11s^2M_s\delta)),\]
    and thus
    \[\liminf_{n\to\infty} \frac{1}{n}\log \inf_{E\in \G\Ccal_s(\eta)}\det(A^n|E)\geq \sum_{t=1}^s m_t\lambda_t - 11s^2M_s\delta.\]
    Since $\delta>0$ was arbitrary, the $\liminf$ must be bounded below by the desired sum. Moreover,
    \[\det(A^n|E)\leq V_{M_s}(A^n)\leq e^{H+n(m_1\lambda_1+\cdots+m_s\lambda_s+\delta)},\]
    so the desired limit converges uniformly to the desired sum.
\end{proof}

\begin{remark}
    Notice that \eqref{eq:FinalAngleOfImagesEstimate} is a weaker form of item (1) in \autoref{thm:refined_sequentialMET_oneSided}.
\end{remark}

\subsubsection*{Proof of \autoref{lem:cauchySeqSpaces}}
It will be useful to start at a higher level of generality in order to establish a basic one-step inequality for future proofs. As such, let $2\leq s'\leq s$ and suppose for some $\tC_{s'}\geq 1$ we also have a sequence $(\tE_{n,s'-1},\tE_{n,s'-1}',\tF_{n,s'},\tF_{n,s'}')\in \SVS_{M_{s'-1}}(\tC_{s'};A^n)$ for all $n$.

\begin{claim}\label{claim:oneStepGapBound}
    For all $n\geq 0$,
    \begin{equation}\label{eq:oneStepGapBound}
    \Gap(F_{n,s},\tF_{n+1,s'}) \leq  \tC_{s'}^{2} C_{s}\norm{A_{n}}\frac{c_{1+M_{s-1}}(A^n)}{c_{M_{s'-1}}(A^{n+1})}.
\end{equation}
\end{claim}
\begin{proof}
By \autoref{prop:gapProjEquiv},
\[\Gap(F_{n,s},\tF_{n+1,s'})\leq\norm{\pi_{\tE_{n+1,s'-1}//\tF_{n+1,s'}}|_{F_{n,s}}},\]
Take $v\in F_{n,s}$ of norm $1$ and let
\[u = \pi_{\tE_{n+1,s'-1}//\tF_{n+1,s'}} v.\]
Then for all $n\geq 0$, we have
\[\norm{A^{n+1}u}\geq \tC_{s'}^{-1}c_{M_{s'-1}}(A^{n+1})\norm{u}\]
and
\[\norm{A^{n+1}v} \leq \norm{A_{n}}\norm{A^n v}\leq C_{s}\norm{A_{n}}c_{1+M_{s-1}}(A^n).\]
Clearly,
\[A^{n+1}u = \pi_{\tE'_{n+1,s'-1}//\tF'_{n+1,s'}}A^{n+1}v,\]
so $\norm{A^{n+1}u}\leq \tC_{s'}\norm{A^{n+1}v}$. Therefore,
\[\norm{u} \leq \tC_{s'}^{2} C_{s}\norm{A_{n}}\frac{c_{1+M_{s-1}}(A^n)}{c_{M_{s'-1}}(A^{n+1})}\]
In other words,
\[\Gap(F_{n,s},\tF_{n+1,s'})\leq \norm{\pi_{\tE_{n+1,s'-1}//\tF_{n+1,s'}}|_{F_{n,s}}}\leq  \tC_{s'}^{2} C_{s}\norm{A_{n}}\frac{c_{1+M_{s-1}}(A^n)}{c_{M_{s'-1}}(A^{n+1})}.\]
\end{proof}
\begin{remark}\label{remark:generalizationOfOneStepGapBound}
    The derivation of \eqref{eq:oneStepGapBound} actually only requires that 
    \[\norm{A^n|_{F_{n,s}}}\leq C_sc_{1+M_{s-1}}(A^n)\]
    for some constant $C_s$ independent of $n$.
\end{remark}

We may now prove the lemma. We note that this part is rather standard even for Banach space cocycles.
\begin{proof}[Proof of \autoref{lem:cauchySeqSpaces}]
    Now assume $s=s'$ and $(E_{n,s-1},E_{n,s-1}',F_{n,s},F_{n,s'})=(\tE_{n,s-1},\tE_{n,s-1}',\tF_{n,s},\tF_{n,s'})$. By \autoref{claim:equiv_of_vol_gelfand_pesin_funcs}, $\norm{A_{n}}c_{1+M_{s-1}}(A^n) \lesssim_{M_{s-1},H,B} \exp(n(\lambda_{s}+3\delta))$ and $c_{M_{s-1}}(A^{n+1})\gtrsim_{M_{s-1},H,\delta,\lambda_{s-1}} \exp(n(\lambda_{s-1}-2\delta))$, and so
    \[2M_{s-1}C_s^3\norm{A_{n}}c_{1+M_{s-1}}(A^n)\leq c_{M_{s-1}}(A^{n+1})\]
    for $n$ large enough. Thus, by \autoref{lem:equivOfGap_and_Hausdorff} and \autoref{claim:oneStepGapBound}, it follows that
    \[d_H(F_{n,s},F_{n+1,s})\leq 4M_{s-1}\Gap(F_{n,s},F_{n+1,s})\leq 4M_{s-1}C_{s}^{3}\norm{A_{n}}\frac{c_{1+M_{s-1}}(A^n)}{c_{M_{s-1}}(A^{n+1})}.\]
    By repeated application of the triangle inequality,
    \begin{align*}
        d_H(F_{n,s},F_{n+k,s}) \leq \sum_{j=0}^{k-1} d_H(F_{n+j,s},F_{n+j+1,s}) \leq 4M_{s-1}C_{s}^3 \sum_{j=n}^{\infty}\norm{A_{j}} \frac{c_{1+M_{s-1}}(A^j)}{c_{M_{s-1}}(A^{j+1})}.
    \end{align*}
    Of course, as we estimated above, we get that $\norm{A_{j}} \frac{c_{1+M_{s-1}}(A^j)}{c_{M_{s-1}}(A^{j+1})}\lesssim \exp(-j(\lambda_{s-1}-\lambda_s-5\delta))$ for $j\geq 0$. As such,
    \[d_H(F_{n,s},F_{n+k,s})\lesssim \exp(-n(\lambda_{s-1}-\lambda_s-5\delta))\]
    for $n$ large enough and $k\geq 1$. This proves that $(F_{n,s})_n$ is a Cauchy sequence. 

    If $(\tE_{n,s-1},\tE_{n,s-1}',\tF_{n,s},\tF_{n,s}')\in \SVS_{M_{s-1}}(\tC_{s};A^n)$  is another sequence of singular value splittings of order $\tC_s\geq 1$, then \autoref{cor:uniquenessOfSlowSpaces} gives us $\tF_{n,s}$ and $F_{n,s}$ have the same limit. This corollary is a sort of quantitative uniqueness result, and we prove it below.
\end{proof}

\begin{cor}[Uniqueness of Slow Spaces]\label{cor:uniquenessOfSlowSpaces}
    Let $2\leq s < r+2$, $C_s \geq 1$, and $(E_{n,s-1},E'_{n,s-1},F_{n,s},F_{n,s}')\in \SVS_{M_{s-1}}(C_s;A^n)$. Let $F_{s}$ be the $\G(\B_0)$-gap-topology limit of $F_{n,s}$ as $n\to\infty$. Take $0\leq \e < \lambda_{s-1}-\lambda_{s}$ and suppose $\tF_{n,s}\in \G^{M_{s-1}}(\B_0)$ is such that, for some $\tC_s\geq 1$,
    \[\norm{A^n|_{\tF_{n,s}}} \leq \tC_s e^{n\e} c_{1+M_{s-1}}(A^n)\]
    for all $n\geq 1$. Take $H,B,$ and $\delta>0$ as in \autoref{lem:cauchySeqSpaces}. Assume further that $5\delta < \lambda_{s-1}-\lambda_s - \e$ and set $\gamma \coloneqq \lambda_{s-1}-\lambda_s - (\e+5\delta) > 0$. Then there are integers $K,N\in \N$ and a constant $C\geq 1$, all only dependent on $C_s, \delta, H, B, \lambda_{s},\lambda_{s-1},$ and $M_{s-1}$ such that for all $n\geq \max(N, (K + \log \tC_s)/\gamma)$,
    \[d_H(\tF_{n,s},F_{s})\leq \tC_sC\exp(-n\gamma).\]
    In particular, $\tF_{n,s}\to F_s$ as $n\to\infty$ in the gap-topology.
\end{cor}
\begin{proof}
    In this proof, the constant in $\lesssim$ shall depend only on $C_s, \delta, H, B, \lambda_{s},\lambda_{s-1},$ and $M_{s-1}$ and let $C\geq 1$ be a generic constant depending on this data. By \autoref{claim:oneStepGapBound} and \autoref{remark:generalizationOfOneStepGapBound}, it follows that
    \[\Gap(\tF_{n,s},F_{n+1,s})\leq C_s^2 \tC_s e^{n\e} \norm{A_n}\frac{c_{1+M_{s-1}}(A^n)}{c_{M_{s-1}}(A^{n+1})}\leq  \tC_s C \exp(-n\gamma).\]
    for all $n\geq 0$. If $K \geq \log(2M_{s-1}C)$ and $n\geq [K + \log(\tC_s)]/\gamma$, we get by \autoref{lem:equivOfGap_and_Hausdorff} that
    \[d_H(\tF_{n,s},F_{n+1,s})\lesssim \tC_s\exp(-n\gamma).\]
    For $n$ large enough (dependent only on $C_s, \delta, H, B, \lambda_{s},\lambda_{s-1},$ and $M_{s-1}$), we know that
    \[d_H(F_{n+1,s},F_{s})\lesssim \exp(-n(\lambda_{s-1}-\lambda_s-5\delta)),\]
    so by the triangle inequality, we obtain the estimate we seek. Since $\e+5\delta < \lambda_{s-1}-\lambda_{s}$, we obtain exponential convergence, and thus uniqueness.
\end{proof}

\subsubsection*{Proof of \autoref{lem:forwardGapSepEstimate}}
\paragraph{Proof strategy} As we noted in the prologue of this subsection, we shall consider an arbitrary index $2\leq S < r+2$. Take $\delta, H,\text{ and } B$ as in the statement of the lemma. We shall proceed inductively, first assuming that there is a $0\leq T\leq S-2$ such that for all pairs $2\leq s'\leq s\leq S$ with $s-s'\leq T$, we have
\[\Gap(F_{n,s},F_{n+k,s'})\leq C \exp(-n(\lambda_{s'-1}-\lambda_{s}-5(1+T)\delta))\]
for all $n\geq N$, where $N\in \N$ and $C\geq 1$ depend on $\alpha_S, \delta, H, B$ and $\OSE_S$. We then derive this estimate for $T+1$, with the induction terminating at $T = S-2$. The base case $T=0$ is proven by \autoref{lem:cauchySeqSpaces}. The constants $N$ and $C$ may differ going from $T$ to $T+1$, but since our induction is finite, this is not an issue. From now on, unnamed constants shall depend solely on $\alpha_S, \delta, H, B$ and $\OSE_S$.

This proof follows along the lines of Ruelle's strategy for the proof of equation (1.8) in \cite{Ruelle1982}, although there is a difference in how we handle projection estimates. Fix a large $n$ (we shall drop $\delta$'s in the following exposition). By \autoref{prop:gapProjEquiv}, $\Gap(F_{n,s},F_{n+k,s'})\leq\norm{\pi_{E_{n+k,s'-1}//F_{n+k,s'}}|_{F_{n,s}}}$, so we would like to take $v\in F_{n,s}$ of unit norm and consider $w_{k} = \pi_{E_{n+k,s'-1}//F_{n+k,s'}} v$. Ruelle obtains the formula
\[w_{k+1} = \pi_{E_{n+k,s'-1}//F_{n+k,s'}} w_k + u_k\]
and, via the induction argument, derives the estimate $\norm{u_k}\lesssim e^{-n(\lambda_{s'-1}-\lambda_s)}e^{-\gamma k}$, $\gamma>0$, at each step $k$. As Ruelle has access to inner products, he can take $E_{n+k,s'-1}$ to be the orthogonal complement of $F_{n+k,s'}$, which means $\pi_{E_{n+k,s'-1}//F_{n+k,s'}}$ has unit operator norm, and so $\norm{w_{k+1}}\leq \norm{w_k} + Ce^{-n(\lambda_{s'-1}-\lambda_s)}e^{-\gamma k}$, giving us a convergent exponential series that bounds $\norm{w_k}$. In our case, however, we can only assume $\norm{\pi_{E_{n+k,s'}//F_{n+k,s'}}}\lesssim 1$, which means we must pay an extra factor of $\norm{\pi_{E_{n+k,s'}//F_{n+k,s'}}}$ at every step. Fortunately, since complementation is open, we may instead fix a single complement for all $\paren{F_{n+k,s'}}_{k\geq 0}$ by \autoref{lem:cauchySeqSpaces}. In particular, for $v\in F_{n,s}$, we can instead take $w_k = \pi_{E_{n,s'-1}//F_{n+k,s'}} v$. Since $w_k\in E_{n,s'-1}$ for all $k$, we obtain $w_{k+1} = w_k + u_k'$ where we estimate $\norm{u_k'}\lesssim e^{-n(\lambda_{s'-1}-\lambda_s)}e^{-\gamma k}$ via the induction argument. As before, we get a summable bounding exponential series, as we have avoided paying a factor of $\norm{\pi_{E_{n,s'-1}//F_{n+k,s'}}}$, and can use \autoref{prop:gapProjEquiv} to bound $\Gap(F_{n,s},F_{n+k,s'})$.

\begin{remark}
    We emphasize that the induction argument we refer to does not induct on $k$; it inducts on the number of spectral levels between the time-$n$ slow space and time-$(n+k)$ fast space.
\end{remark}

\paragraph{Preliminary constructions} Before we begin, we note the following consequence of \autoref{lem:cauchySeqSpaces} and \autoref{prop:complementationOpen}: for all $n$ large enough, $k\geq 0$, and $1\leq t\leq S$, $\B_0 = E_{n,t-1} \oplus F_{n+k,t}$ is a topological splitting with
\begin{equation}\label{eq:uniformTopSplitBound}
    \norm{\pi_{E_{n,t-1}//F_{n+k,t}}} \lesssim 1.
\end{equation}
Together with \autoref{prop:gapProjEquiv}, we get
\begin{equation}\label{eq:mixedGapBound}
    \norm{\pi_{E_{n,t'-1}//F_{n+k,t'}}|_{F_{n,t}}}\lesssim \Gap(F_{n,t},F_{n+k,t'}).
\end{equation}
Thus, for the sake of brevity, we will write
\[\pi_{n,k,t-1}^E = \pi_{E_{n,t-1}//F_{n+k,t}} \text{ and } \pi_{n,k,t}^F = \pi_{F_{n+k,t}//E_{n,t-1}}.\]
It will also be helpful to consider the linear operators $P_{n,k,t}:\B_0\to F_{n+k,t}\cap E_{n,t}$ defined by
\[P_{n,k,t} \coloneqq \pi_{n,k,t}^E - \pi_{n,k,t-1}^E = \pi_{n,k,t}^F - \pi_{n,k,t+1}^F\]
for $1\leq t\leq S-1$. Since $F_{n+k,t+1}\subset F_{n+k,t}$ and $E_{n,t-1}\subset E_{n,t}$, it is easy to see that
\[P_{n,k,t} = P_{n,k,t} \pi_{n,k,t}^E,\]
so by \eqref{eq:mixedGapBound},
\begin{equation}\label{eq:weirdProjBound}
    \norm{P_{n,k,t'}|_{F_{n,t}}} \lesssim \norm{\pi_{n,k,t'}^E|_{F_{n,t}}}\lesssim \Gap(F_{n,t},F_{n+k,t'+1}).
\end{equation}
We now proceed with the induction argument. 

\paragraph{Induction step}
Fix $2\leq s\leq S$. Let $s' = s - T - 1$ and assume $s'\geq 2$, as otherwise the result is trivial. Take $n$ large and $k\geq 0$ and let $v\in F_{n,s}$ be of unit norm. Write $w_k = \pi_{n,k,s'-1}^Ev$. One can easily see
\[v = \pi_{n,k,s}^F v +\paren{w_k + \sum_{t=s'}^{s-1} P_{n,k,t}v}.\]
Fix $s'\leq t\leq s-1$. Then, by the induction hypothesis and \eqref{eq:weirdProjBound},
\[\norm{P_{n,k,t}v}\lesssim \exp(-n(\lambda_{t}-\lambda_{s}-5(1+T)\delta))\]
Moreover, $P_{n,k,t}v\in F_{n+k,t}$, so by \autoref{claim:oneStepGapBound} (after replacing $n$ by $n+k$), \eqref{eq:uniformTopSplitBound}, and \autoref{prop:gapProjEquiv}
\begin{align*}
    \norm{\pi_{n,k+1,s'-1}^E P_{n,k,t}v} &\lesssim \exp(-(n+k)(\lambda_{s'-1}-\lambda_{t}-5\delta))\norm{P_{n,k,t}v} \\
    &\lesssim \exp(-n(\lambda_{s'-1}-\lambda_s-5(2+T)\delta))\exp(-k(\lambda_{s'-1}-\lambda_t-5\delta)).
\end{align*}
Similarly,
\[\norm{\pi_{n,k+1,s'-1}^E\pi_{n,k,s}^F v}\lesssim \exp(-(n+k)(\lambda_{s'-1}-\lambda_{s}-5\delta))\]
since $\pi_{n,k,s}^F v\in F_{n+k,s}$. Notice now that
\[w_{k+1} = w_{k} + \pi_{n,k+1,s'-1}^E\pi_{n,k,s}^F v + \sum_{t=s'}^{s-1} \pi_{n,k+1,s'-1}^E P_{n,k,t}v,\]
which means, by combining \eqref{eq:uniformTopSplitBound} and the previous two estimates, that
\[\norm{w_{k+1}} \leq \norm{w_k} + C\exp(-n(\lambda_{s'-1}-\lambda_s-5(2+T)\delta))\sum_{t=s'}^{s} \exp(-k(\lambda_{s'-1}-\lambda_t-5\delta))\]
where $C$ only depends on $\alpha_S, \delta, H, B$ and $\OSE_S$, and more importantly, is independent of $n$ and $k$. Since $w_0 = 0$, we get
\[\norm{\pi_{n,k,s'-1}^Ev} = \norm{w_k}\lesssim \exp(-n(\lambda_{s'-1}-\lambda_{s}-5(2+T)\delta)).\]
Taking a supremum over $v\in F_{n,s}$ of unit norm, we obtain the desired result for all $k\geq 0$.

\subsubsection*{Proof of \autoref{lem:backwardGapSepEstimate}}
\paragraph{Proof strategy} Fix $2\leq S < r+2$ as in \autoref{lem:forwardGapSepEstimate}, and let $\alpha_S,\delta,H, \text{ and } B$ be as in the statement of \autoref{lem:backwardGapSepEstimate}. Our argument will roughly follow along the lines of Ruelle \cite{Ruelle1982} via a cofactor-inversion argument on a finite-dimensional matrix. Since our ambient space is potentially infinite-dimensional, at times $n$ and $n+k$ we resolve the Lyapunov exponents down to $\lambda_{S-1}$ and then augment down by one vector in the next slow space. However, due to the lack of a global inner product structure, we cannot form a unitary change of basis matrix as done by Ruelle. Instead, we construct two $(M_{S-1}+1) \times (M_{S-1}+1)$-matrices which are approximate inverses of each other. \autoref{lem:forwardGapSepEstimate} will give us estimates on the upper triangular portion of one matrix, and through the approximate inversion relationship, we obtain estimates on the upper triangular portion of the other.

\begin{remark}
    We note that the following argument is more involved than the previous two lemmas, and so it shall be broken down into a variety of steps.
\end{remark}

\paragraph{Preliminary constructions} Let $(u_{n,q})_{q=1}^{M_{S-1}}$ and $(u_{n+k,q})_{q=1}^{M_{S-1}}$ be almost orthonormal right singular vectors for 
    \[A^n|_{E_{n,S-1}} \text{ and } A^{n+k}|_{E_{n+k,S-1}},\] 
respectively, from \autoref{prop:SVD}. By definition, we have for both $\ell=n\text{ and } n+k$ that $u_{\ell,q}\in U_{\ell,t}$ if $M_{t-1} < q \leq M_{t}$. Additionally, take an arbitrary auxiliary vector $f\in F_{n+k,S}$ of norm $1$ and label it $u_{n+k,M_{S-1}+1}$. Write
    \[u_{n,M_{S-1}+1} = \frac{\pi_{F_{n,S}//E_{n,S-1}}f}{\norm{\pi_{F_{n,S}//E_{n,S-1}}f}}.\]
We shall also consider new spaces
\begin{itemize}
    \item $\ES_{n+k,S} = \ES_{n+k,S-1} + \innerp{f}$ and $\FS_{n+k,S+1}\subsetneq \FS_{n+k,S}$ a complement of $\innerp{f}$ such that (guaranteed by \autoref{lem:closed_q_dim_complemented})
    \[\norm{\pi_{\innerp{f}//\FS_{n+k,S+1}}|_{\FS_{n+k,S}}} = 1\]
    \item $\ES_{n,S} = \ES_{n,S-1} + \innerp{u_{n,M_{S-1}+1}}$ and $\FS_{n,S+1} = \pi_{\FS_{n,S}//\ES_{n,S-1}} \FS_{n+k,S+1}\subsetneq \FS_{n,S}$
    \item $\US_{n+k,S} = \innerp{f}$ and $\US_{n,S} = \innerp{u_{n,M_{S-1}+1}}$
\end{itemize}
and denote $m_s'=\dim \US_{n,s}$ and $M_{s}' = \dim \ES_{n,s}$ for $1\leq s\leq S$. Note that $m_s'$ and $M_{s}'$ agree with $m_s$ and $M_s$ for $1\leq s\leq S-1$, respectively, while $m_{S}' = 1$ and $M_{S}' = M_{S-1}+1$. We will drop the $(S)$ superscript for the remainder of this proof.

\autoref{lem:cauchySeqSpaces} shows that
\begin{equation}\label{eq:nearIdentityOfProj}
    \norm{(\id - \pi_{F_{n+k,S}//E_{n+k,S-1}})|_{F_{n,S}}},\norm{(\id - \pi_{F_{n,S}//E_{n,S-1}})|_{F_{n+k,S}}}\lesssim \exp(-n(\lambda_{S-1}-\lambda_{S}-5\delta)).
\end{equation}
Since complementation is open, we can guarantee that $\pi_{F_{n,S}//E_{n,S-1}}|_{F_{n+k,S}}$ is boundedly invertible for all $n$ large enough (with inverse $\pi_{F_{n+k,S}//E_{n,S-1}}|_{F_{n,S}}$), and so $F_{n,S+1}$ is a topological complement to $\innerp{u_{n,M_{S-1}+1}}$.
\begin{remark}
    Should it be the case that $F_{n+k,S} = \innerp{0}$, which is only an issue if $\B_0$ is finite-dimensional and $M_{S-1} = \dim \B_0$, we simply omit this auxiliary choice of $f$.
\end{remark}

\vspace{2mm}
\noindent\textit{Extended inner product structure}
\vspace{2mm}

\noindent It will be useful to extend the inner product structures of $E_{n+k,S-1}$ and $E_{n,S-1}$ to $E_{n+k,S}$ and $E_{n,S}$. Naturally, we do this by inducing
\begin{align*}
    \paren{v_1,v_2}_{E_{n+k,S}} &= \paren{v_1,v_2}_{E_{n+k,S-1}} \text{ if } v_1,v_2\in E_{n+k,S-1} \\
    \paren{v,f}_{E_{n+k,S}} &= 0 \text{ if } v\in E_{n+k,S-1}\\
    \paren{f,f}_{E_{n+k,S}} &= 1
\end{align*}
and extending by linearity. The induced norm is denoted by $\abs{\cdot}_{E_{n+k,S}}$. Enforcing $\abs{u_{n,M_{S}'}}_{E_{n,S}} = 1$, the extension from $E_{n,S-1}$ to $E_{n,S}$ is analogous. We will suppress subscripts for $\paren{\cdot,\cdot}$ and $\abs{\cdot}$ when context is clear.

We must see how these new inner product structures compare with the original Banach space norm. Take $v\in E_{n+k,S}$. Then we may write $v = v_{S-1} + v_{S}$ where $v_{S-1}\in E_{n+k,S-1}$ and $v_{S}\in U_{n+k,S}$. One can see that
    \[\abs{v} \leq \abs{v_{S-1}}+\abs{v_{S}}\leq \paren{\sqrt{M_{S-1}}\norm{\pi_{E_{n+k,S-1}//F_{n+k,S}}}+\norm{\pi_{F_{n+k,S}//E_{n+k,S-1}}}}\norm{v}\leq 2\sqrt{M_{S-1}}C_S \norm{v}\]
and
    \[\norm{v}\leq \norm{v_{S-1}}+\norm{v_{S}}\leq \sqrt{M_{S-1}}\abs{v_{S-1}}+\abs{v_{S}}\leq \sqrt{2M_{S-1}}\abs{v}.\]
The same bounds clearly follow for $E_{n,S}$. Notice the constants in front of $\norm{v}$ and $\abs{v}$ are independent of $f$.

\begin{remark}
    Estimates made in this proof will be uniform over any choice of $f\in F_{n+k,S}$ of norm 1.
\end{remark}

\vspace{2mm}
\noindent\textit{Reduction to finite dimensions}
\begin{claim}\label{claim:projOfSingVecs}
    Let $1\leq t' < s\leq S$. Then there is a constant $C\geq 1$ and integer $N\in \N$, dependent on $\alpha_S, \delta, H, B$, and $\OSE_S$, such that for all $n\geq N$, $k\in \N$, all $f\in F_{n+k,S}$ of norm $1$ we have
    \begin{equation}\label{eq:projOfSingVecs}
        \norm{\pi_{U_{n,t'}//\bigoplus_{t\neq t'} U_{n,t} + F_{n,S+1}}|_{U_{n+k,s}}}\leq C\exp(-n(\lambda_{t'}-\lambda_{s}-5(S-1)(s-t')\delta)).
    \end{equation}
\end{claim}
With this claim, we are able to prove \autoref{lem:backwardGapSepEstimate}. We shall show this and provide the proof of \autoref{claim:projOfSingVecs} in the sequel.
\begin{proof}[Proof of \autoref{lem:backwardGapSepEstimate} given \autoref{claim:projOfSingVecs}]
    Let $2\leq s'\leq s \leq S$. Take $v\in F_{n+k,s}$ and decompose $v = u_{s}+\cdots+u_{S-1} + w$ where $w\in F_{n+k,S}$ and $u_{t}\in U_{n+k,t}$. If $w\neq 0$, we set $f = w/\norm{w}$. Otherwise, the choice of $f$ does not matter. For the sake of convenience, we write $u_{S} \coloneqq w$.
    
    If $1\leq t' \leq s'-1$ and $s\leq t\leq S$, then by \autoref{claim:projOfSingVecs} we have
    \[ \norm{\pi_{U_{n,t'}//\bigoplus_{t''\neq t'} U_{n,t''} + F_{n,S+1}} u_t}\leq C\exp(-n(\lambda_{t'}-\lambda_{s}-5(S-1)(s-t')\delta))\norm{u_t}.\]
    In particular, summing everything between $1\leq t'\leq s'-1$ and $s\leq t\leq S$ up, we get
    \[\norm{\pi_{E_{n,s'-1}//F_{n,s'}} v}\leq (S-s+1)(s'-1)C_SC\exp(-n(\lambda_{s'-1}-\lambda_s-5(S-1)(1+s-s')\delta)\norm{v}.\]
    Taking a supremum over $v\in F_{n+k,s}$ of unit norm, we obtain \eqref{eq:backwardGapSepEstimate}.
\end{proof}
\paragraph{Matrix calculations}
As we have reduced the lemma to proving a finite-dimensional statement, we may leverage the coordinate matrix representations of linear transformations. Let $\delta > 0$ and WLOG assume
    \[\delta<\min\curly{\lambda_{s'-1}-\lambda_s:2\leq s'\leq s\leq S}.\]
Take an arbitrary norm 1 vector $f\in F_{n+k,S}$, with all associated constructions previously defined. Consider the following operators
\begin{itemize}
    \item $A = \pi_{E_{n+k,S}//F_{n+k,S+1}}|_{E_{n,S}}:E_{n,S}\to E_{n+k,S}$
    \item $A' = \pi_{E_{n,S}//F_{n,S+1}}|_{E_{n+k,S}}:E_{n+k,S}\to E_{n,S}$
    \item $B = \pi_{E_{n+k,S}//F_{n+k,S+1}}|_{F_{n,S+1}}:F_{n,S+1}\to E_{n+k,S}$
    \item $C = \pi_{F_{n,S+1}//E_{n,S}}|_{E_{n+k,S}}:E_{n+k,S}\to F_{n,S+1}$
\end{itemize}
Then it is clear that
\begin{equation}\label{eq:futureFastSpaceIdentityDecomp}
    \id_{E_{n+k,S}} = AA' + BC.
\end{equation}
Our goal is to compute $A'$. Now, for $1\leq i,j\leq M_{S}'$, write $[A]_{ij} = \paren{Au_{n,j},u_{n+k,i}}$ and $[A']_{ij} = \paren{A'u_{n+k,j},u_{n,i}}$. We will also write $[A] = ([A]_{ij})_{ij}$ and $[A'] = ([A']_{ij})_{ij}$. These are the matrix representations of $A$ and $A'$ w.r.t. the orthonormal bases of $E_{n+k,S}$ and $E_{n,S}$ we've defined. We similarly let $[BC]$ be the matrix representation of $BC\in L(E_{n+k,S},E_{n+k,S})$ w.r.t. $(u_{n+k,q})_q$.

\vspace{2mm}
\noindent\textit{An identity for $A'$}
\vspace{2mm}

\noindent By \eqref{eq:nearIdentityOfProj}, it follows that
\[d_H(F_{n,S+1},F_{n+k,S+1})\lesssim \exp(-n(\lambda_{S-1}-\lambda_{S}-5\delta)).\]
Thus, from \autoref{prop:complementationOpen}, $\B_0 = E_{n,S}\oplus F_{n+k,S+1}$ is a topological splitting for $n$ large enough. It is easy to see that this implies 
\[\pi_{E_{n,S}//F_{n+k,S+1}} = \pi_{E_{n,S}//F_{n+k,S+1}}\pi_{E_{n+k,S}//F_{n+k,S+1}}\] 
so that $A^{-1} = \pi_{E_{n,S}//F_{n+k,S+1}}|_{E_{n+k,S}}$ and $\norm{A^{-1}}\lesssim 1$. So we aim to estimate some portion of the upper triangular part of
\begin{equation}\label{eq:approximateProjInversion}
    [A'] = [A^{-1}](I - [BC]).
\end{equation}

\vspace{2mm}
\noindent\textit{Estimates on the entries of $[A^{-1}]$}
\vspace{2mm}

\noindent From Cramer's rule we see
\[[A^{-1}] = \frac{1}{\det([A])}\adj([A]).\]
Since we coordinatized with almost orthogonal bases, we have that
\begin{equation}\label{eq:detBoundBelow}
    \abs{\det([A])}\geq \abs{A^{-1}}_{E_{n+k,S}\to E_{n,S}}^{-M_{S}'}\gtrsim \norm{A^{-1}}^{-M_{S}'}\gtrsim 1.
\end{equation}

From now on, we will write $s(i)$ to denote the unique integer between $1$ and $S$ such that $M_{s(i)-1}' < i \leq M_{s(i)}'$. To estimate the adjugate, notice for all pairs $(i,j)$ for which $1\leq s(i) < s(j)\leq S$, it follows from \autoref{lem:forwardGapSepEstimate}, \autoref{prop:gapProjEquiv}, and orthonormality that
\begin{align}
    \abs{[A]_{ij}} &= \abs{\pi_{\innerp{u_{n+k,i}}//\innerp{u_{n+k,q}:q\neq i}+F_{n+k,S+1}} u_{n,j}} \notag\\
    &\leq \abs{\pi_{E_{n+k,s(i)}//F_{n+k,s(i)+1}}u_{n,j}} \notag \\
    &\leq 4M_{S-1}C_S^2 \norm{\pi_{E_{n+k,s(i)}//F_{n+k,s(i)+1}}|_{F_{n,s(j)}}} \notag \\
    &\lesssim \exp\paren{-n(\lambda_{s(i)}-\lambda_{s(j)}-5(s(j)-s(i))\delta)}. \label{eq:AcoordsBound}
\end{align}
Now consider a pair $(i,j)$ such that $s(i)<s(j)$. As the adjugate is the transpose of the cofactor matrix of $A$, we see from Leibniz's formula that
\begin{equation}\label{eq:adjugateLeibnizBound}
    \abs{\adj([A])_{ij}} \leq \sum_{\sigma\in \Bij(\Gamma_j,\Gamma_i)}\abs{\prod_{p\in \Gamma_j}[A]_{p\sigma(p)}}
\end{equation}
where $\Bij(\Gamma_j,\Gamma_i)$ is the set of bijections from $\Gamma_j = \curly{1,\ldots,M'_S}\setminus\curly{j}$ to $\Gamma_i = \curly{1,\ldots,M'_S}\setminus\curly{i}$. Fix an arbitrary bijection $\sigma\in \Bij(\Gamma_j,\Gamma_i)$ and note the following claim.

\begin{claim}
    For all $s(i)\leq s < s(j)$, there is a $p\in \Gamma_j$ such that $s(p)\leq s < s(\sigma(p))$.
\end{claim}
\begin{proof}
    Indeed, since $s(i)\leq s<s(j)$, it follows that
    \[\card\curly{p\in\Gamma_j:s(p)\leq s} = \card\curly{p:s(p)\leq s}= M_s,\]
    while
    \[\card\curly{q\in\Gamma_i:s(q)\leq s} = \card\paren{\curly{q:s(q)\leq s}\setminus\curly{i}} = M_s - 1.\]
    But $\sigma$ is a bijection, so it must be the case that
    \[\card\paren{\curly{p\in\Gamma_j:s(p)\leq s}\cap\sigma^{-1}\curly{q\in\Gamma_i:s(q)\leq s}^{c}}\geq 1.\]
    That is, there must be a $p\in \Gamma_j$ such that $s(p)\leq s < s(\sigma(p))$.
\end{proof}
In particular, this claim implies there is a sequence of $R\leq s(j)-s(i)$ integers $p_{1},\ldots,p_{R}\in \Gamma_j$ such that $s(p_1) \leq s(i) < s(\sigma(p_1))$, $s(p_R) < s(j) \leq s(\sigma(p_R))$, and
    \[s(p_{\ell}) \leq s(\sigma(p_{\ell-1})) < s(\sigma(p_{\ell}))\]
for all $2\leq \ell\leq R$. Thus, letting $K$ be the constant in \eqref{eq:AcoordsBound}, we see by telescoping that
\begin{align*}
    \abs{\prod_{p\in \Gamma_j}[A]_{p\sigma(p)}} &\leq \abs{A}^{M_{S-1}-R}K^R\exp(5(S-1)R\delta n)\prod_{\ell=1}^{R}\exp\paren{-n(\lambda_{s(p_\ell)}-\lambda_{s(\sigma(p_\ell))})}\\
    &\leq \max(\abs{A},K)^{M_{S-1}} \exp\paren{-n(\lambda_{s(i)}-\lambda_{s(j)}-5(S-1)(s(j)-s(i))\delta)}.
\end{align*}
Substituting this into \eqref{eq:adjugateLeibnizBound}, we see 
    \[\abs{\adj([A])_{ij}}\leq M_{S-1}!\max(\abs{A},K)^{M_{S-1}} \exp\paren{-n(\lambda_{s(i)}-\lambda_{s(j)}-5(S-1)(s(j)-s(i))\delta)},\]
and as such, by \eqref{eq:detBoundBelow},
\begin{equation}\label{eq:upperTriangular-invA}
    \abs{[A^{-1}]_{ij}}\lesssim \exp(-n(\lambda_{s(i)}-\lambda_{s(j)}-5(S-1)(s(j)-s(i))\delta))
\end{equation}
for all $(i,j)$ such that $s(j)>s(i)$.
    
\vspace{2mm}
\noindent\textit{Estimates on the entries of $[A']$}
\vspace{2mm}

\noindent Having estimated $[A^{-1}]$, we may now consider the full product \eqref{eq:approximateProjInversion}. Write $[D] = I-[BC]$ and assume $s(j)>s(i)$. Recall that
    \[[A']_{ij} = \sum_{q=1}^{M'_{S}} [A^{-1}]_{iq}[D]_{qj}.\]
Firstly, we note that
    \[\abs{[D]_{qj}}\leq 1+4M_{S-1}C_S^2\norm{\pi_{E_{n+k,S}//F_{n+k,S+1}}}\norm{\pi_{F_{n,S+1}//E_{n,S}}}\lesssim 1\]
for all $q$ and $j$. Now, denote the smallest integer $q > 1$ such that $s(q)\geq s(j)$ by $q^*$. Then for $q\geq q^*$ we have
\[\lambda_{s(i)}-\lambda_{s(q)} - 5(S-1)(s(q)-s(i))\delta \geq \lambda_{s(i)}-\lambda_{s(j)} - 5 (S-1)(s(j)-s(i))\delta\] 
since $\lambda_{s(j)}-\lambda_{s(q)}\geq 5(S-1)(s(q)-s(j))\delta$. Thus, it follows from \eqref{eq:upperTriangular-invA} that
    \[\abs{[A^{-1}]_{iq}[D]_{qj}}\lesssim \exp(-n(\lambda_{s(i)}-\lambda_{s(j)}-5(S-1)(s(j)-s(i))\delta)).\]
If $q < q^*$, then by \autoref{lem:forwardGapSepEstimate} and the fact that $F_{n,S+1}\subseteq F_{n,s(j)}$ we have
\begin{align*}
    \abs{[D]_{qj}} &= \abs{\pi_{\innerp{u_{n+k,q}}//\innerp{u_{n+k,p}:p\neq q}+F_{n+k,S+1}}\pi_{F_{n,S+1}//E_{n,S}}u_{n+k,j}} \\
    &\leq \abs{\pi_{E_{n+k,s(q)}//F_{n+k,s(q)+1}}\pi_{F_{n,S+1}//E_{n,S}}u_{n+k,j}}\\
    &\leq 4M_{S-1}C_S^2\norm{\pi_{F_{n,S+1}//E_{n,S}}}\norm{\pi_{E_{n+k,s(q)}//F_{n+k,s(q)+1}}|_{F_{n,s(j)}}} \\
    &\lesssim \exp(-n(\lambda_{s(q)}-\lambda_{s(j)}-5(s(j)-s(q))\delta)).
\end{align*}
Denote the largest integer $q < M_{S}'$ such that $s(q) \leq s(i)$ by $q_*$. If, moreover, $q\leq q_*$,
\[\abs{[A^{-1}]_{iq}[D]_{qj}}\lesssim \abs{[D]_{qj}}\lesssim\exp(-n(\lambda_{s(i)}-\lambda_{s(j)}-5(S-1)\delta)).\]
Otherwise, for $q_*<q<q^*$ we have
\[\abs{[A^{-1}]_{iq}}\lesssim \exp(-n(\lambda_{s(i)}-\lambda_{s(q)}-5(S-1)(s(q)-s(i))\delta)),\]
and so
\[\abs{[A^{-1}]_{iq}[D]_{qj}}\lesssim \exp(-n(\lambda_{s(i)}-\lambda_{s(j)}-5(S-1)(s(j)-s(i))\delta)).\]
Combining these three cases, we obtain
\[\abs{[A']_{ij}} \leq \sum_{q=1}^{M_{S}'} \abs{[A^{-1}]_{iq}[D]_{qj}}\lesssim \exp(-n(\lambda_{s(i)}-\lambda_{s(j)}-5(S-1)(s(j)-s(i))\delta)).\]
Hence, we have shown that there is a constant $C'' = C''(\alpha_S, \delta, H, B, \OSE_S)\geq 1$, independent of $f$, such that if $s(j) > s(i)$, then
\begin{equation}\label{eq:basisProjections}
    \abs{\pi_{\innerp{u_{n,i}}//\innerp{u_{n,q}:q\neq i}+F_{n,S}} u_{n+k,j}}_{E_{n,S}} \leq C''\exp(-n(\lambda_{s(i)}-\lambda_{s(j)}-5(S-1)(s(j)-s(i))\delta))
\end{equation}
for all $n$ large enough (depending on the same data as $C''$) and $k\geq 1$. 

\begin{proof}[Proof of \autoref{claim:projOfSingVecs}]
    Since $E_{n,t'} = \innerp{u_{n,i}:s(i) = s'}$ and $U_{n+k,s} = \innerp{u_{n+k,j}:s(j) = s}$, we easily see that \eqref{eq:basisProjections} implies \eqref{eq:projOfSingVecs} for $n$ large enough with the constant $C = 4m_{t'}\sqrt{m'_s}M_{S-1} C_S^2C''$.
\end{proof}

\subsection{Proof of \autoref{thm:refined_sequentialMET_oneSided}}
We shall now prove \autoref{thm:refined_sequentialMET_oneSided}. Let $(A_n:\B_n \to \B_{n+1})_{n\geq 0}$ be a sequence of linear operators satisfying \autoref{assump:refined_oneidedSequences}. For $m\geq 0$, we write $A^n_m = A^n(m) = A_{m+n-1}\cdots A_{m}$. As before, we shall assume $l_{\alpha} > -\infty$ if $r < \infty$ and discuss necessary modifications in the final subsection. We note that \autoref{assump:refined_oneidedSequences} implies $\Hplus_q(m,\delta)$ is finite for all $q\in \N$, $m\in \Z_{\geq 0}$, and $\delta > 0$. Similarly, it implies that for every $1\leq q < 2 + M_r$ there is an $N^*_q\in \N$ such that $\Htplus_q(m,\delta,N)$ is finite for all $m\in\Z_{\geq 0}$, $\delta > 0$, and $N\geq N^*_q$. We can assume $N^*_1\leq N^*_2\leq \cdots$

Now, for the proof of \autoref{thm:refined_sequentialMET_oneSided}, we first note that the equivariance of the filtration follows immediately from the form given in \autoref{remark:limsupDefnOfSlowSpaces}. With this out of the way, we fix $m = 0$. 

\subsubsection*{Proof of Item (1)}
Consider now $0 < \eta \leq 1$ and $1\leq s < r+1$ and let $E\in \G\Ccal_s(\eta)$. Take $0<\delta < \frac{1}{9(s+1)^2+4}\min_{2\leq t \leq s+1} \lambda_{t-1}-\lambda_t$, $B \geq \Bplus(0,\delta)$, and
\[H\geq \max_{1\leq q\leq 1 + M_s} \Hplus_q(0,\delta) \vee \Htplus_{q}(0,\delta,N^*_{M_s}).\]
Write $E(n) = A_0^n E$ and $E_k(n) = A_n^k E(n) = E(n+k)$.

\begin{claim}\label{claim:forwardTemperedApproxSlowSpace}
    There are constants $C,C'\geq 1$ and integer $N\in \N$ such that for all $n,k\geq N\vee N^*_{M_s}$, there are complements $F_{k}(n)\in \G^{M_s}(\B_{n})$ and $F_{k}'(n)\in \G^{M_s}(\B_{n+k})$ to $E(n)$ and $E_k(n)$, respectively, such that
    \[\det(A^k_n|E(n))\geq C^{-1}e^{-(n+k)C'\delta} V_{M_s}(A^k_n) \text{ and } (E(n),E_k(n),F_k(n),F_{k}'(n))\in \SVS_{M_s}(A_n^k;Ce^{(n+k)C'\delta}).\]
    Here, $C$ and $N$ depend on $\eta, \delta, B, H,$ and $\OSE_{s+1}$, while $C'$ depends only on $M_s$.
\end{claim}
\begin{proof}
    Notice that
    \[\det(A^{n+k}|E) = \det(A_n^k|E(n)) \det(A^n | E)\leq \det(A_n^k | E(n)) V_{M_s}(A^n).\]
    The key determinant estimate in the proof of item (4) of \autoref{thm:sequentialMET_oneSided} gives us
    \[\det(A^{n+k}|E)\gtrsim \exp((n+k)(l_{M_s}-11s^2M_s\delta))\]
    for large enough $n$ depending on $\delta, B, H$ and $\OSE_{s+1}$. Since $V_{M_s}(A^n) \leq e^H\exp(n(l_{M_s}+\delta))$, it follows that
    \[\det(A^k_n|E(n))\gtrsim \exp(kl_{M_s}-12s^2M_s(n+k)\delta).\]
    Now, \autoref{assump:refined_oneidedSequences} allows us to say that $V_{M_s}(A^k_n) \leq e^H\exp(kl_{M_s}+(n+k)\delta)$ for $n,k\geq N^*_{M_s}$, and thus
    \[\det(A^k_n|E(n))\gtrsim \exp(-13s^2M_s(n+k)\delta) V_{M_s}(A^k_n)\]
    for $k\geq N^*_{M_s}$ and $n$ large enough. The rest then follows by \autoref{prop:SVD} with the form of $D$ in that proposition as noted in \autoref{remark:formOfSVDconsts}.
\end{proof}

We shall shrink $\delta$ at various points in the following argument. Now, let $C'\geq 1$ be a generic constant that depends only on $M_s$, and whose value may change from line to line, and let $C$ and $N$ be as in \autoref{claim:forwardTemperedApproxSlowSpace}, assuming additionally that $N\geq N_{1+M_s}^*$. The constant in $\lesssim$ shall depend on the parameters in \autoref{claim:forwardTemperedApproxSlowSpace}.

By \autoref{claim:forwardTemperedApproxSlowSpace}, \autoref{claim:equiv_of_vol_gelfand_pesin_funcs}, and \autoref{claim:oneStepGapBound}, we obtain
\[\Gap(F_{k}(n),F_{k+1}(n))\lesssim \exp(-k(\lambda_{s}-\lambda_{s+1})+C'(n+k)\delta)\]
for all $n,k\geq N$. Let $\gamma = \lambda_s-\lambda_{s+1}$ and shrink $\delta$ so that $2C'\delta < \gamma$, and thus $\Gap(F_{k}(n),F_{k+1}(n))\lesssim \exp(-k(\gamma-2C'\delta))$ for all $k\geq n$. Since $C'$ is a generic constant, we absorb the factor of $2$ into $C'$. Clearly, this decays exponentially, so, taking $n$ large enough so we get by \autoref{lem:equivOfGap_and_Hausdorff} and the triangle inequality that
\begin{equation}\label{eq:moving_approx_cauchy_estimate}
    d_H(F_{k}(n),F_{k+j}(n))\lesssim \exp(-k(\gamma-C'\delta))
\end{equation}
for all $k\geq n$ and $j\geq 0$. By \autoref{cor:uniquenessOfSlowSpaces}, one can see $F_{k+j}(n)\to F_{s+1}(n)$ in the Grassmannian topology as $j\to\infty$, and so $d_H(F_{k}(n),F_{s+1}(n))$ also satisfies the same upper bound. Hence, by \autoref{lem:continuityOfDistToSpaces} and \autoref{claim:forwardTemperedApproxSlowSpace}, if we fix $k = n$ then we obtain
\begin{align}
    \sin\theta(E(n),F_{s+1}(n)) &\geq \sin\theta(E(n),F_{n}(n))-4d_H(F_{n}(n),F_{s+1}(n))\notag\\ 
    &\geq C^{-1}\exp(-C'\delta n) - C\exp(-n(\gamma-C'\delta)) \eqqcolon \eta_{n}\label{eq:effectiveForwardAngleTempering}
\end{align}
for all $n \geq N$. Shrinking $\delta$ so that $\gamma-C'\delta > C'\delta$, we get
\[\liminf_{n\to\infty}\frac{1}{n} \log\sin\theta(E(n),F_{s+1}(n))\geq -C'\delta.\] 
Since $\delta>0$ is arbitrarily small, item (1) follows.

\subsubsection*{Proof of Item (2)}
Item (2) follows immediately via the proof given in \autoref{lem:dual_SVD}. Let $1\leq s < r+1$. Fix any $0 < \eta \leq 1$ such that $\G\Ccal_s(\eta) \neq \varnothing$, and fix any $E\in \G\Ccal_{s}(\eta)$. With these choices, let all notation be as in the proof of item (1), with $\delta > 0$ small enough so that the above argument holds (including $\gamma - C'\delta > C'\delta)$.

For $m\geq 0$, write $\Ehat_{s}(m) = \ann{F_{s+1}(m)}$ and $\Fhat(m) = \ann{E(m)}$. Let $(u_{n,q})_{1\leq q\leq M_s} \text{ and } (v_{n,q})_{1\leq q\leq M_s}$ be the left and right singular vectors, respectively, of $A_0^n|_{E}:E\to E(n)$, and define $f_{n,q}(u) = \paren{\pi_{E(n)//F_{s+1}(n)} u, u_{n,q}}$ and $g_{n,q}(v) = \paren{\pi_{E(0)//F_{s+1}(0)} v, v_{n,q}}$ for $u\in \B_n$ and $v\in \B_0$. Clearly, $f_{n,q}\in \Ehat_{s}(n)$ and $g_{n,q}\in \Ehat_s(0)$, and writing $\sigma_{n,q} = \sigma_q(A^n_0|E)$, we get that $(A^n_0)^*f_{n,q} = \sigma_{n,q} g_{n,q}$. Now, the proof of \autoref{lem:dual_SVD} applies verbatim, except that the projection norm at time $n$ is no longer uniformly bounded through $n$. Item (1) instead gives that 
\[\norm{\pi_{E(n)//F_{s+1}(n)}} \leq \eta_n^{-1}\lesssim e^{C'\delta n}\]
for all $n$ large enough, and the proof of \autoref{thm:sequentialMET_oneSided}(4) gives us
\[\det(A^n_0|E(0))\gtrsim e^{-C'\delta n}V_{M_s}(A^n_0).\]
Therefore, we will obtain by the exact same pseudo-singular vector calculations in \autoref{lem:dual_SVD} that
\[\det((A^n_0)^*|\Ehat_s(n))\gtrsim e^{-C'\delta n} V_{M_s}((A^n_0)^*)\]
and
\[m((A^n_0)^*|_{\Ehat_s(n)})\gtrsim e^{-C'\delta n} c_{M_s}((A^n_0)^*)\gtrsim e^{-C'\delta n} c_{M_s}(A^n_0).\]
The matching upper bound for the determinant is given by $V_q((A^n_0)^*)\approx V_q(A^n_0)$. Similarly, $m((A^n_0)^*|_{\Ehat_s(n)})\leq \sqrt{M_s}\norm{(A^n_0)^* f_{n,M_s}}\lesssim \sigma_{n,M_s}\lesssim c_{M_s}(A^n_0)$. Thus, since $\delta>0$ was arbitrarily small, we obtain item (2).

\subsection*{The case $l_\alpha = -\infty$ when $r < \infty$}
We shall consider the modifications necessary for the proofs of both \autoref{thm:sequentialMET_oneSided} and \autoref{thm:refined_sequentialMET_oneSided} in the case that $r < \infty$ and $l_\alpha = -\infty$. Fix such $r<\infty$ and let $q_\alpha = 1+M_r$. All arguments for both theorems involving only the finite exceptional Lyapunov exponents $\lambda_1,\ldots,\lambda_r > -\infty$ remain unchanged, and the only necessary modifications occur when $\lambda_{r+1} = l_\alpha$ was used to bound growth on $F_{\alpha}$. To do this, one fixes an arbitrary $\mu < \lambda_{r}$ and writes $\Delta = -l_{M_r} - \mu$. We then define the $H$-type quantities
\begin{itemize}
    \item $\Hplus_\alpha(m,\Delta) = \sup_{n\geq 0} \log V_{q_\alpha}(A_m^n) + n\Delta$;
    \item $\Htplus_\alpha(m,\Delta,\delta,N) = \sup_{n,k\geq N} \log V_{q_\alpha}(A_{m+n}^k) + k\Delta - (n+k)\delta$.
\end{itemize}

\noindent \textbf{\autoref{thm:sequentialMET_oneSided}:} If $H\geq \Hplus_\alpha(0,\Delta), \Hplus_{M_r}(0,\delta)$, then $c_{q_\alpha}(A^n)\lesssim_{M_r} e^{2H} \exp(n(\mu + \delta))$ for all $n\geq 1$. Therefore, the proofs of \autoref{thm:sequentialMET_oneSided} and the associated lemmas go through replacing every occurrence of $\Hplus_{q_\alpha}(0,\delta)$ with $\Hplus_\alpha(0,\Delta)$, $l_{q_\alpha}$ with $-\Delta$, and $\lambda_{r+1}$ with $\mu$. Since $\mu$ was arbitrarily negative, for item (1) one should find that the approximate $F_{n,\alpha}$ obtained by \autoref{prop:SVD} converge to $F_\alpha$ superexponentially. Moreover, item (2) follows through with the $\limsup$ being $-\infty$. For items (3) and (4), the specific choice of $\mu$ does not change the end result, we just need to fix $\mu$ sufficiently negative.

\noindent \textbf{\autoref{thm:refined_sequentialMET_oneSided}:} If $H\geq \Htplus_\alpha(m,\Delta,\delta,N), \Htplus_{M_r}(m,\delta,N)$, then $c_{q_\alpha}(A^k_{m+n})\lesssim_{M_r} e^{2H} \exp(k\mu + 2(n+k)\delta))$ for all $n,k\geq N$. We of course must take $N$ large enough so that $\Htplus_\alpha$ is finite (which exists by item (iii) of \autoref{assump:refined_oneidedSequences}). Then, as before, the proofs of \autoref{thm:refined_sequentialMET_oneSided} and the associated sub-results go through replacing every occurrence of $\Htplus_{q_\alpha}(m,\delta,N)$ with $\Htplus_\alpha(m,\Delta,\delta,N)$, $l_{q_\alpha}$ with $-\Delta$, and $\lambda_{r+1}$ with $\mu$. As with items (3)-(4) of \autoref{thm:sequentialMET_oneSided}, we need only fix $\mu$ sufficiently negative.

\section{Splittings for Two-Sided Sequences}\label{sec:sequentialMET_twoSided}
In this section, we will be considering two-sided sequences of bounded linear operators, giving us information on history along with the future, which will enable us to construct equivariant Oseledets splittings. Before we start, we note that the ``midpoint construction" for approximate fast spaces in \autoref{subsubsec:items(1)(2)_twoSided} originates from Lee \cite{Lee2024}.

\subsection{Setting and Results}
Let $(\B_m)_{m\in\Z}$ be Banach spaces and $A_{m}:\B_m\to \B_{m+1}$ a sequence of bounded linear operators. For $m\in \Z$ and $n\geq 1$, define $A^n_m = A^n(m) = A_{m+n-1}\cdots A_{m}$ and $A^{*n}_m = (A^n_{m-n})^* = A^*_{m-n}\cdots A^*_{m-1}:\B_{m}^*\to\B_{m-n}^*$

\begin{assumption}\label{assump:twoSidedSequences}
    For $m\in \Z$ and $n\in \N$, write $A^n_m = A^n(m) = A_{m+n-1}\cdots A_{m}$.
    \begin{enumerate}[label = (\roman*)]
        \item $\displaystyle \limsup_{n\to\infty} \frac{1}{n}\log \norm{A_n}_{\B_n\to \B_{n+1}},\limsup_{n\to\infty} \frac{1}{n}\log \norm{A_{-n}}_{\B_{-n}\to \B_{-n+1}}\leq 0$; and
        \item for all $m\in \Z$ and $q\in \N$, we have 
        \[\infty>l_q = \lim_{n\to\infty} \frac{1}{n} \log V_q(A^n(m)) =\lim_{n\to\infty} \frac{1}{n} \log V_q(A^n(m-n)) = \lim_{n\to\infty}\frac{1}{2n} \log V_q(A^{2n}(m-n)).\]
    \end{enumerate}
\end{assumption}
Notice that \autoref{assump:twoSidedSequences} implies that the sequence $(A_{m+n}:\B_{m+n}\to\B_{m+n+1})_{n\geq 0}$ satisfies \autoref{assump:oneSidedSequences} for all $m\in \Z$, and that the Lyapunov exponents and their multiplicities are constant w.r.t. $m$. For $m\in\Z$, we shall let 
\[\B_m = F_1(m)\supsetneq F_2(m)\supsetneq\cdots\supsetneq F_\alpha(m)\supseteq\innerp{0}\]
denote the filtration of slow spaces for the sequence $(A_{m+n})_{n\geq 0}$.
\begin{thm}\label{thm:sequentialMET_twoSided}
    Let $(A_n)_{n\in\Z}$ be a sequence of operators subject to \autoref{assump:twoSidedSequences} and take the notations established in \autoref{thm:sequentialMET_oneSided}. Then the sequence of operators $(A_n^*)_{n\in\Z}$ satisfies \autoref{assump:twoSidedSequences} and its Lyapunov spectrum coincides with the Lyapunov spectrum of $(A_n)_{n\in\Z}$. For $m\in\Z$, we shall let 
    \[\B^*_m = \Fhat_1(m)\supsetneq \Fhat_2(m)\supsetneq\cdots\supsetneq \Fhat_\alpha(m)\supseteq\innerp{0}\]
    denote the filtration of slow spaces for the sequence $(A_{m-n}^*)_{n\geq 1}$.

    If $\chi_1 > l_\alpha$, then there is for all $m\in \Z$ and $1\leq s < r+1$ an equivariant splitting of $\B_m$ into closed subspaces
    \[\B_m = F_{s+1}(m) \oplus \bigoplus_{t=1}^s U_{t}(m)\]
    (where $F_{r+1} = F_\alpha$ if $r<\infty$) such that, if $1\leq t < r+1$, then
    \begin{equation}\label{eq:splittingForwardGrowth}
        \lim_{n\to\infty} \frac{1}{n}\log\norm{A^n(m)|_{U_t(m)}} = \lim_{n\to\infty} \frac{1}{n}\log m\paren{A^n(m)|_{U_t(m)}} = \lambda_t.
    \end{equation}
    The following also hold for all $m\in \Z$ and $1\leq s < r + 1$.
    \begin{enumerate}[label = (\arabic*)]
        \item $\dim U_s(m) = m_s$ and $F_{s}(m) = U_s(m) \oplus F_{s+1}(m)$. If $r<\infty$, then the analogous statements hold writing $s = r$ and identifying $F_{r+1} = F_\alpha$.
        \item Writing $E_{s}(m)\coloneqq \preann{\Fhat_{s+1}(m)}$, we have $\dim E_s(m) = M_s$ and $E_s(m) = U_1(m)\oplus\cdots\oplus U_s(m)$. If $r<\infty$, then the analogous statement holds writing $s = r$ and identifying $\Fhat_{r+1} = \Fhat_\alpha$.
        \item $\displaystyle \lim_{n\to\infty}\frac{1}{n} \log \norm{A^n(m-n)|_{F_{s}(m-n)}} = \lambda_s$ and $\displaystyle \limsup_{n\to\infty}\frac{1}{n} \log \norm{A^n(m-n)|_{F_{\alpha}(m-n)}}\leq l_\alpha$. If $r<\infty$, then the $\limsup$ is a convergent limit equal to $l_\alpha$.
    \end{enumerate}
\end{thm}
\begin{remark}
    Note that if we weaken \autoref{assump:twoSidedSequences} so that it only holds for $m = 0$, then items (1) and (2) and equation \eqref{eq:splittingForwardGrowth} continue to hold at $m = 0$.
\end{remark}

\begin{assumption}\label{assump:refined_twoSidedSequences}
    In addition to \autoref{assump:twoSidedSequences}, we have the following.
    \begin{enumerate}[label = (\roman*)]
        \item For $q\in \N$ such that $l_q > -\infty$, we have
        \[\lim_{n,k\to\infty}\frac{1}{n+k} \abs{\log V_{q}\paren{A^k_{\pm n}}-kl_{q}} = 0.\]
        \item If $q\in \N$ is the smallest integer such that $l_q = -\infty$, then for all $\Delta \in \R$,
        \[\limsup_{n,k\to\infty} \frac{1}{n+k}\paren{\log V_q(A^k_{\pm n}) + k\Delta} \leq 0.\]
    \end{enumerate}
\end{assumption}

\begin{thm}\label{thm:refined_sequentialMET_twoSided}
    Let $(A_{n})_{n\in\Z}$ be a sequence of operators subject to \autoref{assump:refined_twoSidedSequences} and take the notations established in \autoref{thm:sequentialMET_twoSided}. Then, for all $1\leq s < r+1$ and $m\in \Z$, we have the following.
    \begin{enumerate}[label = (\arabic*)]
        \item $\displaystyle \lim_{n\to\infty} \frac{1}{n} \log m\paren{A^n(m-n)|_{U_s(m-n)}} = \lim_{n\to\infty} \frac{1}{n} \log m\paren{A^n(m-n)|_{E_s(m-n)}} = \lambda_{s}$.
        \item $\displaystyle \lim_{n\to\infty} \frac{1}{n}\log\det\paren{A^n(m-n)|E_s(m-n)} = \sum_{t=1}^s m_{t}\lambda_t$.
        \item $\displaystyle \lim_{n\to \pm\infty}\frac{1}{n}\log \sin\theta\paren{E_s(n),F_{s+1}(n)} = 0$. If $r<\infty$, then the analogous statement holds writing $s = r$ and identifying $F_{r+1} = F_{\alpha}$.
    \end{enumerate}
    Moreover, for each $1\leq s < r+1$, $m\mapsto U_s(m)$ is uniquely characterized by equivariance, its dimension, and the growth identity
    \begin{equation}\label{eq:uniqueSplittingGrowthCharacterization}
        \lim_{n\to\infty} \frac{1}{n} \log m\paren{A^n(m-n)|_{U_s(m-n)}} = \lim_{n\to\infty} \frac{1}{n} \log \norm{A^n(m)|_{U_s(m)}} = \lambda_s.
    \end{equation}
\end{thm}

\subsection{Proof of \autoref{thm:sequentialMET_twoSided}}
We shall prove \autoref{thm:sequentialMET_twoSided} in this section. Firstly, the fact that $(A^*_n)_{n\in \Z}$ satisfies assumption \autoref{assump:twoSidedSequences} and that its Lyapunov spectrum coincides with that of $(A_n)_{n\in\Z}$ follows immediately from \autoref{cor:equiv_volGrowth_and_dual}. We now move on to proving the rest of the theorem.

Now, we shall first assume $l_\alpha > -\infty$ in the case that $r<\infty$ and describe the necessary modifications for $l_\alpha = -\infty$ at the end of this subsection. It will be helpful to write $\lambda_{r+1} \coloneqq l_\alpha$ in this case. There is a natural extension of the quantities $\Bplus(m,\delta)$ and $\Hplus_q(m,\delta)$ from \autoref{sec:sequentialMET_oneSided} for $m\leq -1$. We shall also define the analogous backward-time and centered quantities:
\begin{itemize}
    \item $\Bminus(m,\delta) = \sup_{n \geq 1} \log \norm{A_{m-n}} - n \delta$
    \item $\Hminus_q(m,\delta) = \sup_{n \geq 1} \abs{\log V_q(A^n_{m-n}) - nl_q} - n\delta$
    \item $\Hpm_q(m,\delta) = \sup_{n\geq 1} \abs{\log V_q(A^{2n}_{m-n}) - 2nl_q} - 2n\delta$
\end{itemize}

\subsubsection*{Proof of Items (1) and (2)}\label{subsubsec:items(1)(2)_twoSided}
For $1\leq s < r + 1$ and $m\in \Z$, let $W_{n,s}(m)\subset \B_m$ be an $M_s$-dimensional subspace such that
\begin{equation}\label{eq:defOfWSpaces}
    \det\paren{A_m^n|W_{n,s}(m)}\geq \frac{1}{2} V_{M_s}(A^n_m)
\end{equation}
and write $E_{n,s}(m) = A^{n}_{m-n} W_{2n,s}(m-n)$. Fix $m\in \Z$ and let
\begin{equation}\label{eq:twoSidedEffectivenessConsts}
    B\geq \Bminus(m,\delta),\Bplus(m,\delta) \text{ and } H\geq \max_{-1\leq j\leq 1}\Hpm_{j+M_s}(m,\delta),\Hminus_{j+M_s}(m,\delta),\Hplus_{j+M_s}(m,\delta).
\end{equation}
Then 
\[\det(A^n_m|E_{n,s}(m))\det(A^n_{m-n}|W_{2n,s}(m-n))\geq \frac{1}{2} V_{M_s}(A^{2n}_{m-n})\geq \frac{1}{2e^H}\exp(2n(l_{M_s}-\delta))\]
for all $n\geq 1$. In particular, this means
\begin{equation}\label{eq:backwardApproxDetBound}
    \det(A^n_{m-n}|W_{2n,s}(m-n))\geq \frac{\exp(2n(l_{M_s}-\delta))}{2e^H V_{M_s}(A^n_m)}\geq \frac{\exp(n(l_{M_s}-3\delta))}{2e^{2H}} \geq \frac{V_{M_s}(A^n_{m-n})}{2e^{3H+4\delta n}}
\end{equation}
and
\begin{equation}\label{eq:forwardApproxDetBound}
    \det(A^n_{m}|E_{n,s}(m))\geq \frac{\exp(2n(l_{M_s}-\delta))}{2e^H V_{M_s}(A^n_{m-n})}\geq \frac{\exp(n(l_{M_s}-3\delta))}{2e^{2H}} \geq \frac{V_{M_s}(A^n_m)}{2e^{3H+4\delta n}}.
\end{equation}
For the sake of brevity, we shall now fix $m = 0$ and suppress parentheses in front of the spaces (i.e. $W_{2n,s}$ instead of $W_{2n,s}(-n)$ and $E_{n,s}$ instead of $E_{n,s}(0)$) unless necessary, as the following argument requires a trivial adaptation for general $m\in \Z$.

\vspace{2mm}
\noindent\textit{Connection to the Adjoint Sequence} Let $G:\N\to\N$ be as in \autoref{remark:formOfSVDconsts} and write $D_s = G(M_s)(2e^{3H})^{G(M_s)}$. Then, by \eqref{eq:backwardApproxDetBound} and \autoref{lem:dual_SVD}, there are complements $V_{n,s+1}$ and $V_{n,s+1}'$ of $W_{2n,s}$ and $E_{n,s}$, respectively, such that 
\[(L_{n,s},\Lhat_{n,s},K_{n,s+1},\Khat_{n,s+1})\in \SVS_{M_s}(D_se^{4G(M_s)\delta n};A_0^{*n}),\]
where $L_{n,s} = \ann{V_{n,s+1}'}$, $\Lhat_{n,s} = V_{n,s+1}$, $K_{n,s+1} = \ann{E_{n,s}}$, and $\Khat_{n,s+1} = \ann{W_{2n,s}}$. Shrinking $\delta$ so that $5G(M_s)\delta < \lambda_s-\lambda_{s+1}$, it follows by \autoref{cor:uniquenessOfSlowSpaces} (applied to the adjoint sequence) that
\begin{equation}\label{eq:fastSpaceDistBound}
    d_{H}(K_{n,s+1},\Fhat_{s+1}(0)) \lesssim \exp(-n(\lambda_{s}-\lambda_{s+1}-5G(M_s)\delta))
\end{equation}
for all $n$ large enough. In particular, $K_{n,s+1}\to \Fhat_{s+1}(0)$ as $n\to\infty$ in the Grassmannian topology on $\B_0^*$. Hence, by \autoref{lem:equivOfGap_and_Hausdorff} it follows that $E_{n,s} = \preann{K_{n,s+1}}\to \preann{\Fhat_{s+1}(0)} = E_{s}(0)$.

\vspace{2mm}
\noindent\textit{Growth and Angle Estimates} Write $E'_{n,s} = A^n_0 E_{n,s}$. By \eqref{eq:forwardApproxDetBound} and \autoref{prop:SVD}, there are topological complements $F_{n,s+1}$ and $F'_{n,s+1}$ to $E_{n,s}$ and $E_{n,s}'$, respectively, such that
\[(E_{n,s},E_{n,s}',F_{n,s+1},F_{n,s+1}')\in \SVS_{M_s}(D_s e^{4G(M_s)\delta n};A^n_0)\]
for all $n\geq 1$. As before, by \autoref{cor:uniquenessOfSlowSpaces},
\[d_H(F_{n,s+1},F_{s+1}(0))\lesssim \exp(-n(\lambda_{s}-\lambda_{s+1}-5G(M_s)\delta))\]
for all $n$ large enough. Combining this and \eqref{eq:fastSpaceDistBound}, we get
\[\abs{\sin\theta(E_s(0),F_{s+1}(0))-\sin\theta(E_{n,s},F_{n,s+1})}\lesssim \exp(-n(\lambda_{s}-\lambda_{s+1}-5G(M_s)\delta))\]
for all $n$ large enough. Letting $C$ be the constant in the above inequality, this means
\begin{equation}\label{eq:equivariantSlowFastSplittingAngles}
    \sin\theta(E_s(0),F_{s+1}(0))\geq \sup_{n\geq N} \bracket{D_s^{-1}e^{-4G(M_s)\delta n} - C\exp(-n(\lambda_{s}-\lambda_{s+1}-5G(M_s)\delta))}.
\end{equation}
If we shrink $\delta>0$ so that $10 G(M_s)\delta < \lambda_{s}-\lambda_{s+1}$, then the above supremum is strictly positive. This proves $E_{s}(0)$ is a topological complement to $F_{s+1}(0)$. We may then define $U_s(0) = E_s(0)\cap F_s(0)$, giving us items (1)-(2). Moreover, using this effective angle estimate, we can apply \autoref{thm:sequentialMET_oneSided} (and the proof therein) to get effective volume growth estimates and derive \eqref{eq:splittingForwardGrowth}.

\vspace{2mm}
\noindent\textit{Equivariance}
We have proven that, for all $m\in \Z$ and $1\leq s < r+1$, $\Fhat_{s+1}(m) = \ann{E_{s}(m)}$. Now, as in \autoref{remark:limsupDefnOfSlowSpaces}, we know that $A^*_{m}\Fhat_{s+1}(m+1) \subseteq \Fhat_{s+1}(m)$. Applying the pre-annihilator to both sides, we get $(A^m)^{-1}E_{s}(m+1) = \preann{\paren{A^*_{m}\Fhat_{s+1}(m+1)}} \supseteq \preann{\Fhat_{s+1}(m)} = E_{s}(m)$, where $(A^m)^{-1}$ denotes the pre-image by $v\mapsto A_{m}v$. Of course, this means $A_{m}E_s(m) \subseteq E_s(m+1)$. But, as $A_{m}$ is injective on $E_s(m)$ and $\dim E_s(m) = \dim E_s(m+1)$, it follows that $A_m E_{s}(m) = E_{s}(m+1)$. Moreover, since $U_{s}(m) = E_{s}(m)\cap F_{s}(m)$, equivariance follows immediately.

\subsubsection*{Proof of Item (3)}
We will show that item (3) is essentially a corollary of items (1)-(2) and \autoref{thm:sequentialMET_oneSided}, so we will not spend much time explicitly writing effective estimates.

\begin{remark}
    One can make the following argument effective by applying \autoref{cor:equiv_volGrowth_and_dual} to $B$ and $H$ as in \eqref{eq:twoSidedEffectivenessConsts} (decreasing $s$ by 1) and
    \[H\geq \max_{1\leq q \leq M_{s-1}+1} \Hminus_q(m,\delta)\]
    for $\delta>0$ sufficiently small.
\end{remark}

Running the proofs of items (1)-(2) again for the dual sequence $(A^*_{m})_{m\in \Z}$, we get for each $m\in \Z$ and $1\leq s < r+1$ the equivariant splitting $\B_m^* = \Ehat_s(m) \oplus \Fhat_{s+1}(m)$. Here, equivariance is shifted by one and goes backwards in time (i.e. $A_{m-1}^*\Ehat_s(m) = \Ehat_s(m-1)$). If $F^{**}_{s+1}(m)$ is the $\lambda_{s+1}$-slow space for the bi-adjoint sequence $(A^{**}_{m},A^{**}_{m+1},\ldots)$, then $F_{s+1}(m) = \B\cap F_{s+1}^{**}(m)$, by \autoref{remark:limsupDefnOfSlowSpaces}, under the identification via the isometric pointwise evaluation map. Then item (2) gives $F_{s+1}^{**}(m)$ is the $\B^{**}$-annihilator of $\Ehat_{s}(m)$, which means $F_{s+1}(m) = \preann{\Ehat_s(m)}$ since $\codim F_{s+1}(m) = \dim \Ehat_{s}(m)$. In other words, we get $\Ehat_s(m) = \ann{F_{s+1}(m)}$, and we may write $\Uhat_s(m) = \Ehat_s(m)\cap \Fhat_s(m)$.

As before, $m = 0$ and $1\leq s < r+1$. Item (3) follows immediately for $s = 1$ since $F_1(-n) = \B_{-n}$ for all $n\geq 0$, so assume $s\geq 2$. Then, by \autoref{lem:distance_vs_orthogonalProj}, we get that
\[\norm{A^n_{-n}|_{F_{s}(-n)}} = \sup_{\ell\in B_{\B_0^*}} d(A^{*n}_0\ell,\Ehat_{s-1}(-n)).\]
Now, for all $\ell\in \B_0^*$,
\[d(A^{*n}_0\ell,\Ehat_{s-1}(-n)) = d(A^{*n}_0\pi_{\Fhat_s(0)//\Ehat_{s-1}(0)}\ell,\Ehat_{s-1}(-n))\leq \norm{A^{*n}_0|_{\Fhat_{s}(0)}}\norm{\pi_{\Fhat_s(0)//\Ehat_{s-1}(0)}}\norm{\ell},\]
which gives us
\[\limsup_{n\to\infty} \frac{1}{n} \log \norm{A^n_{-n}|_{F_{s}(-n)}} \leq \lambda_s\]
by \autoref{thm:sequentialMET_oneSided}. Moreover, applying \eqref{eq:FinalAngleOfImagesEstimate} to the adjoint sequence, we see that if we take $\ell\in \Uhat_s(0)$ of unit norm, then
\[\liminf_{n\to\infty} \frac{1}{n} \log \norm{A^n_{-n}|_{F_{s}(-n)}}\geq\liminf_{n\to\infty} \frac{1}{n} \log d(A^{*n}_0\ell,\Ehat_{s-1}(-n))\geq \lambda_s.\]
\begin{remark}
    Notice $\sin\theta(\Ehat_{s-1}(0),\Fhat_{s}(0)) = \sin\theta(E_{s-1}(0),F_{s}(0))$, so angles can be estimated effectively from pre-adjoint angle estimates.
\end{remark}

Now, the above work provides $\limsup_{n\to\infty} \frac{1}{n} \log\norm{A^n_{-n} |_{F_\alpha(-n)}}\leq l_\alpha$ when $r < \infty$. For the equality clause, notice that since $r<\infty$, \autoref{assump:twoSidedSequences} implies
\[\liminf_{n\to\infty}\frac{1}{n} \log \norm{A^n_{-n}|_{F_\alpha(-n)}} \geq \lim_{n\to\infty} \frac{1}{n} \log c_{1+M_r}(A^n_{-n}) = l_\alpha.\]

Lastly, if $r = \infty$, then we get
\[\limsup_{n\to\infty} \frac{1}{n} \log\norm{A^n_{-n} |_{F_\alpha(-n)}}\leq \inf_{s\geq 1}  \lim_{n\to\infty} \frac{1}{n} \log\norm{A^n_{-n} |_{F_s(-n)}} = \inf_{s\geq 1} \lambda_s = l_\alpha.\]

\subsubsection*{The case $l_\alpha = -\infty$ when $r<\infty$}
Suppose now $r<\infty$ and $l_\alpha = -\infty$. The modifications are analogous to those described in \autoref{sec:sequentialMET_oneSided}. Fix a sufficiently negative $\mu < \lambda_{r}$ and write $\Delta = -l_{M_r} - \mu$. Writing $q_\alpha = 1 + M_r$, extend $\Hplus_\alpha(m,\Delta)$ to $m\in \Z$. We also analogously define $H^-_\alpha(m,\Delta) = \sup_{n\geq 1} \log V_{q_\alpha}(A^n_{m-n}) + \Delta n$ and
\[\Hpm_\alpha(m,\Delta) = \sup_{n\geq 1} \log V_{M_r+1}(A^{2n}_{m-n}) + 2n\Delta.\]
Then the same arguments as before proceed when $\lambda_{r+1}$ is replaced by $\mu$ and $l_{M_r+1}$ is replaced by $-\Delta$.

\subsection{Proof of \autoref{thm:refined_sequentialMET_twoSided}}
We now assume $(A_n:\B_n\to\B_{n+1})_{n\in\Z}$ satisfies \autoref{assump:refined_twoSidedSequences} and prove \autoref{thm:refined_sequentialMET_twoSided}. Here, we do not need to distinguish between the case $l_\alpha > -\infty$ and $l_\alpha = -\infty$. 

It will turn out that items (1)-(3) are immediate corollaries of \autoref{thm:sequentialMET_twoSided} and \autoref{thm:refined_sequentialMET_oneSided}. Indeed, by \autoref{cor:equiv_volGrowth_and_dual}, it is clear that $(A_n^*:\B_{n+1}^*\to\B_{n}^*)_{n\in\Z}$ satisfies \autoref{assump:refined_twoSidedSequences}, and so we may just apply \autoref{thm:refined_sequentialMET_oneSided} to the sequence $(A_{m-1}^*,A_{m-2}^*,A_{m-3}^*,\ldots)$ and get the desired results, noting that
\[\sin\theta(E_{s}(n), F_{s+1}(n)) = \sin\theta(\Ehat_{s}(n), \Fhat_{s+1}(n))\]
for all $n\in \Z$. We will not attempt to explicitly make effective estimates, as the form of such estimates is clear upon reading the proofs of \autoref{thm:sequentialMET_twoSided} and \autoref{thm:refined_sequentialMET_oneSided}.

Therefore, all that is left to show is uniqueness of the splitting. In particular, suppose we have another collection of equivariant spaces $\{m\mapsto U_{s}'(m)\}_{1\leq s < r+1}$ such that $U_s'$ satisfies \eqref{eq:uniqueSplittingGrowthCharacterization} for each $1\leq s < r+1$. Notice first that the second limit in the characterization given by \eqref{eq:uniqueSplittingGrowthCharacterization} implies $U_s'(m)\subseteq F_{s}(m)$, so if we can show that $U_s'(m) \subseteq E_{s}(m)$, we get $U'_s(m)\subseteq U_s(m) = E_s(m)\cap F_s(m)$. Then, the fact that $\dim U_{s}'(m) = m_s$ forces $U_s'(m) = U_s(m)$.

Notice first that \eqref{eq:uniqueSplittingGrowthCharacterization} implies $A_m$ is injective on $U_s'(m)$, so equivariance and dimensionality imply $A_m U_s'(m) = U_s'(m+1)$. Let $\delta > 0$ and fix $m\in \Z$. By the first limit in \eqref{eq:uniqueSplittingGrowthCharacterization}, it follows that there is an integer $N = N(m,\delta)$ such that $m(A^n_{m-n}|_{U_s'(m-n)}) \geq \exp(n(\lambda_s-\delta))$ for all $n\geq N$. Take $u_0\in U_s'(m)$ of norm $1$ and let $u_{-n} = (A^n_{m-n}|_{U_s'(m-n)})^{-1} u_0$. Then, by the lower bound on the minimum norm,
\[\norm{u_{-n}}\leq \exp(-n(\lambda_{s}-\delta))\]
for all $n\geq N$. But then
\begin{align*}
    \norm{\pi_{F_{s+1}(m)//E_{s}(m)}u_0} &= \norm{A^n_{m-n}\pi_{F_{s+1}(m-n)//E_{s}(m-n)} u_{-n}}\\
    &\leq \norm{A^n_{m-n}|_{F_{s+1}(m-n)}}\norm{\pi_{F_{s+1}(m-n)//E_s(m-n)}} \exp(-n(\lambda_s-\delta)).
\end{align*}
Take any $\mu \in (\lambda_{s+1}, \lambda_s)$ (where $\lambda_{r+1} \coloneqq l_\alpha$ even if $l_\alpha = -\infty$). By \autoref{thm:sequentialMET_twoSided} item (3), enlarging $N$ if necessary, we get
\[\norm{A^n_{m-n}|_{F_{s+1}(m-n)}}\leq \exp(n\mu)\]
for all $n\geq N$. In particular, taking a supremum over all $v\in U_s'(m)$ of norm one, we get
\[\Gap(U_s'(m),E_{s}(m))\leq \norm{\pi_{F_{s+1}(m-n)//E_{s}(m-n)}}\exp(-n(\lambda_{s}-\mu-\delta)).\]
Since $n\mapsto\norm{\pi_{F_{s+1}(m-n)//E_{s}(m-n)}}$ is subexponential, if we take $\delta < \lambda_s-\mu$, it follows that 
\[\lim_{n\to\infty}\norm{\pi_{F_{s+1}(m-n)//E_{s}(m-n)}}\exp(-n(\lambda_{s}-\mu-\delta)) = 0.\] 
In particular, $\Gap(U_s'(m),E_{s}(m)) = 0$, which implies $U_s'(m)\subseteq E_s(m)$. Hence, $U_s'(m) = U_s(m)$ for all $m\in \Z$ and $1\leq s < r+1$.

\section{The Multiplicative Ergodic Theorems}\label{sec:MET}
There are three measurability issues we must contend with if we want to formulate a proper ergodic theorem suitable for most applications. These are the existence of measurable volume functions, measurability of the filtration, and (in the semi-invertible case) measurability of the splitting and associated projection operators. Let $(\Omega,\F,\p)$ be a probability space with mpt $\sigma:\Omega\to\Omega$ and let $(\B,\norm{\cdot})$ be a Banach space. Let $A:\Omega\mapsto L(\B)$ be an operator-valued map and write $A^n_\omega = A_{\sigma^{n-1}\omega}\cdots A_\omega$. 

If for every $q\in \N$ we can find a map $G_q:L(\B)\to [0,\infty)$ such that $G_q(A)\approx_q  V_q(A)$ for all $A\in L(\B)$ and $\omega\mapsto G_q(A^n_\omega)$ is measurable for all $n\in \N$, then so long as $\log^+ G_1(A_\omega)\in L^1$, Kingman's subadditive ergodic theorem tells us that the cocycle $\omega\mapsto A_\omega$ $\p$-generically satisfies sufficient conditions for the multiplicative growth theorems we developed in the preceding sections to hold. 

However, for many applications, it is not just enough to get a measurable Lyapunov spectrum and existence of a filtration/splitting; one often needs the Oseledets subspaces to satisfy measurability conditions themselves. This is another one of the advantages of our approach; we just need to feed minimal topological information to build measurable approximants to the filtration and fast spaces.

\subsection{Setting and Results}
Our setting will be applicable to Banach spaces $\B$ for which the unit ball $B_{\B^*}$ in the continuous dual of $\B$ is weak-* separable. With this first topological assumption, we formulate the following measurability condition.

\begin{defn}
    Let $S\subseteq \B$ and $S^*\subseteq \B^*$ be arbitrary subsets. An operator-valued map $A:\Omega\to L(\B)$ is said to be $(S,S^*)$\textit{-measurable} if the map
    \[\omega\mapsto \innerp{A_\omega v,\ell}\]
    is Borel measurable for all $v\in S$ and $\ell\in S^*$.
\end{defn}

\begin{measAssumption*}[M]\label{ass:measurability}
    There is a vector space topology $\tau$ on $\B$, coarser than the norm topology, such that
    \begin{enumerate}[label = (\roman*)]
        \item the norm-closed unit ball $B_\B$ is $\tau$-closed and $\tau$-separable;
        \item for almost every $\omega\in \Omega$, $A_\omega|_{B_{\B}}:(B_\B,\tau|_{B_\B})\to(\B,\tau)$ is $\tau$-continuous; and
        \item there is a countable weak-* dense set $D\subset B_{\B^*}$ such that
        \begin{enumerate}
            \item the subspace $\tau$-topology $\tau|_{B_{\B}}$ is finer than $\sigma(B_{\B},D)$; and
            \item for all $n\geq 1$, $A^n:\Omega\to L(\B)$ is $(\B,D)$-measurable.
        \end{enumerate}
    \end{enumerate}
    If any such $\tau$ and $D$ exist, we say $A$ satisfies condition $\measurable_\sigma$; we also say $(\tau;D)\in \measurable_{\sigma,A}$.
\end{measAssumption*}
\begin{remark}\label{remark:lowersemicontinuityOfNorm}
    The unit ball $B_\B$ is $\tau$-closed if and only if the norm $\norm{\cdot}$ is lower semi-continuous w.r.t. $\tau$.
\end{remark}
Item (iii)(a) is equivalent to the following statement: suppose $(x_\alpha)_{\alpha}$ is a norm-bounded net of vectors in $\B$ such that $x_\alpha$ converges in $\tau$ to some $x\in \B$. Then, for all $\ell\in D$, we have that
\[\ell(x) = \lim_{\alpha} \ell(x_\alpha).\]
Notice that a similar interpretation is true for item (ii), as it is equivalent to the statement that $Ax_\alpha\to Ax$ in $\tau$ for norm bounded nets such that $x_\alpha\to x$ in $\tau$. Thus, it will be a useful bookkeeping tool to define the set
\[BL_\tau(V) \coloneqq \curly{T\in L(\B,V) \text{ such that } T|_{B_{\B}}:(B_{\B},\tau|_{B_\B})\to V \text{ is continuous}},\]
for a topological vector space $V$. Then item (iii)(a) says $D\subset BL_{\tau}(\R,\abs{\cdot})$ and item (ii) says $A_{\omega}\in BL_{\tau}(\B,\tau)$ a.s.
\begin{remark}
    Notice that if $A\in L(\B,\norm{\cdot})\cap BL_\tau(\B,\tau)$ and $\ell\in BL_\tau(\R)$, then $A^*\ell\in BL_\tau(\R)$. Moreover, by \autoref{appLem:closednessOfBLtau}, $BL_\tau(\R)$ is dual norm closed in $\B^*$, and $L(\B)\cap BL_\tau(\B,\tau)$ is operator norm closed in $L(\B)$.
\end{remark}

Using our newly formulated measurability condition in conjunction with \autoref{thm:refined_sequentialMET_oneSided}, we obtain a multiplicative ergodic theorem for non-invertible Banach space cocycles. By \autoref{prop:norm_and_volume_measurability}, $\omega\mapsto \norm{A^n_{\omega}},\norm{A^n_\omega v}$ is measurable for all $n\in \N$ and $v\in \B$. Denote
\[\lambda(\omega,v) = \limsup_{n\to\infty}\frac{1}{n} \log \norm{A^n_\omega v}\]
for $\omega\in\Omega$ and $v\in \B$. Moreover, for $\lambda\in \R\cup\{-\infty\}$, we write
\[F_\lambda(\omega) = \curly{v\in \B:\lambda(\omega,v)\leq \lambda}.\]

Lastly, for an operator $A\in L(\B)$, define the Kuratowski measure of noncompactness (see \cite{Akhmerov1992}) by
\[\norm{A}_\alpha = \inf\curly{d>0:AB_{\B} \text{ admits a finite cover by sets of diameter } d}.\]

\begin{thm}[The Non-invertible Multiplicative Ergodic Theorem]\label{thm:non-invMET}
    Let $A:\Omega\to L(\B)$ be a cocycle that satisfies condition \hyperref[ass:measurability]{$\measurable_\sigma$} and fix $(\tau; D) \in \measurable_{\sigma,A}$. Suppose
    \[\int \log^+\norm{A_\omega} \,d\p(\omega) < \infty.\]
    There is then a measurable $\sigma$-invariant set $\Gamma\subseteq \Omega$ of full probability such that for all $\omega\in \Gamma$, the sequence of operators $(A_\omega,A_{\sigma\omega},A_{\sigma^2\omega},\ldots)$ satisfies \autoref{assump:refined_oneidedSequences}, and writing $\omega\mapsto l_q(\omega),l_\alpha(\omega),\chi_q(\omega), r(\omega)$ for $q\in \N$ as in \autoref{thm:sequentialMET_oneSided},
    \begin{equation}\label{eq:MNC_for_MET}
        l_\alpha(\omega)\coloneqq \inf_{q} \chi_q(\omega) = \lim_{n\to\infty} \frac{1}{n}\log\norm{A^n_\omega}_{\alpha}.
    \end{equation}
    For $\omega\in \Gamma$ such that $r(\omega) \geq 1$, write $\omega\mapsto m_s(\omega)$, $M_{s}(\omega)$, and $\lambda_s(\omega)$ for $1\leq s < r(\omega) + 1$ as in \autoref{thm:sequentialMET_oneSided} applied to $(A_\omega,A_{\sigma\omega},A_{\sigma^2\omega},\ldots)_{\omega\in \Gamma}$. Then these maps are measurable and $\sigma$-invariant. Moreover, we have the following.
    \begin{enumerate}[label = (\arabic*)]
        \item The maps $\omega\mapsto \lambda(\omega,v)$ for $v\in \B$ are measurable and $\lambda(\sigma\omega,A_\omega v) = \lambda(\omega,v)$ for all $\omega\in \Gamma$ and $v\in \B$.
        \item The maps $\omega\mapsto d(v,F_{\lambda_s}(\omega)),\ d(v,F_{l_\alpha}(\omega))\indic_{r<\infty}$ for $v\in \B$ and $1\leq s < r(\omega) + 1$ are measurable.
        \item The maps $\omega\mapsto \norm{A^n_\omega |_{F_{\lambda_s}(\omega)}},\ \norm{A^n_\omega |_{F_{l_\alpha}(\omega)}}\indic_{r<\infty}$ for $n\in \N$ and $1\leq s < r(\omega) + 1$ are measurable.
        \item For almost every $\omega\in \Omega$, $B_{\B}\cap F_{\lambda_s}(\omega)$ and $B_{\B}\cap F_{l_\alpha}(\omega)$ are $\tau$-closed.
    \end{enumerate}
\end{thm}

\begin{remark}
    We will not directly prove $\omega\mapsto \norm{A^n_\omega}_\alpha$ is measurable as the topological data we feed is too weak. Regardless, $l_\alpha$ is measurable.
\end{remark}

Under the same measurability conditions, if the base dynamics are invertible, then we may also apply \autoref{assump:refined_twoSidedSequences} to obtain a semi-invertible multiplicative ergodic theorem for Banach space cocycles. Consider the set
\[\Mcal^*_{\sigma,A} \coloneqq \curly{\ell\in \B^*:A^n \text{ is } (\B,\ell)\text{-measurable}\ \forall n\in\N}.\]

\begin{thm}[The Semi-Invertible Multiplicative Ergodic Theorem]\label{thm:semi-invMET}
    In addition to the assumptions of \autoref{thm:non-invMET}, suppose $\sigma:\Omega\to\Omega$ is measurably invertible. There is then a measurable, $\sigma$-invariant set $\Delta\subseteq \Gamma$ of full probability such that for all $\omega\in \Delta$, the two-sided sequence of operators $(\ldots,A_{\sigma^{-1}\omega},A_\omega,A_{\sigma\omega},\ldots)$ satisfies \autoref{assump:refined_twoSidedSequences}. Write the fast spaces $\omega\mapsto E_{s}(\omega)$ for $1\leq s < r(\omega) + 1$ as in \autoref{thm:sequentialMET_twoSided}. We have the following measurability conditions.
    \begin{enumerate}[label = (\arabic*)]
        \item The maps $\omega\mapsto d(v,E_s(\omega))$ for $v\in \B$, and $1\leq s < r(\omega) + 1$ are measurable.
        \item The maps $\omega\mapsto \norm{A^n_\omega|_{E_s(\omega)}},\ m(A^n_\omega|_{E_s(\omega)})$ for $n\geq 1$ are measurable.
        \item The $L(\B,\norm{\cdot})$-valued maps $\omega\mapsto \pi_{E_s(\omega)//F_{\lambda_{s+1}}(\omega)},\ \indic_{r<\infty}\pi_{E_r(\omega)//F_{l_\alpha}(\omega)}$ are $(\B,\Mcal^*_{\sigma,A})$-measurable and $BL_\tau(\B,\tau)$-valued.
        \item The maps $\omega\mapsto \sin\theta\paren{E_s(\omega),F_{\lambda_{s+1}}(\omega)},\ \sin\theta\paren{E_{r}(\omega),F_{l_\alpha}(\omega)}\indic_{r<\infty}$ are measurable.
    \end{enumerate}
\end{thm}

\subsection{Preliminary Results}
    For the rest of this subsection, we will make the following standing assumption. Let $(B,\norm{\cdot})$ be a Banach space such that $B_{\B^*}$ is weak-* separable and let $(\Omega,\F,\p)$ be a probability space. Let $\tau$ be a vector space topology on $\B$ that is coarser than the norm topology for which $B_\B$ is $\tau$-closed and separable. Let $D\subset B_{\B^*}\cap BL_{\tau}(\R)$ be countable and weak-* dense in $B_{\B^*}$. Suppose the operator-valued map $A:\Omega\to L(\B)\cap BL_{\tau}(\B,\tau)$ is $(\B,D)$-measurable.

\subsubsection*{Measurable Volume Functions}
The key to measurability of the volume functions is lower semi-continuity (lsc), for which we point to \autoref{appSubSec:LSC} for a basic overview.

\begin{prop}\label{prop:norm_and_volume_measurability}
    The following real-valued functions are Borel measurable.
    \begin{enumerate}[label = $\bullet$]
        \item $\omega\mapsto \norm{A_\omega v},\norm{A_\omega^*\ell}$ for all $v\in \B$ and all $\ell\in \overline{\vspan(D)}$, where the closure is taken w.r.t. the dual norm.
        \item $\omega\mapsto \norm{A_\omega}$.
        \item $\omega\mapsto \vol_q(A_\omega v_1,\ldots,A_\omega v_q),\vol_q(A_\omega^* \ell_1,\ldots,A_\omega^* \ell_q)$ for all $q\in \N$, $v_1,\ldots,v_q\in \B$, and $\ell_1,\ldots,\ell_q\in \overline{\vspan(D)}$.
        \item $\omega\mapsto \VLL_q(A_\omega),\VLL_q(A_\omega^*)$ for all $q\in \N$.
    \end{enumerate}
\end{prop}
\begin{proof}
    We start with the norms. Take an arbitrary vector $v\in \B$. Then, since the adjoint of a bounded linear operator is weak-* continuous, it follows that
    \[\norm{A_\omega v} = \sup_{\ell\in D} \abs{\innerp{A_\omega v,\ell}},\]
    which implies $\omega\mapsto \norm{A_\omega v}$ is measurable for all $v$. If $\ell\in D$, then $\ell\in BL_\tau(\R)$. Clearly, since $A_\omega\in  BL_\tau(\B,\tau)$, it follows that $A_\omega^*\ell\in BL_\tau(\R)$. Hence, since $B_\B$ is $\tau$-separable, we can take any countable, $\tau$-dense set $S\subset B_\B$ and get that
    \[\sup_{x\in S} \abs{\innerp{x,A_\omega^*\ell}} = \norm{A_\omega^*\ell}\]
    is measurable. The extension to $\ell\in \overline{\vspan(D)}$ follows by continuity of $\ell\mapsto A_\omega^*\ell$ w.r.t. the dual norm.

    To get measurability of the operator norm, let $S\subset B_{B}$ be $\tau$-dense. By \autoref{appLem:LianLuVolLSC} and \autoref{appLem:LSCimpliesSupOverDense}, it follows that
    \[\norm{A_\omega} = \sup_{v\in S} \norm{A_\omega v}.\]
    Since $\omega\mapsto \norm{A_\omega v}$ is measurable for all $v\in S$, and $S$ is countable, it follows that $\omega\mapsto \norm{A_\omega}$ is measurable.

    We now move on to the volume functions. To simplify the proof, we only demonstrate the case for the parallelepiped volume on $\B$, as the case for the dual follows in a similar fashion.
    
    For $q\in \N$ and $v_1,\ldots,v_q$, we shall prove the measurability of
    \[\omega\mapsto \vol_q(A_\omega v_1,\ldots,A_\omega v_q)\]
    through a series of reductions. First, since the product of measurable functions is measurable, it suffices to show
    \[\omega\mapsto d(A_\omega v_1,\innerp{A_\omega v_2,\ldots,A_\omega v_q})\]
    is Borel-measurable for all $q\in\N$ and $v_1,\ldots,v_q\in \B$. Notice that
    \[d(A_\omega v_1,\innerp{A_\omega v_2,\ldots,A_\omega v_q}) = \inf_{t_2,\ldots,t_q\in \Q}\norm{A_\omega \paren{v_1 - \sum_{k=2}^q t_kv_k}},\]
    which is an infimum over a countable set of maps. We have just shown that these maps are measurable, which means the distance function is measurable. Therefore, for all $v_1,\ldots,v_q\in \B$, the map
    \[\omega\mapsto \vol_q(A_\omega v_1,\ldots,A_\omega v_q)\]
    is measurable.
    
    Now, consider any countable set $S\subset B_\B$ that is $\tau$-dense. By \autoref{appLem:LianLuVolLSC} and \autoref{appLem:LSCimpliesSupOverDense},
    \[\VLL_q(A_\omega) = \sup_{v_1,\ldots,v_q\in S} \vol_q(A_\omega v_1,\ldots,A_\omega v_q),\]
    which implies it is a measurable map. Note that $\norm{\cdot}_{\B^*}$ is l.s.c. w.r.t. the weak-* topology, analogously to how $\norm{\cdot}_{\B}$ is l.s.c. w.r.t. $\tau$. This allows us to use \autoref{appLem:LianLuVolLSC} for weak-* l.s.c. of the parallelepiped volumes on $\B^*$. So we may similarly write
    \[\VLL_q(A^*_\omega) = \sup_{\ell_1,\ldots,\ell_q\in D} \vol_q(A^*_\omega \ell_1,\ldots,A^*_\omega \ell_q),\]
    and thus $\omega\mapsto \VLL_q(A^*_\omega)$ is measurable.
\end{proof}

\begin{lem}\label{lem:measOfOperatorActionOnStrMeas}
    Suppose $M\subseteq \B^*$ is any set of linear functionals such that $A$ is $(\B,M)$-measurable. Let $v:\Omega\to \B$ and $\ell:\Omega\to M$ be strongly measurable (in the sense of \autoref{appDefn:strong_meas}). Then $\omega\mapsto \innerp{A_\omega v_\omega,\ell_\omega}$ is Borel-measurable. Moreover, $\omega\mapsto\norm{A_\omega v_\omega}$ is Borel-measurable, and if $M\subseteq BL_\tau(\R)$, then $\omega\mapsto \norm{A^*_\omega \ell_\omega}$ is also Borel-measurable.
\end{lem}
\begin{proof}
    For the first statement, since $v$ and $\ell$ are strongly-measurable, we may assume they are separably valued. So, $v:\Omega\to\B$ and $\ell:\Omega\to M$ can be written as the limit of a sequence of finite-valued measurable maps $v_n:\Omega\to\B$ and $\ell_n:\Omega\to M$, respectively. As such, it follows immediately that 
    \[\omega\mapsto\innerp{A_\omega v_\omega,\ell_\omega} = \lim_{n\to\infty} \innerp{A_{\omega} v_n(\omega),\ell_n(\omega)}\]
    is Borel-measurable. For the second statement, we can write
    \[\norm{A_\omega v_\omega} = \lim_{n\to\infty} \sup_{\ell\in D} \innerp{A_{\omega}v_n(\omega),\ell},\]
    which means it is measurable. For the third statement, if $M\subseteq BL_\tau(\R)$, then we can take any countable $\tau$-dense set $S\subseteq B_\B$ to get that
    \[\norm{A^*_\omega \ell_\omega} = \lim_{n\to\infty} \sup_{v\in S} \innerp{A_{\omega}v,\ell_n(\omega)}\]
    is measurable.
\end{proof}

\subsubsection*{Measurability of Subspaces}
It will be useful to write $G = \overline{\vspan(D)}$, where the closure is taken w.r.t. the dual norm on $\B^*$. We will use the following proposition for the purposes of ascertaining slow-space measurability properties.
\begin{prop}\label{prop:measurableSlowSpace}
    Let $q\in \N$ and $\alpha\in (0,1)$. Then the following hold.
    \begin{enumerate}[label=(\arabic*)]
        \item There is a strongly measurable selection of maps $\ell_1,\ldots,\ell_q:\Omega\to B_{G}$ such that
        \[\vol_q(A_\omega^*\ell_1(\omega),\ldots,A_\omega^*\ell_q(\omega))\geq \alpha\VLL_q(A_\omega^*).\]
        \item Writing $F(\omega) = \preann{\innerp{A_\omega^*\ell_1(\omega),\ldots,A_\omega^*\ell_q(\omega)}}$, the map
        \[\omega\mapsto d(v,F(\omega))\]
        from $\Omega\to \R$ is Borel-measurable for all $v\in \B$.
    \end{enumerate}
\end{prop}
\begin{proof}
    Partition $\Omega$ into $\Omega_0 = \{\VLL_q(A_\omega^*) = 0\}$ and $\Omega' = \Omega_0^c$. On $\Omega_0$, we can simply define $\Psi(\omega) = \{(0,\ldots,0)\}$, so we focus on $\omega\in \Omega'$. 
    
    Notice that $G\subset \B^*$ is a norm-closed and separable subspace of $\B^*$, and hence, it is a Polish vector space. In particular, this means $B_G^q = B_G\times\cdots\times B_G$ is a Polish space. Moreover,
    \[\VLL_q(A^*_\omega) = \sup_{\ell_1,\ldots,\ell_q\in B_G} \vol_q(A_\omega^* \ell_1,\ldots,A_\omega^*\ell_q),\]
    by \autoref{appLem:LianLuVolLSC}. As such, the multifunction
    \[\Psi(\omega) = \curly{(\ell_1,\ldots,\ell_q)\in B^q_G:\vol_q(A_\omega^*\ell_1,\ldots,A_\omega^*\ell_q) > \alpha\VLL_q(A_\omega^*)}\]
    is non-empty and open-valued for all $\omega\in \Omega'$. To show it is weakly measurable (see \autoref{appDefn:measurableMultifunction}), consider a norm-open set $U\subseteq B_G^q$. Let $D_G$ be a countable, dense set in $G$. Then, by the openness of $U$ and $\Psi$, $D_G^q\cap U$ is dense in $U$ and
    \[\Psi^{-}(U) \coloneqq \curly{\omega:\Psi(\omega)\cap U\neq\varnothing}=\bigcup_{(\ell_1,\ldots,\ell_q)\in D_G^q\cap U} \curly{\omega:\vol_q(A_\omega^*\ell_1,\ldots,A_\omega^*\ell_q) > \alpha\VLL_q(A_\omega^*)}.\]
    The right-most set is the countable union of measurable sets, and as such, is itself measurable, proving $\Psi$ is weakly-measurable. Hence, by \autoref{appCor:LianLuMeasSelection}, the norm-closure $\omega\mapsto \overline{\Psi}(\omega)$ admits a measurable selection. That is, there is a Borel-measurable selection of maps $\ell_k:\Omega'\to B_{G}$ such that
    \[\vol_q(A_\omega^*\ell_1(\omega),\ldots,A_\omega^*\ell_q(\omega))\geq \alpha\VLL_q(A_\omega^*).\]
    Since $G$ is separable, Borel measurability of $G$-valued maps is equivalent to strong measurability. 
    
    Now, write $F(\omega) = \preann{\innerp{A_\omega^*\ell_1(\omega),\ldots,A_\omega^*\ell_q(\omega)}}$ as in the lemma statement and let $v\in \B$. Then we know by \autoref{lem:distance_vs_orthogonalProj} that
    \[d(v,F(\omega)) = \sup_{t_1,\ldots,t_q\in \Q\setminus\{0\}}\frac{\innerp{A_\omega v,t_1\ell_1(\omega)+\cdots+t_q\ell_q(\omega)}}{\norm{t_1A^*_\omega\ell_1(\omega)+\cdots+t_qA^*_\omega\ell_q(\omega)}}\]
    for $\omega\in \Omega'$. As the above is a countable supremum of measurable maps (by \autoref{lem:measOfOperatorActionOnStrMeas}), and $F(\omega) = \B$ for $\omega\in \Omega_0$, this shows $\omega\mapsto d(v,F(\omega)) = d(v,F(\omega))\indic_{\Omega'}$ is measurable.
\end{proof}

A similar proposition may be formulated for the purposes of ascertaining fast-space measurability properties.
\begin{prop}\label{prop:measurableFastSpace}
    Let $q\in \N$ and $\alpha\in (0,1)$. Then the following hold.
    \begin{enumerate}[label=(\arabic*)]
        \item There is a strongly measurable selection of maps $v_1,\ldots,v_q:\Omega\to \B$ such that
        \[\vol_q(A_\omega v_1(\omega),\ldots,A_\omega v_q(\omega))\geq \alpha\VLL_q(A_\omega).\]
        \item Writing $E(\omega) = \innerp{A_\omega v_1(\omega),\ldots,A_\omega v_q(\omega)} $, the map
        \[\omega\mapsto d(v,E(\omega))\]
        from $\Omega\to \R$ is Borel-measurable for all $v\in \B$.
    \end{enumerate}
\end{prop}
\begin{proof}
    The proof of item (1) follows the same logic as the proof of item (1) of \autoref{prop:measurableSlowSpace}, with the only difference being that the analogue to $G$ in that proof can be given in this proof by the closed span of any countable $\tau$-dense set in $B_\B$. Then, for item (2) we can write
    \[d(v,E(\omega)) = \inf_{t_1,\ldots,t_q\in \Q} \norm{v - \sum_{t=1}^q t_i A_\omega v_i(\omega)}.\]
    Applying \autoref{lem:measOfOperatorActionOnStrMeas} and the fact that the above norm can be written as a countable supremum over $\ell\in D$ gives us measurability.
\end{proof}

One last result that will be useful is the following.
\begin{lem}\label{lem:lipContOfRestrictedSubspaceNorms}
    Let $A\in L(\B)$ and take $E,F\in \G(\B)$. Then
    \[\abs{m(A|_E)-m(A|_F)}\leq \norm{A} d_H(E,F) \text{ and } \abs{\norm{A|_E} - \norm{A|_F}}\leq \norm{A} d_H(E,F).\]
\end{lem}
\begin{proof}
    The proof is straightforward and details are left to the reader.
\end{proof}

\begin{remark}
    For the rest of this section, we shall take the convention that $0\cdot\text{undefined} = 0$. This is for the sake of brevity, as we would like to use indicator maps to give a global definition to maps only defined on a partial region of $\Omega$.
\end{remark}

\subsection{The Non-Invertible MET}
In this section, we shall prove \autoref{thm:non-invMET}. Let $(\Omega, \F, \p)$ be a probability space, and let $\sigma:\Omega\to\Omega$ be an mpt. Let $\B$ be a Banach space such that $B_{\B^*}$ is weak-* separable. Suppose $A$ satisfies condition \hyperref[ass:measurability]{$\measurable_\sigma$} so that $\measurable_{\sigma,A}$ is non-empty and fix some $(\tau;D)\in \measurable_{\sigma,A}$. As before, we write $G = \overline{\vspan(D)} \subseteq \Mcal_{\sigma,A}^*\cap BL_\tau(\R)$. As the values that $A$ takes on a measurable null set do not matter, and $\sigma$ preserves $\p$, we may WLOG assume $A_{\omega} \in BL_{\tau}(\B,\tau)$ for all $\omega\in \Omega$.

\subsubsection*{The Lyapunov Spectrum}
Recall, by \autoref{prop:norm_and_volume_measurability}, the map $\omega\mapsto \norm{A_\omega}$ is measurable. By assumption, $\log^+\norm{A_\omega}\in L^1$, and so the Birkhoff ergodic theorem implies
\[\lim_{n\to\infty} \frac{1}{n}\sum_{k=0}^{n-1} \log^+\norm{A_{\sigma^k\omega}}\]
exists and is finite almost surely. It is a routine exercise to show that this implies
\begin{equation}\label{eq:assumption_i_MET}
    \limsup_{n\to\infty}\frac{1}{n} \log \norm{A_{\sigma^n\omega}}\leq \lim_{n\to\infty} \frac{1}{n} \log^+ \norm{A_{\sigma^n\omega}} = 0
\end{equation}
for almost every $\omega\in \Omega$. Now, for $q,n\in \N$, write
\[\phiq_n(\omega) = \log \VLL_q(A^n_\omega),\]
which is measurable by \autoref{prop:norm_and_volume_measurability}. Since $d(v,E)\leq\norm{v}$ for any subspace $E$, it is clear from the definition of $\vol_q$ that
\[\paren{\phiq_1(\omega)}^+ \leq q\log^+\norm{A_\omega} \in L^1.\]
By submultiplicativity of $\VLL_q$, we can see for $n,m\in \N$ that
\[\phiq_{n+m}(\omega)\leq \log \paren{\VLL_q(A^m_{\sigma^n\omega})\VLL_q(A^n_\omega)} = \phiq_n(\omega)+\phiq_m(\sigma^n \omega),\]
which means it is a subadditive family satisfying the assumptions of Kingman's subadditive ergodic theorem (\autoref{appThm:KingmanSET}). Thus, by Kingman's theorem,
\[l_q(\omega) \coloneqq \lim_{n\to\infty} \frac{1}{n} \vphi_n(\omega) = \lim_{n\to\infty} \frac{1}{n} \log \VLL_q(A^n_\omega)\]
exists almost surely, is measurable, and is $\sigma$-invariant. Combining the above with \autoref{lem:equivalenceOfVolNorms} and \autoref{appThm:temperedSET}, it follows that there is a measurable $\sigma$-invariant set $\Gamma\subseteq \Omega$ of full probability such that $(A_\omega,A_{\sigma\omega},A_{\sigma^2\omega},\ldots)$ satisfies \autoref{assump:refined_oneidedSequences} for all $\omega\in \Gamma$. 

We may now take all the notations of \autoref{thm:non-invMET}. As the $(l_q:\Omega\to\R\cup\{-\infty\})_{q\in\N}$ are now just standard $\R\cup\{-\infty\}$-valued maps, measurability and $\sigma$-invariance of $l_q$ imply measurability and $\sigma$-invariance of $\chi_q, r,l_\alpha = \inf_{q\in \N} \chi_q$ for $q\in \N$ and $m_s, M_{s} = m_1+\cdots+m_s,\lambda_s$ for $1\leq s < r+1$. For item (1), since $\omega\mapsto \norm{A^n_\omega v}$ is measurable for all $n\in\N$ and $v\in \B$, we have that $\lambda(\omega,v) = \limsup_{n\to\infty} \frac{1}{n}\log\norm{A^n_\omega v}$ is measurable for all $v\in \B$. Moreover, the fact that
\[\lambda(\sigma\omega,A_\omega v) = \lambda(\omega,v)\]
follows directly from the definition of $\lambda(\omega,v)$. All that remains is to prove \eqref{eq:MNC_for_MET}.

\paragraph{The Measure of Noncompactness}
Our proof follows along the lines of the proof of Lemma 3.2(3) in \cite{Blumenthal2016}. We note by Corollary 2.5.5 and Theorem 1.1.7 in \cite{Akhmerov1992} that 
\[\frac{1}{4}\norm{A}_\alpha \leq c(A)\leq 4\norm{A}_\alpha,\]
where $c(A) = \inf_{q\in \N} c_q(A)$. Notice that for $\omega\in \Gamma$,
\[\limsup_{n\to\infty}\frac{1}{n}\log \norm{A^n_\omega}_\alpha \leq \inf_{q\in \N} \lim_{n\to\infty} \frac{1}{n} \log (4c_q(A^n_\omega)) = \inf_{q\in \N} \chi_q = l_\alpha. \]
To show the reverse inequality, note first that for an operator $A\in L(\B)$, it follows by Lemma 4.5 of \cite{LianLu2010} that the sequence $q\mapsto (\VLL_q(A))^{1/q}$ is monotonically decreasing. Define the quantity
\[\gamma(A) \coloneqq \inf_{q\in \N} (\VLL_q(A))^{1/q}.\]
By Lemma 4.6 of \cite{LianLu2010}, $\gamma(A)\leq 2\norm{A}_\alpha$. Notice, by submultiplicativity of $\VLL_q$, $\gamma$ is a submultiplicative quantity. Moreover, by \autoref{prop:norm_and_volume_measurability}, $\omega\mapsto \gamma(A^n_\omega)$ is measurable for all $n$, and we also have $\log^+\gamma(A_\omega)\leq \log^+\norm{A_\omega}\in L^1(\Omega)$. Therefore, writing $\phi_n(\omega) = \log \gamma(A_\omega^n)$, we have by \autoref{appThm:KingmanSET} that
\[\lim_{n\to\infty}\frac{1}{n} \phi_n = \inf_{n\in \N} \frac{1}{n}\E\bracket{\log\gamma(A^n_\omega)|\Ical} \text{ a.s.},\]
where $\Ical = \{E\in\F:\sigma^{-1}E = E\}$ is the sub-$\sigma$-algebra of $\sigma$-invariant sets. Thus, for $\p$-a.e. $\omega\in \Gamma$,
\[l_\alpha(\omega) = \inf_{q\in \N} \frac{1}{q} l_q(\omega) = \inf_{n\in \N} \inf_{q\in \N} \frac{1}{nq}\E\bracket{\log \VLL_q(A^n_\omega)|\Ical}(\omega) = \lim_{n\to\infty} \frac{1}{n} \phi_n(\omega)\leq\liminf_{n\to\infty} \frac{1}{n}\log\norm{A^n_\omega}_\alpha.\]
Restricting $\Gamma$ by a null set so that the above holds for all $\omega \in \Gamma$, it follows that $\Gamma$ is $\sigma$-invariant, $\p(\Gamma) = 1$, and \eqref{eq:MNC_for_MET} holds for all $\omega\in \Gamma$.

\subsubsection*{Measurability of the Filtration}\label{subsubsec:nonInvMET_measurability}
Write $G = \overline{\vspan(D)}$. For $2\leq s < r(\omega)+1$ and $n\in \N$, we may apply \autoref{prop:measurableSlowSpace} to get a measurable selection of maps $\ell_{1,n},\ldots,\ell_{M_{s-1},n}:\Omega\to B_{G}$ such that
\begin{equation}\label{eq:slowSpaceApproxMeasSelection}
    \vol_{M_{s-1}}((A_\omega^n)^*\ell_1(\omega),\ldots,(A_\omega^n)^*\ell_{M_s-1}(\omega))\geq \frac{1}{2}\VLL_{M_{s-1}}((A_\omega^n)^*)
\end{equation}
\begin{remark}
    Since $M_{s-1}$ is random, we technically must restrict the domain to the $\sigma$-invariant strata $\{M_{s-1} = m\}$ for $m\in \N$. This shall be the convention we take in this section and the next without further mention.
\end{remark}
Write $L_{n,s-1}(\omega) = \innerp{\ell_{n,1}(\omega),\ldots,\ell_{n,M_{s-1}}(\omega)}$ and $F_{n,s}(\omega) = \preann{\paren{(A_\omega^n)^*L_{n,s-1}(\omega)}}$. By \autoref{lem:equivalenceOfVolNorms} and \autoref{cor:reverse_dual_SVD}, $F_{n,s}$ is a closed $M_{s-1}$-codimensional subspace of $\B$ such that
\[\norm{A^n_\omega |_{F_{n,s}(\omega)}}\leq Cc_{M_{s-1}+1}(A^n_\omega),\]
for some $C\geq 1$ that depends only on $M_{s-1}$. By \autoref{cor:uniquenessOfSlowSpaces}, $F_{n,s}(\omega)\to F_{\lambda_s}(\omega)$ a.s. in the gap-topology. 

If $r(\omega) < \infty$, then the exact same logic holds. By \autoref{prop:measurableSlowSpace}, there is a measurable selection for a basis of $L_{n,r}(\omega)\subset G$ such that \eqref{eq:slowSpaceApproxMeasSelection} holds for $q = M_{r}(\omega)$, and thus $F_{n,\alpha}(\omega) \coloneqq \preann{(A^n_\omega)^*L_{n,r}}(\omega)\to F_{l_\alpha}(\omega)$ in the gap-topology as $n\to\infty$.

\noindent \textbf{Item (2):} If $v\in \B$, then $\omega\mapsto d(v,F_{n,s}(\omega))$ is measurable by item (2) of \autoref{prop:measurableSlowSpace}. Hence, the map
\[d(v,F_{\lambda_s}(\omega)) = \lim_{n\to\infty} d(v,F_{n,s}(\omega))\]
is measurable for all $v\in \B$. If $r(\omega)<\infty$, then the same reasoning applies to $F_{l_\alpha}(\omega)$.

\noindent \textbf{Item (3):} Since $F_{\lambda_1}(\omega) = \B$, measurability in the case $s = 1$ is satisfied by \autoref{prop:norm_and_volume_measurability}, so we may assume $s\geq 2$. Notice that
\[\norm{A^n_\omega|_{F_{\lambda_{s}}(\omega)}} = \sup_{\ell\in D} \norm{\paren{A^n_\omega}^*\ell|_{F_{\lambda_{s}}}} = \sup_{\ell\in D} d(\paren{A^n_\omega}^*\ell,\ann{F_{\lambda_s}})\]
by \autoref{appLem:LianLuVolLSC} and \autoref{lem:distance_vs_orthogonalProj}. As this is a countable supremum, it suffices to show the arguments in the supremum are measurable.

Since $G = \overline{\vspan(D)}\subseteq BL_\tau(\R)$, it follows for $\ell\in D$, $n\in \N$, and $2 \leq s < r(\omega) + 1$ that
\begin{align*}
    d((A^n_\omega)^*\ell,\ann{F_{\lambda_s}(\omega)}) &= \lim_{m\to\infty}d((A^n_\omega)^*\ell,(A^m_\omega)^*L_{m,s-1}(\omega)) \\
    &= \lim_{m\to\infty}\inf_{t_1,\ldots,t_{M_{s-1}}\in \Q} \sup_{v\in S}\bracket{\innerp{A_{\omega}^n v,\ell} - \innerp{A^m_\omega v,t_1\ell_{m,1}+\cdots+t_{M_{s-1}}\ell_{m,M_{s-1}}}},
\end{align*}
where $S$ is any countable $\tau$ dense subset of $B_{\B}$. Thus, the above map is measurable, which gives us measurability of $\omega\mapsto \norm{A^n_\omega|_{F_{\lambda_{s}}(\omega)}}$ for all $n\geq 1$. If $r(\omega)<\infty$, then the same reasoning applies to $F_{l_\alpha}$.

\noindent \textbf{Item (4):} For $1\leq s < r(\omega) + 1$, we know that $\ann{F_{n,s}}\subseteq BL_\tau(\R)$. Since $\ann{F_{n,s}}\to \ann{F_{\lambda_s}}$ in the Grassmannian topology, \autoref{appLem:closednessOfBLtau} implies $\ann{F_{\lambda_s}} \subseteq BL_\tau(\R)$, and hence, that $B_\B\cap F_{\lambda_s}$ is $\tau$-closed. If $r(\omega) < \infty$, the same logic applies to $F_{l_\alpha}$. Otherwise, we can write
\[B_\B \cap F_{l_\alpha} = \bigcap_{s=1}^\infty B_\B\cap F_{\lambda_{s}},\]
which is $\tau$-closed as the intersection of $\tau$-closed sets.

\subsection{The Semi-Invertible MET}
In this section, we prove \autoref{thm:semi-invMET}. We shall take the same setting as for the non-invertible MET except that we also assume $\sigma:\Omega\to\Omega$ has measurable inverse.

\subsubsection*{The Lyapunov Spectrum}
As before, Birkhoff's ergodic theorem may be used to show that
\[\limsup_{n\to\infty} \frac{1}{n}\log \norm{A_{\sigma^{-n}\omega}} \leq 0\]
for almost every $\omega\in \Omega$. Moreover, by Kingman's theorem and \autoref{appThm:balancedKingmanSET} (cf. \cite{Lee2024}), we have for $q\in \N$ and almost every $\omega\in \Omega$ that
\[l_q(\omega) = \lim_{n\to\infty} \frac{1}{n} \phiq_n(\sigma^{-n}\omega) = \lim_{n\to\infty} \frac{1}{2n} \phiq_{2n}(\sigma^{-n}\omega).\]
Combining the above with \autoref{lem:equivalenceOfVolNorms} and \autoref{appThm:temperedSET}, it follows that there is a measurable $\sigma$-invariant set $\Delta\subseteq \Omega$ of full probability such that, for all $\omega\in \Delta$, the two-sided sequence $(\ldots,A_{\sigma^{-1}\omega},A_\omega,A_{\sigma\omega},\ldots)$ satisfies \autoref{assump:refined_twoSidedSequences}.

\subsubsection*{Measurability of the Fast Spaces}
Now, let $n\in \N$. For $1\leq s < r(\omega) + 1$, we may apply \autoref{prop:measurableFastSpace} to get a strongly measurable selection of maps $w_{1,n},\ldots,w_{M_s,n}:\Omega\mapsto B_\B$ such that
\[\vol_{M_s}(A^n_\omega w_{1,n}(\omega),\ldots,A^n_\omega w_{M_s,n}(\omega))\geq \frac{1}{2} \VLL_{M_s}(A^n_\omega).\]

Write $W_{n,s}(\omega) = \innerp{w_{1,n}(\omega),\ldots,w_{M_s,n}(\omega)}$ and
\[E_{n,s}(\omega) \coloneqq A^n_{\sigma^{-n}\omega} W_{n,s}(\sigma^{-n}\omega).\]
Denoting $A^{*n}_{\omega} = (A^n_{\sigma^{-n}\omega})^* = A_{\sigma^{-n}\omega}^*\cdots A_{\sigma^{-1}\omega}^*$, it follows by \autoref{lem:equivalenceOfVolNorms} and \autoref{lem:dual_SVD} that
\[\norm{A^{*n}_{\omega}|_{\Fhat_{n,s+1}(\omega)}}\lesssim_{M_{s}} c_{1+M_{s}}(A^{*n}_{\omega}),\]
where $\Fhat_{n,s+1}(\omega) = \ann{E_{n,s}(\omega)}$. Thus, by \autoref{cor:uniquenessOfSlowSpaces} (applied to the adjoint cocycle), and utilizing the connection between slow spaces for the adjoint cocycle and fast spaces for the original cocycle by \autoref{thm:sequentialMET_twoSided}, we get that $E_{n,s}(\omega) \to E_{s}(\omega)$ as $n\to\infty$ in the gap-topology.

\vspace{2mm}
\noindent \textbf{Item (1):} By \autoref{prop:measurableFastSpace}, $\omega\mapsto d(v,E_{n,s}(\omega))$ is measurable for all $v\in \B$, so
\[\omega\mapsto d(v,E_{s}(\omega)) = \lim_{n\to\infty} d(v,E_{n,s}(\omega))\]
is measurable for all $v\in \B$.

\vspace{2mm}
\noindent \textbf{Item (2):} Since $E_{k,s}(\omega)\to E_{s}(\omega)$ as $k\to\infty$ in the Grassmannian topology, it follows by \autoref{lem:lipContOfRestrictedSubspaceNorms} that
\[\norm{A^n_{\omega}|_{E_{k,s}(\omega)}}\underset{k\to\infty}{\to}\norm{A^n_{\omega}|_{E_{s}(\omega)}} \text{ and } m\paren{A^n_{\omega}|_{E_{k,s}(\omega)}}\underset{k\to\infty}{\to}m\paren{A^n_{\omega}|_{E_{s}(\omega)}}\]
for all $n\in \N$. To simplify the argument, let us focus on the minimum norm. Notice that
\[m\paren{A^n_{\omega}|_{E_{k,s}(\omega)}} = \inf_{t_1,\ldots,t_{M_s}\in \Q\setminus \{0\}} \frac{\norm{t_1 A_{\sigma^{-k}\omega}^{n+k} w_{k,1}+\cdots+t_{M_s} A_{\sigma^{-k}\omega}^{n+k} w_{k,M_s}}}{\norm{t_1 A_{\sigma^{-k}\omega}^{k} w_{k,1}+\cdots+t_{M_s} A_{\sigma^{-k}\omega}^{k} w_{k,M_s}}}.\]
This is an infimum of a countable collection of measurable functions (by \autoref{lem:measOfOperatorActionOnStrMeas}), so the infimum is measurable. Thus, by the limit identity, $\omega\mapsto m\paren{A^n_{\omega}|_{E_{s}(\omega)}}$ is measurable for all $n$. Measurability of the operator norm follows exactly similarly as it is simply a countable supremum of the same measurable functions instead of an infimum.

\vspace{2mm}
\noindent \textbf{Items (3) and (4):} 
To determine measurability of the projection operators, we require the following lemma.
\begin{lem}\label{lem:projectionOperatorDecomposition}
    Let $q\in \N$ and suppose $E\in\G_q(\B)$ and $F\in \G^q(\B)$ are complements. Let $(v_1,\ldots,v_q)$ be a basis of $E$ and $(\ell_1,\ldots,\ell_q)$ a basis of $\ann{F}\in \G_q(\B^*)$. Consider the matrix $M\in \R^{q\times q}$ defined by
    \[M_{ij} = \innerp{v_j,\ell_i}.\]
    Then, $M$ is invertible and for all $v\in \B$,
    \begin{equation}\label{eq:projectionCoordinateForm}
        \pi_{E//F}v = \sum_{i = 1}^q \sum_{j = 1}^q [M^{-1}]_{ij}\innerp{v,\ell_{j}} v_i.
    \end{equation}
    Conversely, invertibility of $M$ implies $E$ and $F$ are topological complements.
\end{lem}
\begin{proof}
    Let us first do the converse direction. Suppose $M$ is invertible. Let $T:\B\to\B$ be the linear operator defined by the right side of \eqref{eq:projectionCoordinateForm}. Then it is immediately clear that $T|_{E} = \id_E$, $(I-T)|_{F} = \id_F$, and $\norm{T} \leq \sum_{i,j} \abs{[M^{-1}]_{ij}}\norm{l_j}\norm{v_i} < \infty$. Hence, $T = \pi_{E//F}$ is well-defined and bounded.
    
    We now do the forward-direction. By Hahn-Banach, there must exist a linear functional $f_i\in \ann{\paren{\innerp{v_j:j\neq i}+F}}$ such that $f_i(v_i) = 1$. Since $f_i\in \ann{F}$, it follows that there must exist unique $n_{ij}\in \R$ such that
    \[f_i = \sum_{j=1}^q n_{ij}\ell_j.\]
    Let $N$ be the matrix defined by $N_{ij} = n_{ij}$. Then
    \[[NM]_{ik} = \sum_{j=1}^q N_{ij} M_{jk} = \sum_{j=1}^q n_{ij}\innerp{v_k,\ell_j} = f_i(v_k) = \delta_{ik},\]
    which means $NM = \id_{\R^q}$. Hence, $M$ is invertible and $N = M^{-1}$. Now that we know $M$ is invertible, the fact that \eqref{eq:projectionCoordinateForm} holds follows as in the converse direction.
\end{proof}

In the case $r(\omega) < \infty$, it will be notationally convenient to write $\lambda_{r+1}(\omega) = l_\alpha(\omega)$. Now, let $L_{n,s}(\omega) = \innerp{\ell_{n,1}(\omega),\ldots,\ell_{n,M_{s}}(\omega)}$ and $F_{n,s+1}(\omega) = \preann{\paren{(A_\omega^n)^*L_{n,s}(\omega)}}$ be as in \autoref{subsubsec:nonInvMET_measurability}. For the sake of brevity, write $v_{n,i} = A^n_{\sigma^{-n}\omega} w_{n,i}$ and $f_{n,i} = (A_\omega^n)^*\ell_{n,i}$. Then the matrix-valued map
\[[M^{n,s}]_{ij} = \innerp{v_{n,j},f_{n,i}} = \innerp{A^{2n}_{\sigma^{-n}\omega} w_{n,j},\ell_{n,i}}\]
is measurable by \autoref{lem:measOfOperatorActionOnStrMeas}. Since invertibility of $M^{n,s}$ determines whether $E_{n,s}$ and $F_{n,s+1}$ are complements, and $M^{n,s}$ is a measurable real $(M_s\times M_s)$-matrix valued function, it follows that there is a measurable map $\nt_s:\{s < r(\omega)+1\}\to \N$ such that $E_{n,s}$ and $F_{n,s+1}$ are topological complements for all $n\geq \nt_s(\omega)$; for example,
\[\nt_s(\omega) = \min\curly{N:\det M^{n,s} \neq 0\ \forall n\geq N}.\]
Writing $N^{n,s} = \paren{M^{n,s}}^{-1}$ for $n\geq \nt_s(\omega)$, this means the coefficient maps $\omega\mapsto [N^{n,s}(\omega)]_{ij}$ are measurable. 

Now, since $E_{n,s}\to E_s$ and $F_{n,s+1}\to F_{\lambda_{s+1}}$ in the Grassmannian topology and $\curly{\nt_s < \infty} = \curly{s < r+1}$, it follows that 
\[\pi_{E_s//F_{\lambda_{s+1}}} = \lim_{n\to\infty} \indic_{n\geq \nt_s}\pi_{E_{n,s}//F_{n,s+1}}\]
in the $L(\B,\norm{\cdot})$-operator norm topology.
If $\ell\in \Mcal_{\sigma,A}^*$ and $v\in \B$, then by \autoref{lem:projectionOperatorDecomposition}, we have
\[\innerp{\pi_{E_{n,s}//F_{n,s+1}}v,\ell} = \sum_{i=1}^{M_s}\sum_{j=1}^{M_s} [N^{n,s}]_{ij} \innerp{A_{\omega}^n v, \ell_{n,j}}\innerp{A^{n}_{\sigma^{-n}\omega}w_{n,i},\ell},\]
which is clearly measurable for $n\geq \nt_s$. As such
\[\innerp{\pi_{E_s//F_{\lambda_{s+1}}} v, \ell} = \lim_{n\to\infty} \indic_{n\geq \nt_s}\innerp{\pi_{E_{n,s}//F_{n,s+1}}v,\ell}\]
is measurable for $1\leq s < r + 1$. In particular, this implies $\omega\mapsto \pi_{E_s(\omega)//F_{\lambda_{s+1}}(\omega)}$ for $1\leq s < r(\omega) + 1$ is $(\B,\Mcal_{\sigma,A}^*)$-measurable. Additionally, \eqref{eq:projectionCoordinateForm} in \autoref{lem:projectionOperatorDecomposition} implies $\pi_{E_{n,s}//F_{n,s+1}}\in BL_{\tau}(\B,\tau)$, so by \autoref{appLem:closednessOfBLtau}, $\pi_{E_s//F_{\lambda_{s+1}}}$ is $ BL_\tau(\B,\tau)$-valued. Item 4 then immediately follows by \autoref{prop:norm_and_volume_measurability}.

\section{Applications}\label{sec:applications}
In this section, we shall first explore some examples of Banach spaces which satisfy the topological needs of our measurability setting. We end the section with an application of our framework to give a multiplicative ergodic theory interpretation to the work of Buzzi on exponential decay of correlations for random Lasota-Yorke maps \cite{Buzzi1999}.

\subsection{Admissible Banach Spaces}
\begin{defn}
    A Banach space $(\B,\norm{\cdot})$ will be \textit{admissible} if it has a weaker (read: coarser than the $\norm{\cdot}$-topology) vector space topology $\tau$ for which the norm-closed unit ball is $\tau$-closed and separable, and for which there is a countable collection of linear functionals in $BL_\tau(\R)$ that is weak-* dense in $B_{\B^*}$. If $\tau$ is such a topology and $D\subset BL_\tau(\R)\cap B_{\B^*}$ is such a collection of linear functionals, we shall call $(\tau; D)$ an \textit{admissions ticket} for $\B$.
\end{defn}

\begin{prop}[Admissibility by Embedding]\label{prop:AdmissibleEmbedding}
    Let $(\B,\norm{\cdot})$ be a Banach space and assume there is a weaker locally convex topology $\tau$ on $\B$ such that $B_\B$ is $\tau$-closed and separable. Then for any countable set $D\subset (\B,\tau)^*\cap B_{\B^*}$ that is $\sigma((\B,\tau)^*,\B)$-dense, $(\tau; D)$ is an admissions ticket for $\B$. Moreover, if $\tau$ is metrizable and locally convex, then there exists at least one such $D$.
\end{prop}
\begin{proof}
    By assumption, we already have that $B_\B$ is $\tau$-closed and separable, and clearly $(\B,\tau)^*\subseteq BL_\tau(\R)$. Suppose there is a countable $\sigma((\B,\tau)^*,\B)$-dense set $D\subset B' = (\B,\tau)^*\cap B_{\B^*}$. If we can show that $B'$ is norming, then the Bipolar theorem implies $B_{\B}$ is the $\sigma(\B^*,\B)$-closure of $B'$, and thus $B_\B$ is weak-* separable and the weak-* closure of $D$. Indeed, if $B'$ is norming, then the $\sigma(\B^*,\B)$-polar of $B'$ is $B_{\B}$, and so its bipolar is $B_{\B^*}$. The bipolar is the weak-* closed convex balanced hull. Since $B'$ is already convex and balanced, its $\sigma(\B^*,\B)$-bipolar is its $\sigma(\B^*,\B)$-closure.

    Now, $B'$ is clearly norming for $x = 0$, so take any $x\in \B\setminus \innerp{0}$. Let $0<\e< 1$ and notice that $(1-\e)\norm{x} B_{\B}$ is a $\tau$-closed and convex set disjoint from $\{x\}$. Thus, by the Hahn-Banach separation theorem, there is a $\tau$-continuous linear functional $f:\B\to\R$ such that
    \[(1-\e)\norm{x} < \innerp{\frac{f}{\norm{f}},x} \leq\sup_{\ell\in B'}\innerp{\ell,x} .\]
    Taking $\e \to 0$, it follows that $B'$ is norming.

    Now, if $\tau$ is metrizable and locally convex, then \autoref{appLem:hereditaryWeak*SeparabilityOfDual} implies $B' = (\B,\tau)^*\cap B_{\B^*}$ is $\sigma((\B,\tau)^*,\B)$-separable. So we are done.
\end{proof}

From the above proposition, we have the following three corollaries, the first two of which are immediate.
\begin{cor}\label{cor:separableIsAdmissible}
    If $\B$ is separable, then $\B$ is admissible.
\end{cor}

\begin{cor}\label{cor:admissibleViaSepBanachEmbedding}
    If $(X,\norm{\cdot}_X)$ is a separable Banach space such that $\B\subseteq X$ continuously embeds into $X$ and $\norm{\cdot}_\B$ is lower semi-continuous w.r.t. the $\norm{\cdot}_X$-topology, then $\B$ is admissible and $(\norm{\cdot}_X,D)$ is an admissions ticket for any countable $\sigma(X^*,X)$-dense set $D\subset X^*\cap B_{\B^*}$.
\end{cor}

\begin{cor}\label{cor:dualOfSepIsAdmissible}
    Suppose $(X,\norm{\cdot}_X)$ is a separable Banach space and let $\B = X^*$. Then $\B$ is admissible. If $D\subseteq B_X$ is a countable $\norm{\cdot}_X$-dense subset, then $(\sigma(\B,X);D)$ is an admissions ticket for $\B$, where we identify $X$ as a subspace of $\B^* = X^{**}$ via the isometric pointwise evaluation map.
\end{cor}
\begin{proof}
    By Banach-Alaoglu, $B_\B$ is $\sigma(\B,X)$-compact, and thus closed. If $D = \{x_j\}_{j\geq 1}\subset B_X$ is countable and $\norm{\cdot}_X$-dense, then compactness implies that
    \[d(v,u) = \sum_{j=1}^\infty \frac{1}{2^j}\frac{\abs{x_j[v-u]}}{1+\abs{x_j[v-u]}}\]
    forms a $\sigma(\B,X)$-compatible metric on $B_\B$, and thus $B_\B$ is $\sigma(\B,X)$-separable. As $(\B,\sigma(\B,X))^* = (X,\sigma(X,\B))$, norm density of $D$ in $B_X = X\cap B_{\B^*}$ implies weak-* density in $B_X$, so we are done by \autoref{prop:AdmissibleEmbedding}.
\end{proof}

\subsubsection*{Examples}
Many nonseparable Banach spaces that arise in dynamics are admissible:
\begin{itemize}
    \item For any standard probability space, $L^\infty$ is the dual of $L^1$. It can also be embedded in $L^p$ for any $p\geq 1$, and the $L^\infty$-norm is lsc w.r.t. the $L^p$-norm.
    \item $C^k_{b}(\R)$ with the uniform-norm $f\mapsto \sup_{0\leq i\leq k} \norm{f^{(i)}}_\infty$ embeds into the separable Fr\'{e}chet topology of convergence on compact subsets, where the uniform-norm is lsc.
    \item On a compact metric space $M$, the spaces of Lipschitz and $\alpha$-H\"{o}lder continuous functions (under the standard norms) embed in the space of continuous functions, and their norms are lsc w.r.t. the uniform-norm.
    \item For sufficiently regular bounded domains $\Omega\subseteq \R^d$, $W^{k,\infty}(\Omega)$ and $\BV(\Omega)$ have separable preduals. Moreover, $W^{k,\infty}$ can be embedded into the separable spaces $W^{k,p}$ for $p<\infty$, while $\BV$ can be embedded in $L^p$, with their respective norms being lsc in the respective weak topologies.
\end{itemize}

We shall not prove all these points in detail, many are standard or folklore results. However, since our main motivating example is transfer operators on $\BV$, we shall establish its admissibility as the dual of a separable Banach space and through its embedding in $L^1$.

\subsubsection*{The Space of Functions of Bounded Variation}
We shall focus on the domain $\Omega = [0,1]$, though the following work goes through similarly for other domains of sufficient regularity \cite{ambrosio2000bv}. For a map $u\in L^1([0,1],\Leb)$, define its total variation to be
\[\var(u) \coloneqq \sup_{\substack{\phi\in C^1_c(0,1) \\ \norm{\phi}_\infty \leq 1}} \int_0^1 u(x)\phi'(x)\,dx.\]
Note that this is exactly equal to Buzzi's \cite{Buzzi1999} partition based definition of variation which assigns two a.s. equal maps the same variation. The space of bounded variation functions on $[0,1]$ is the space $\BV[0,1] = \curly{u\in L^1[0,1] : \var(u) < \infty}$, and the norm we assign on $\BV[0,1]$ is $\norm{u}_{\BV} \coloneqq \norm{u}_{L^1} + \var(u)$. This is a nonseparable Banach space.

\begin{lem}[see Remark 3.12 in \cite{ambrosio2000bv}]
    There is a separable Banach space $X$ such that $\BV[0,1]$ is isometrically isomorphic to $X^*$.
\end{lem}
\begin{proof}
    We claim that $\BV[0,1]$ may be identified as a closed subspace of the space of finite $\R^2$-valued Radon measures $(\Mcal(0,1))^2$. Indeed, since $C^1_c(0,1)$ is dense in $C_0(0,1)$ and 
    \[\sup_{\substack{\phi\in C^1_c(0,1) \\ \norm{\phi}_\infty \leq 1}} \innerp{Du,\phi} = \var(u)<\infty,\]
    where $Du$ is the distributional derivative of $u$, it follows that $\phi\mapsto \innerp{Du,\phi}$ has a continuous linear extension to all of $C_0(0,1)$. Since $C_0(0,1)^* = \Mcal(0,1)$, it follows that $Du$ is a finite Radon measure. In particular, this means that the map $u\mapsto (u\, d\Leb, Du)$ from $\BV[0,1]\to (\Mcal(0,1))^2 = (C_0(0,1))^{2*}$ is an isometry when $(C_0(0,1))^{2}$ is equipped with the $\sup$-norm. Thus, $\BV[0,1]$ is a closed subspace of $(\Mcal(0,1))^2$. 

    Take $E = \{(\phi_1,\phi_2)\in C_0(0,1)\times C_c^1(0,1): \phi_1 = \phi_2'\}$ and let $Y$ be the $(C_0(0,1))^2$ closure of $E$. Then $Y$ is a closed linear subspace and it follows immediately that $\ann{Y} \supseteq \BV[0,1]$ under the identification into $(\Mcal(0,1))^2$. If $(\mu_1,\mu_2)\in \ann{Y}$, then by definition we have $D\mu_1 = \mu_2$ in the sense of distributional derivatives. Writing $u(x) = \mu_2(0,x)$, it follows that $d\mu_1 = (u+c)\, d\Leb$ for some constant $c\in \R$ and $u\in \BV[0,1]$. Therefore, $\ann{Y} = \BV[0,1]$. By standard Banach space theory (see, e.g., theorem V.2.2 in \cite{ConwayFunctionalAnalysis}), it follows that $((C_0(0,1))^2/Y)^*$ is isometrically isomorphic to $\BV[0,1]$. Separability of $X \coloneqq (C_0(0,1))^2/Y$ follows from separability of $C_0(0,1)$ and the fact that quotients of separable Banach spaces are separable Banach spaces.
\end{proof}

We shall write $[\phi_1,\phi_2]$ for the equivalence class in $X$ containing $(\phi_1,\phi_2)\in (C_0(0,1))^2$. The induced norm on $X$ is 
\[\norm{[\phi_1,\phi_2]}_X = \inf_{(\psi_1,\psi_2)\in Y} \max\paren{\norm{\phi_1-\psi_1}_\infty,\norm{\phi_2-\psi_2}_\infty}.\]

\begin{prop}\label{prop:admissibilityOfBV}
    The space $\BV[0,1]$ is admissible. If $D\subseteq B_X$ is countable and norm dense, then $(\sigma(BV[0,1],X); D)$ is an admission ticket. Additionally, if $D\subseteq L^\infty\cap B_{\BV[0,1]^*}$ is countable and $\sigma(L^\infty,L^1)$-dense, then $(\norm{\cdot}_{L^1};D)$ is an admission ticket.
\end{prop}
\begin{proof}
    The first two assertions follow immediately from \autoref{cor:dualOfSepIsAdmissible}. For the last, notice that Helly's Selection Theorem implies $B_{\BV[0,1]}$ is $\norm{\cdot}_{L^1}$-closed (and compact, even). Then, by \autoref{cor:admissibleViaSepBanachEmbedding}, the assertion follows.
\end{proof}
\begin{remark}
    Notice that $L^\infty[0,1]$ is isometrically identified with $(\BV[0,1],\norm{\cdot}_{L^1})^*$ via the $L^2$-pairing since $\BV[0,1]$ is $\norm{\cdot}_{L^1}$-dense in $L^1[0,1]$.
\end{remark}

\subsection{Random Lasota-Yorke Maps}\label{subsec:randomLYmaps}
Let $(\Omega,\F,\p)$ be a probability space, and let $\sigma:\Omega\to \Omega$ be an invertible measure preserving transformation; for the sake of simplicity, we shall assume $\sigma$ is ergodic. We will consider the random dynamical system $f:\Omega\times [0,1] \to [0,1]$, writing $f^n_{\omega} = f_{\sigma^{n-1}\omega}\circ \cdots \circ f_\omega$. Following Buzzi, we require that the map $(\omega,x)\mapsto f_\omega(x)$ is $(\F\otimes \Bor([0,1]))-\Bor([0,1])$-measurable and assume the following conditions on $f$:
\begin{enumerate}[label=(\roman*)]
    \item \textit{piecewise monotonic}: there are measurable maps $\Nt:\Omega\to\N$ and $(a_i)_{0\leq i\leq \Nt}$ such that, for a.e. $\omega$, $0 = a_0<a_1<\cdots<a_{\Nt} = 1$ and $f_\omega|_{(a_{i-1}(\omega),a_{i}(\omega))}$ is a homeomorphism onto its image for $1\leq i \leq \Nt(\omega)$;
    \item \textit{non-singular}: for a.e. $\omega$ and $1\leq i \leq \Nt(\omega)$,  $f_\omega|_{(a_{i-1}(\omega),a_{i}(\omega))}$ is absolutely continuous;
    \item \textit{covering}: for each nontrivial subinterval $I\subset [0,1]$, $\p$-a.e. $\omega\in\Omega$, there is $n_c = n_c(\omega,I) \in \N$ such that 
    \[\essinf_{x\in [0,1]}(\Lcal^n_\omega \indic_{I})>0 \text{ for all } n\geq n_c,\]
    where $\Lcal_\omega^n$ is the Perron-Frobenius operator (a.k.a. transfer operator) associated to $f_\omega^n$;
    \item $\displaystyle \lim_{K\to\infty} \int_{\Omega} \log \min\paren{\essinf_{x\in [0,1]} \abs{f'_\omega(x)},K}\, d \p(\omega) > 0$;
    \item both $\omega\mapsto\log^+\bracket{\Nt(\omega)/(\essinf_{x\in [0,1]} \abs{f'_\omega(x)})}$ and $\omega\mapsto\log^+ \var(1/\abs{f'_\omega})$ are integrable.
\end{enumerate}
\begin{remark}
    For a more detailed discussion of these conditions, we point to \cite{Buzzi1999}. A few comments: (1) Discarding $\omega$, any deterministic map $f$ satisfying (i) and (ii) is called a piecewise monotonic and non-singular (p.m.n.s.) map; (2) Buzzi explicitly asks for the measurability of $\essinf_x \abs{f'_\omega(x)}$ and $\var(1/\abs{f'_\omega})$, but this is immediate from joint measurability of $f$ and our equivalent reformulation of $\var$; (3) Transfer operators associated to RDS satisfying (i)-(v) are \textit{good transfer operators} as defined in \cite{Buzzi1999}.
\end{remark}

Items (i)-(v) define the notion of a \textit{random Lasota-Yorke} map on $[0,1]$. The transfer operator associated to $f^n_\omega$ is defined by
\[u\mapsto \Lcal^n_\omega u(x) = \sum_{y\in (f^n_\omega)^{-1}(x)} \frac{u(y)}{\abs{(f_\omega^n)'(y)}}.\]
Notice that $\Lcal^n_\omega = \Lcal_{\sigma^{n-1}\omega}\cdots \Lcal_\omega$. Moreover, the transfer operator is dual to the Koopman operator on $L^\infty[0,1]$.

\subsubsection*{Measurability of Transfer Operators}
\begin{lem}\label{lem:dualityPairingMeasurability_TransferOps}
    For all $n\geq 1$, the map $\omega\mapsto \Lcal_\omega^n$ is $(\BV[0,1],L^\infty[0,1])$ and $(\BV[0,1],X)$-measurable.
\end{lem}
\begin{proof}
    Take $u\in \BV[0,1]$ and $g\in L^\infty[0,1]$. Then, by duality with the Koopman operator,
    \[\innerp{\Lcal_\omega^n u,g} = \int_0^1 u(x)g(f^n_\omega(x))\,dx.\]
    Since $f$ is jointly measurable, it follows that the above is a measurable map of $\omega$. Now, take $[\phi_1,\phi_2]\in X$. Suppose first $\phi_2\in C_c^1(0,1)$. Then
    \[\innerp{\Lcal_\omega^n u, [\phi_1,\phi_2]} = \int_0^1 \Lcal^n_\omega u(x) \phi_1(x)\, dx + \int_0^1 \phi_2(x)\,dD(\Lcal^n_\omega u)(x) = \int_0^1 u(x)(\phi_1(f_\omega^n(x)) - \phi_2'(f_\omega^n(x)))\,dx.\]
    As before, this is a measurable map of $\omega$. If $\phi_2\notin C_c^1(0,1)$, there is a sequence of maps in $C_c^1(0,1)$ that converges in the uniform-norm to $\phi_2$. This gives us a sequence of measurable maps that converges pointwise to $\innerp{\Lcal_\omega^n u, [\phi_1,\phi_2]}$, so we are done.
\end{proof}

\begin{lem}\label{lem:weakContOfTransferOps}
    For a.e. $\omega\in \Omega$, $\Lcal_\omega\in L(\BV[0,1],\norm{\cdot}_{L^1})$ and $\Lcal_\omega\in L(\BV[0,1],\sigma(\BV[0,1],X))$.
\end{lem}
\begin{proof}
    Let $\omega\in \Omega$. That $\Lcal_\omega\in L(\BV[0,1],\norm{\cdot}_{L^1})$ is immediate from duality with the Koopman operator. For the second claim, note that $\sigma(\BV[0,1],X)$-convergence of a sequence is equivalent to $\norm{\cdot}_{L^1}$-convergence and boundedness of the total variation (see Proposition 3.13 in \cite{ambrosio2000bv}). So, if $u_n$ is a sequence that weakly-* converges to $u$, then it converges in $L^1$ to $u$. Thus, $\Lcal_\omega u_n\to \Lcal_\omega u$ in $L^1$. Moreover, $\sup_{n\geq 0} \norm{u_n}_{\BV} = M <\infty$, so $\sup_{n\geq 0} \norm{\Lcal_\omega u_n}_{\BV} \leq \norm{\Lcal_\omega}_{\BV} M < \infty$. This proves that $\Lcal_\omega$ is sequentially weak-* continuous. For the weak-* topology on the dual of a separable Banach space, sequential weak-* continuity and full weak-* continuity of a linear operator are equivalent (see Corollary V.12.8 of \cite{ConwayFunctionalAnalysis}), so we are done.
\end{proof}

\begin{prop}
    The map $\Lcal:\Omega\to L(\BV[0,1])$ satisfies condition \hyperref[ass:measurability]{$\measurable_\sigma$}. If $D\subset B_X$ is countable and $\norm{\cdot}_X$-dense, then $(\sigma(\BV[0,1]);D)\in \Mcal_{\sigma,\Lcal}$. Similarly, if $D\subset L^\infty\cap B_{\BV[0,1]^*}$ is countable and $\sigma(\norm{\cdot}_{L^\infty},\norm{\cdot}_{L^1})$-dense, then $(\norm{\cdot}_{L^1};D)\in \Mcal_{\sigma,\Lcal}$.
\end{prop}
\begin{proof}
    This follows immediately from \autoref{prop:admissibilityOfBV}, \autoref{lem:dualityPairingMeasurability_TransferOps}, and \autoref{lem:weakContOfTransferOps}.
\end{proof}

\begin{remark}
    It is worth noting that because $\BV[0,1]$ has a separable predual $X$ for which the transfer operators $\Lcal_\omega$ are weak-* continuous, it follows that the transfer operators admit a bounded preadjoint, i.e., a linear operator $T_\omega\in L(X)$ such that $T^*_\omega = \Lcal_\omega$ (in fact, $T_\omega = \Lcal_\omega^*|_{X}$). This opens up another potential approach to studying the Oseledets structure by applying the MET to the cocycle generated by $\omega\mapsto T_{\sigma^{-1}\omega}$ w.r.t. $\sigma^{-1}$ on $X$, as $\omega\mapsto T_\omega$ is indeed strongly measurable.

    To the best of our knowledge, this predual route has not been used to obtain a multiplicative ergodic theorem interpretation of \cite{Buzzi1999}. Nevertheless, we do not pursue these possibilities here, as we intend to show that the multiplicative ergodic theorem can be applied directly to transfer operator cocycles on $\BV[0,1]$. Moreover, our measurability framework does not require that our Banach space admit a predual, nor that the cocycle be weak-* continuous w.r.t. such a predual.
\end{remark}

\subsubsection*{Lyapunov Spectrum and Oseledets Spaces}
\begin{thm}\label{thm:MET_applied_to_transferOps}
    The operator cocycle $\omega\mapsto \Lcal_\omega$ on $\BV[0,1]$ is quasicompact. Take the notations of \autoref{thm:non-invMET} and \autoref{thm:semi-invMET} applied to the random transfer operator cocycle. Then we have $0 = \lambda_1 > \cdots > l_\alpha$ and that $\lambda_1$ is simple, i.e., $m_1 = 1$. For $1\leq s < r+1$, the map
    \[E_s:\Omega\to \G_{M_s}(\BV[0,1])\subset \G_{M_s}(L^1[0,1])\]
    is strongly-measurable under the $L^1[0,1]$-gap-topology. Moreover, for $2\leq s < r+2$ (where we write $\lambda_{r+1} = l_\alpha$ if $r<\infty$), the map
    \[F_{\lambda_s}:\Omega\to \G^{M_{s-1}}(\BV[0,1])\]
    is strongly-measurable under the $\BV[0,1]$-gap-topology, and $F_{\lambda_{s}}$ is almost surely $\sigma(\BV[0,1],X)$-closed. We also have the following:
    \begin{enumerate}[label=(\arabic*)]
        \item If $h \in L^1(\Omega\times [0,1])$ denotes the unique invariant density of the random dynamical system induced by $f$, then $E_1(\omega) = \innerp{h_\omega}$ almost surely.
        \item For a.e. $\omega\in \Omega$, $F_{\lambda_2}(\omega) = \BV_0[0,1] \coloneqq \curly{u\in \BV[0,1] : \int_0^1 u(x) \, dx = 0}$.
    \end{enumerate}
\end{thm}
We now prove \autoref{thm:MET_applied_to_transferOps}, which is a more detailed version of \autoref{summaryThm:transferOps_cocycle}. We shall first show $\chi_1 = 0 > \chi_2$ (adapting an argument in \cite{DS21quenchedLimThms}), which shall give us $\lambda_1 = 0$, quasicompactness, and simplicity of $\lambda_1$. We then obtain items (1) and (2). The proof is then complete once we show the measurability clauses. Throughout this proof, $\omega\in \Omega$ shall be an Oseledets regular point.

First, if we take any nonnegative non-zero function $u\in \BV[0,1]$, then by mass preservation of transfer operators, we get
\[\chi_1 \geq \limsup_{n\to\infty} \frac{1}{n}\log \norm{\Lcal^n_\omega u}_{\BV}\geq \lim_{n\to\infty} \frac{1}{n}\log \norm{\Lcal^n_\omega u}_{L^1} = 0.\]
On the other hand, by Lemma 2.1 in \cite{Buzzi1999}, $\sup_{n\geq 0}\norm{\Lcal^n_{\sigma^{-n}\omega}}_{\BV} < \infty$ for a.e. $\omega\in \Omega$. As such, by \autoref{thm:semi-invMET}, $\chi_1(\omega) \leq 0$, and hence, $\chi_1 = 0$ a.s. Now, following Dragi{\v{c}}evi{\'c} and Sedro's proof of Proposition 24 (specifically, equation (30)) in \cite{DS21quenchedLimThms}, we claim that there is an $\eta\in (0,1)$ such that, for a.e. $\omega\in \Omega$ and all $u\in \BV_0[0,1]$,
\begin{equation}\label{eq:massFreeSpaceGrowth}
    \limsup_{n\to\infty} \frac{1}{n}\log\norm{\Lcal_\omega^n u}_{\BV} \leq \log \eta < 0.
\end{equation}
Indeed, although in their setting they take Buzzi's setting in \cite{Buzzi1999} with an additional continuity assumption on $\omega\mapsto \Lcal_\omega$ as a map from $\Omega\to L(\BV[0,1])$ (the so called $\p$-continuity, due to Ma\~n\'e \cite{Mane1983}), their derivation of \eqref{eq:massFreeSpaceGrowth} (equation 30 in \cite{DS21quenchedLimThms}) only utilizes properties of \textit{good transfer operators} from \cite{Buzzi1999} and Lemma C.5 from \cite{GonzalezTokmanQuas2014}. In particular, this Lemma C.5 is applicable in Buzzi's more general setting, and as such, is applicable in our weak measurability setting. The only place where Dragi{\v{c}}evi{\'c} and Sedro require $\p$-continuity is for when they apply their equation (30) to the multiplicative ergodic theorem, so while we may not use the full power of their Proposition 24, we may use their equation 30.

Let $\e>0$. Having established \eqref{eq:massFreeSpaceGrowth}, we get for a.e. $\omega\in \Omega$, all $u\in \BV_0[0,1]$, and all $n\geq 0$ that
\[\sup_{n\geq 1} \norm{\eta^{-(1-\e)n}\Lcal_{\omega}^n u}_{\BV} < \infty.\]
Hence, by the uniform boundedness principle, we get $M_\e = M_\e(\omega)$ such that $\norm{L^n_{\omega}|_{\BV_0[0,1]}} \leq M_{\e}(\omega) \eta^{(1-\e)n}$ for all $n\geq 1$. Since $\e>0$ was arbitrarily small and $\BV_0[0,1]$ is a closed 1-codimensional subspace of $\BV[0,1]$, it follows that
\[\chi_2(\omega) \leq \limsup_{n\to\infty} \frac{1}{n}\log\norm{\Lcal_\omega^n|_{\BV_0[0,1]}}_{\BV} \leq \inf_{\e>0} \lim_{n\to\infty} \frac{1}{n}\log (M_\e(\omega)\eta^{(1-\e)n}) = \log \eta\]
for a.e. $\omega\in \Omega$, giving us quasicompactness and simplicity of $\lambda_1 = 0$. By \autoref{thm:sequentialMET_oneSided}, it follows that $F_{\chi_2}(\omega) = \BV_0[0,1]$ a.s., which gives us item (2).

That $E_1(\omega) = \innerp{h_\omega}$ follows immediately from the fact that $\Lcal_\omega h_\omega = h_{\sigma \omega}$,
\[\lim_{n\to\pm\infty}\frac{1}{n}\log \norm{h_{\sigma^{n}\omega}}_{\BV} = 0 = \lambda_1,\]
and the uniqueness clause of \autoref{thm:refined_sequentialMET_twoSided}. Logarithmic temperedness of $\omega\mapsto \norm{h_\omega}_{\BV}$ follows by Lemma 4.2 in \cite{Buzzi1999}.

\paragraph{Measurability of the Oseledets Subspaces} 
Since $L^1[0,1]$ is separable, density of $\BV[0,1]$ in $L^1[0,1]$, \autoref{lem:dualityPairingMeasurability_TransferOps}, and \autoref{appThm:pettisMeasurabilityTheorem} imply that $\Lcal_\omega$ is strongly measurable as an operator on $L^1$. For $1\leq s < r+1$, the construction of $E_s(\omega)$ in \autoref{thm:semi-invMET} is as the $\G(\BV[0,1])$-limit of the spans of a sequence of bases of $L^1[0,1]$-strongly-measurable maps. If $E_{n,s}(\omega)$ are spans, then by continuity of the span map it follows that $E_{n,s}:\Omega\to \G_{M_s}(L^1[0,1])$ is $\G(L^1[0,1])$-strongly-measurable for all $n\geq 1$. Since the $\G_{M_s}(\BV[0,1])$-gap-topology is stronger than the $\G_{M_s}(L^1[0,1])$-gap-topology, it follows that $E_s:\Omega\to \G_{M_s}(\BV[0,1])\subset \G_{M_s}(L^1[0,1])$ is $\G(L^1[0,1])$-strongly-measurable.

The argument actually proceeds similarly for the slow spaces, except we first move to the dual of $\BV[0,1]$. In particular, identifying $X$ as a subspace of $(\BV[0,1])^*$, it follows by \autoref{lem:weakContOfTransferOps} and basic Banach space theory (see exercise VI.1.11 of \cite{ConwayFunctionalAnalysis}) that $\Lcal_\omega^*X\subseteq X$ for all $\omega$. Now, reasoning as before, in the proof of \autoref{thm:non-invMET}, (since $X$ is $\Lcal_\omega^*$-invariant) we obtain for $2\leq s < r+2$ a $\G(X)$-strongly-measurable sequence of maps $L_{n,s}:\Omega\to \G_{M_{s-1}}(X)\subseteq \G_{M_{s-1}}(\BV[0,1]^*)$ that converges to $L_s(\omega) = \ann{F_{\lambda_s}(\omega)}$ in the $\G(\BV[0,1]^*)$-gap-topology. By \autoref{lem:equivOfGap_and_Hausdorff}, $F_{n,s}(\omega) = \preann{L_{n,s}(\omega)}$ is $\G(\BV[0,1])$-strongly-measurable and converges to $F_{\lambda_s}(\omega)$ in the $\G(\BV[0,1])$-gap-topology. Hence, $F_{\lambda_s}:\Omega\to \G(\BV[0,1])$ is $\G(\BV[0,1])$-strongly-measurable. Of course, since $X$ is a dual-norm-closed subspace of $\BV[0,1]^*$, it follows that $L_{s}(\omega)\subseteq X$, which implies $F_{\lambda_s}(\omega) = \preann{L_{s}(\omega)}$ is weak-* closed. Strong-measurability implies $F_{\lambda_s}$ is a.s. separably-valued.

\begin{appendix}
\section{Fine Properties of Coarse Topologies}
\subsection{Closedness of \texorpdfstring{$BL_\tau(V)$}{BL tau}}
Let $\tau$ be a vector space topology on $(\B,\norm{\cdot})$ coarser than the $\norm{\cdot}$-topology. For a topological vector space $V$, let $BL_\tau(V)$ be defined as in \autoref{sec:MET}.
\begin{lem}\label{appLem:closednessOfBLtau}
    Let $(V,\norm{\cdot}_V)$ be a normed vector space, and let $\tau_V$ be any vector space topology on $V$ coarser than the $\norm{\cdot}_V$ topology. Then
    \[L((\B,\norm{\cdot}),(V,\norm{\cdot}_V))\cap BL_\tau(V,\tau_V)\]
    is closed in the $L((\B,\norm{\cdot}),(V,\norm{\cdot}_V))$-operator norm topology.
\end{lem}
\begin{proof}
    For the sake of brevity, we write $L(\B,V) = L((\B,\norm{\cdot}),(V,\norm{\cdot}_V))$. Let $T_n\in L(\B,V)\cap BL_\tau(V,\tau_V)$ be a sequence of operators converging in the operator norm topology to some $T\in L(\B,V)$. Let $x_\alpha\in B_{\B}$ be a net converging to some $x\in B_\B$ w.r.t. $\tau$. We need to show $Tx_\alpha\to Tx$ w.r.t. $\tau_V$. Let $U$ be a $\tau_V$-open neighborhood around $0$. By the continuity of addition, there is a $\tau_V$-open neighborhood $W$ around $0$ such that $W+W\subset U$. Since $\tau_V$ is coarser than the $\norm{\cdot}_V$-topology, $W$ is $\norm{\cdot}_V$-open. Thus, we can fix an $n$ large enough so that $(T-T_n)(x_\alpha-x) \in W$. Since $T_n\in BL_\tau(V,\tau_V)$, it follows that $T_n(x_\alpha-x)\in W$ for all $\alpha$ large enough. Therefore,
    \[T(x_\alpha-x) = T_n(x_\alpha-x) + (T-T_n)(x_\alpha-x) \in W + W \subseteq U.\]
    That is, $Tx_\alpha$ eventually enters and stays in every $\tau_V$-open neighborhood of $Tx$. Hence, $T\in BL_\tau(V,\tau_V)$.
\end{proof}

\subsection{Lower Semi-Continuity}\label{appSubSec:LSC}
\begin{lem}\label{appLem:LSCimpliesSupOverDense}
    Let $(X,\tau)$ be a topological space. Suppose $f:X\to \R\cup\{\pm\infty\}$ is lower semi-continuous (lsc) and let $D\subset X$ be $\tau$-dense. Then
    \[\sup_{x\in X} f(x) = \sup_{x\in D} f(x).\]
\end{lem}

\begin{lem}\label{appLem:LianLuVolLSC}
    Let $(X_1,\tau_1)$ and $(X_2,\tau_2)$ be topological vector spaces, and let $\norm{\cdot}_2$ be any lower semi-continuous norm on $X_2$. For $q\in \N$, let $\vol_q:X_2^q\to \R$ denote the $q$-dimensional parallelepiped volume given by $\norm{\cdot}_2$. If $Y\subseteq X_1$ is a $\tau_1$-closed set and $A:X_1\to X_2$ is a linear map such that $A|_{Y}:(Y,\tau_{1}|_{Y})\to (X_2,\tau_2)$ is continuous, then the map from $X_1^q\to \R$ defined by
    \[(x_1,\ldots,x_q)\mapsto \vol_q(Ax_1,\ldots,Ax_q)\]
    is lower semi-continuous on $\{x\in  Y:\norm{Ax}_2\leq M\}^q$ w.r.t. $\tau_1$ for all $M\geq 0$.
\end{lem}
\begin{proof}
    First, we note that the set
    \[C_M = \{x\in Y:\norm{Ax}_2\leq M\}\]
    is $\tau_1$-closed for all $M\geq 0$ since it is the continuous pre-image of $M B_{X_2}$ where
    \[B_{X_2} = \curly{x\in X_2 : \norm{x}_2 \leq 1},\]
    which is $\tau_2$-closed by assumption. Therefore, the map $x\mapsto \norm{Ax}_2$ is lsc on $Y$ w.r.t. $\tau_1$. This proves the lemma for $q=1$, providing a base case from which we may proceed inductively. Assume now $(x_1,\ldots,x_k)\mapsto \vol_k(Ax_1,\ldots,Ax_k)$ is lower semi-continuous for each $k\leq q-1$. We shall prove it for $q$. 

    Firstly, if $\vol_q(Ax_1,\ldots,Ax_q) = 0$, then we are done by non-negativity of volume, so we may assume this is not the case. Consider a net of vectors $(x_{1,\alpha},\ldots,x_{q,\alpha})\in C_{M}^q$ that converges to some vector $(x_1,\ldots,x_q)\in C_M^q$ in the product topology $\tau_1^q$. Now, this implies
    \[(Ax_{1,\alpha},\ldots,Ax_{q,\alpha})\overset{\tau_2}{\to} (Ax_1,\ldots,Ax_q).\]
    Let $w_\alpha\in \innerp{Ax_{2,\alpha},\ldots,Ax_{q,\alpha}}$ denote any minimizer of 
    \[\norm{Ax_{1,\alpha} - w}_2\]
    over $w\in \innerp{Ax_{2,\alpha},\ldots,Ax_{q,\alpha}}$. Then we can write
    \[w_\alpha = t_{2,\alpha}Ax_{2,\alpha} + \cdots + t_{q,\alpha}A x_{q,\alpha}\]
    for some $t_{2,\alpha},\ldots,t_{q,\alpha}\in \R$. Our assumption tells us that $\norm{Ax_{k,\alpha}}$ is bounded away from 0 for large enough $\alpha$. Moreover, our assumption of non-zero volume and the induction give us that 
    \[\vol_{q-1}(Ax_{2,\alpha},\ldots,Ax_{q,\alpha})\]
    is bounded below by $\frac{1}{2}\vol_{q-1}(Ax_2,\ldots,Ax_{q}) > 0$ for all large enough $\alpha$. Thus, the product formula for the volumes implies that the angles between the spaces $\curly{\innerp{Ax_{k,\alpha}}}_{2\leq k\leq q}$ are uniformly bounded away from 0 for large enough $\alpha$, and so $\abs{t_{k,\alpha}}$ for all $2\leq k\leq q$ is uniformly bounded for large enough $\alpha$, which implies by the Heine-Borel property of $\R^{q-1}$ that there is a subnet $\beta$ such that $t_{k,\beta}\to t_k\in \R$, and so
    \[w_\beta\to w = t_2Ax_2+\cdots+t_q Ax_q.\]
    It is no issue to assume $\beta$ is a subnet such that
    \[\lim_{\beta} d(Ax_{1,\beta},\innerp{Ax_{2,\beta},\ldots,Ax_{q,\beta}}) = \liminf_{\alpha} d(Ax_{1,\alpha},\innerp{Ax_{2,\alpha},\ldots,Ax_{q,\alpha}}).\]
    Hence,
    \[d(Ax_1,\innerp{Ax_2,\ldots, Ax_q}) \leq \norm{Ax_1-w} \leq \liminf_\alpha d(Ax_{1,\alpha},\innerp{Ax_{2,\alpha},\ldots,Ax_{q,\alpha}}).\]
    As the product of nonnegative lsc functions is lsc, we are done.
\end{proof}

\section{The Subadditive Ergodic Theorem}
Let $(\Omega, \F, \p)$ be a probability space and let $\sigma:\Omega\to\Omega$ be a measure preserving transformation (mpt). Let $\Ical = \{E\in\F:\sigma^{-1} E = E\}$ denote the sub-$\sigma$-algebra of invariant sets.

\begin{thm}[Kingman's Subadditive Ergodic Theorem \cite{kingman1973subadditive}]\label{appThm:KingmanSET}
    Let $\{f_{n}\}_{n\geq 0}$ be a subadditive sequence of measurable maps (i.e. $f_{n+m}(\omega)\leq f_n(\omega) + f_{m}(\sigma^n \omega)$) such that $f_{1}^{+}\in L^1(\Omega,\p)$. Then there is a $\sigma$-invariant measurable map $f:\Omega\to \R\cup\{-\infty\}$ such that $\frac{1}{n}f_{n}\to f$ almost surely as $n\to\infty$. Moreover,
    \[f = \inf_{n\geq 1} \frac{1}{n}\E[f_n|\Ical] \text{ and } \E[f] = \inf_{n\geq 1} \frac{1}{n}\E[f_n].\]
    If $\E[f] > -\infty$, then $\frac{1}{n}f_n\to f$ as $n\to\infty$ in $L^1$ as well.
\end{thm}
\begin{remark}
    The assumptions that $f_1^+\in L^1$ and $\E[f] > -\infty$ imply $f,f_n\in L^1$ for all $n\geq 1$.
\end{remark}
For the rest of this appendix, we shall assume $(f_n)_{n\geq1}$ is a subadditive family such that $f_1^+\in L^1$.

\begin{cor}\label{appCor:invertibleSET}
    Suppose $\sigma:\Omega\to\Omega$ is invertible. Then we have that
    \[\lim_{n\to\infty} \frac{1}{n}f_n(\sigma^{-n}\omega) = f(\omega)\]
    for almost every $\omega\in \Omega$.
\end{cor}
\begin{proof}
    The proof is straightforward and details are left to the reader.
\end{proof}

For our purposes, we shall need a refinement of the subadditive ergodic theorem that uniformly tempers growth rates along a whole orbit. To do this, we use the following due to Kingman \cite{kingman1968ergodic}.
\begin{thm}[Kingman's Subadditive Decomposition Theorem \cite{kingman1968ergodic}]\label{appThm:SET_decomp}
    Assume that $\E[f] > -\infty$. Then there is a map $\vphi\in L^1$ and a purely subadditive family $(E_n)_{n\geq 1}$, i.e., $E_n$ is subadditive and nonnegative, such that $\E[\vphi|\Ical] = f$ and
    \[f_n(\omega) = S_n[\vphi](\omega) + E_n(\omega),\]
    where $S_n[\vphi](\omega) = \vphi(\omega)+\cdots+\vphi(\sigma^{n-1}\omega)$.
\end{thm}

\begin{thm}[A Tempered Subadditive Ergodic Theorem]\label{appThm:temperedSET}
    Then there is an invariant measurable set $\Gamma\subseteq \Omega$ of full probability such that for all $\omega\in \Gamma$,
    \[\lim_{n,k\to\infty} \frac{1}{n+k}\abs{f_k(\sigma^{n}\omega)-kf(\omega)} = 0\]
    if $f(\omega) > -\infty$, or
    \[\limsup_{n,k\to\infty} \frac{f_{k}(\sigma^n\omega) + k \Delta}{n+k} \leq 0\]
    for all $\Delta \in \R$ if $f(\omega) = -\infty$. If $\sigma$ is invertible, then the same statements hold writing $\sigma^{-n}$ instead of $\sigma^{n}$.
\end{thm}
\begin{proof}
    We shall only consider the non-invertible case, as the invertible case is a straightforward adaptation of the proof. Notice also that, for both results, we may actually reduce to the case of $\E[f] > -\infty$. Indeed, for $\omega\in \{f>-\infty\}$, cover $\{f>-\infty\}$ by a countable union of $\sigma$-invariant sets $\{f>-i\}_{i\geq 1}$. For $\omega\in \{f = -\infty\}$, consider instead the subadditive families $g_k^\Delta = \max(-k\Delta, f_k)$ for $\Delta \in \N$. Thus, we assume $\E[f] > -\infty$. The details are straightforward and left to the reader.
    
    By \autoref{appThm:SET_decomp}, there is an integrable map $\vphi\in L^1$ such that $\E[\vphi|\Ical]=f$ and a purely subadditive family $E_n:\Omega\to \R$ such that $f_n = S_n[\vphi]+E_n$. 
    For the additive term, Birkhoff's ergodic theorem gives that for all $\delta'>0$, there is an integer $N = N(\omega,\delta') > 0$ such that 
    \[\abs{S_{n}[\vphi]-nf}\leq n\delta'\]
    for all $n\geq N(\omega,\delta')$. In particular, for $n,k\in \N$, notice that
    \[S_{k}[\vphi](\sigma^{n}\omega) = S_{n+k}[\vphi](\omega) - S_n[\vphi](\omega),\]
    which means
    \[\abs{S_k[\vphi](\sigma^{n}\omega)-kf(\omega)}\leq 2(n+k)\delta'\]
    for all $n\geq N(\omega,\delta')$ and $k\geq 1$.

    We now consider the purely subadditive error term. Fix an integer $1\leq K \leq k$ and let $0\leq j \leq K-1$. Performing Euclidean division, we may write $k-j = q_j K + r_j$ for some integers $q_j\geq 0$ and $0\leq r_j \leq K-1$. For convenience, we write $E_0 = 0$. By subadditivity,
    \[E_{k}(\omega) \leq E_{j}(\omega) + E_{q_jK}(\sigma^j\omega) + E_{r_j}(\sigma^{j+q_jK} \omega),\]
    By subadditivity (again) and non-negativity of $(E_n)_n$, uniqueness of quotients and remainders from the Euclidean algorithm, we get after substituting $\omega\mapsto \sigma^n \omega$ that
    \begin{align*}
        \frac{1}{n+k} E_{k}(\sigma^{n}\omega)\leq \frac{1}{n+k}\sum_{j=0}^{n+k-1}\frac{1}{K}E_K(\sigma^{j}\omega) + \frac{1}{n+k}\frac{1}{K}\sum_{j=0}^{K-1} E_j(\sigma^{n}\omega) + E_{r_j}(\sigma^{n+j+q_jK}\omega).
    \end{align*}
    Then by the Birkhoff ergodic theorem, for $n$ and $k$ large enough, the right summand in the upper bound is bounded above by $\delta'$, while the left summand is bounded  above by $\frac{1}{K} \E[E_K |\I] + \delta'$. By the subadditive ergodic theorem, we may fix $K$ large enough so that $\E[E_K|\Ical]\leq K\delta'$. Combining these estimates, we get
    \[\frac{1}{n+k}E_{k}(\sigma^{n}\omega)\leq 3\delta'\]
    for $n,k$ large. If we take $0<\delta'<\delta/6$ and combine the above estimate with the estimate for the additive term, we get
    \[\abs{f_k(\sigma^{n}\omega)-kf(\omega)}\leq \abs{S_k[\vphi](\sigma^n\omega)-kf(\omega)} + E_{k}(\sigma^{n}\omega)\leq (n+k)\delta\]
    for all $n,k$ large enough. Thus, the desired limit is attained.
\end{proof}

We note that the above generalizes the following ergodic theorem due to Lee \cite{Lee2024} to general measure preserving systems (Lee requires the underlying probability space to be Lebesgue).
\begin{thm}[Kingman's Theorem for Balanced Intervals, c.f. \cite{Lee2024}]\label{appThm:balancedKingmanSET}
    Suppose $\sigma:\Omega\to\Omega$ is invertible. Then we have that
    \[\lim_{n\to\infty} \frac{1}{2n}f_{2n}(\sigma^{-n}\omega) = f(\omega)\]
    for almost every $\omega\in \Omega$.
\end{thm}

\section{Miscellaneous Measure Theory and Analysis}\label{app:MeasureTheory}
\subsection{Functional Analysis}

Let $(X,\tau)$ be a topological vector space.

\begin{lem}\label{appLem:hereditaryWeak*SeparabilityOfDual}
    Suppose $X$ is a separable and metrizable LCTVS. Then every subset of $X^*$ is weak-* separable.
\end{lem}
\begin{proof}
    Since $X$ is locally convex and metrizable, its topology is generated by a countable increasing family of seminorms $\{p_n\}_{n\geq 1}$. Let $U_n = \{x\in X : p_n(x)\leq 1\}$. By the Alaoglu-Bourbaki theorem, the set
    \[U^*_n = \curly{f\in X^*:\sup_{x\in U_n} \abs{f(x)} \leq 1}\]
    is weak-* compact.

    Now, fix $n\geq 1$ and let $\{x_k\}_{k\geq 1}$ be a dense set of points in $U_n$. Define the metric
    \[\rho_n(f,g) = \sum_{k=1}^\infty \frac{1}{2^k} \abs{\innerp{x_k,f-g}}\]
    for $f,g\in U_n^*$. It is easy to see that $U_n^*$ under the weak-* topology continuously embeds in $U_n^*$ with the topology generated by $\rho_n$. Since $U_n^*$ is weak-* compact, it must be the case that $\rho_n$ generates the weak-* subspace topology on $U_n^*$.

    Now, it follows from the definition of continuity that
    \[X^* = \bigcup_{k,n\in \N} k U_n^*.\]
    We have shown that $kU_n^*$ is metrizable for all $k,n\in \N$, and since the countable union of hereditarily separable spaces is hereditarily separable, it follows that $X^*$ is hereditarily separable. That is, every subset of $X^*$ is weak-* separable.
\end{proof}

\subsection{Pettis Measurability}
Let $(\Omega,\F,\p)$ be a probability space. The following provides a measurability criterion for maps in a metric space. It is a generalization of the Pettis measurability theorem to abstract metric spaces, and the proof is an adapted version of the proof of Pettis' theorem in \cite{Pettis1938}.
\begin{defn}\label{appDefn:strong_meas}
    Let $(X,\tau)$ be a Hausdorff topological space and $f:\Omega\mapsto X$ a map. We say $f$ is \textit{strongly measurable} if $f$ is the a.s. limit of a sequence of simple functions.
\end{defn}
\begin{defn}\label{appDefn:separablyValued}
    Let $(X,\tau)$ be a topological space and $f:\Omega\to X$ a map. We say $f$ is \textit{separably-valued} if $f(\Omega)$ is separable. We say $f$ is \textit{almost surely separably-valued} if there is a measurable null set $N\subseteq \Omega$ such that $f(\Omega\setminus N)$ is separable.
\end{defn}

\begin{defn}\label{appDefn:weak_meas}
    Let $(X,\tau)$ be a Hausdorff topological vector space and $f:\Omega\mapsto X$ a map. We say that $f$ is \textit{weakly measurable} if, for all linear functionals $\ell\in (X,\tau)^*$, the map $\omega\mapsto \innerp{f(\omega),\ell}$ from $\Omega\mapsto \R$ is $\F-\Bor(\R)$ measurable.
\end{defn}

\begin{thm}[Pettis Measurability Theorem \cite{Pettis1938}]\label{appThm:pettisMeasurabilityTheorem}
    Let $X$ be a Banach space and $f:\Omega\to \B$ a map. Then $f$ is strongly measurable if and only if $f$ is almost surely separable valued and weakly measurable.
\end{thm}

\begin{defn}\label{appDefn:weak_and_strong_OpMeas}
    Let $(X,\tau_X)$ and $(Y,\tau_Y)$ be topological vector spaces and $A:\Omega\mapsto L(X,Y)$ a $\tau_X-\tau_Y$ continuous operator-valued map. We say $A$ is \textit{weakly measurable} if, for all fixed $x\in X$, the map $f(\omega) = A_\omega x$ from $\Omega\to Y$ is weakly measurable. Similarly, we say $A$ is \textit{strongly measurable} if, for all fixed $x\in X$, the map $f(\omega) = A_\omega x$ from $\Omega\to Y$ is strongly measurable.
\end{defn}

\subsection{Measurable Multifunctions}
Most of the material here may be found in Chapter III of \cite{CastaingValadier2006}. Let $(X,d)$ be a metric space. A multifunction is a map $\Psi:\Omega\to 2^X$, where $2^X$ denotes the powerset of $X$. If $S\subseteq X$ is a set and $\Psi$ a multifunction, then we define
\[\Psi^{-}(S) = \curly{\omega\in \Omega: \Psi(\omega) \cap S\neq \varnothing}.\]
\begin{defn}\label{appDefn:measurableMultifunction}
    A multifunction $\Psi:\Omega\to 2^X$ is \textit{weakly measurable} if, for all open sets $U\subset X$, the set $\Psi^-(U)\subset \Omega$ is measurable.
\end{defn}
\begin{thm}\label{appThm:KRN_measSelection}
    Suppose $(X,d)$ is complete and separable and $\Psi$ is a closed-valued and non-empty multifunction. Then, if $\Psi$ is weakly measurable, it admits a Borel-measurable selection.
\end{thm}
\begin{cor}[Corollary 6.16 in \cite{LianLu2010}]\label{appCor:LianLuMeasSelection}
    Suppose $(X,d)$ is complete and separable. If the multifunction $\Psi:\Omega\to 2^X$ is weakly measurable and non-empty, then its closure $\omega\mapsto \overline{\Psi}(\omega)\coloneqq \overline{\Psi(\omega)}$ is weakly measurable, and thus admits a Borel-measurable selection.
\end{cor}

\end{appendix}

\section*{Declaration of Generative AI Use}
The work presented here is the result of human-led efforts, with assistance from large language models (LLMs) in some specific capacities listed below. In particular, for the first four sections of this paper, LLMs provided no novel mathematical content. All the results and proofs therein not otherwise cited were developed by the author, with LLMs assisting with proof verification, exposition, and copyediting issues such as revising expository and mathematical typos, grammar, and readability.

Regarding Sections \ref{sec:MET}, \ref{sec:applications}, and the appendices, LLMs assisted with certain mathematical arguments used in proofs. In particular, an LLM supplied the application of the bipolar theorem in \autoref{prop:AdmissibleEmbedding}. An LLM also supplied the proof of \autoref{appLem:closednessOfBLtau}. Lastly, an LLM supplied the portion of the proof of \autoref{appThm:temperedSET} that estimated the purely subadditive error term.

Besides those listed above, LLMs made no other mathematical contributions to the present paper. All arguments supplied by LLMs were independently verified. The entire manuscript, including any argument provided by an LLM, was written by the author. The author takes full responsibility for correctness of the present work.

\printbibliography

@book{Akhmerov1992,
  author = {Akhmerov, R. R. and Kamenskii, M. I. and Potapov, A. S. 
            and Rodkina, A. E. and Sadovskii, B. N.},
  title = {Measures of noncompactness and condensing operators},
  series = {Operator Theory: Advances and Applications},
  volume = {55},
  publisher = {Birkh\"{a}user},
  year = {1992}
}

@article {Blumenthal2016,
    AUTHOR = {Blumenthal, Alex},
     TITLE = {A volume-based approach to the multiplicative ergodic theorem
              on {B}anach spaces},
   JOURNAL = {Discrete Contin. Dyn. Syst.},
  FJOURNAL = {Discrete and Continuous Dynamical Systems},
    VOLUME = {36},
      YEAR = {2016},
    NUMBER = {5},
     PAGES = {2377--2403},
   MRCLASS = {37H15 (37L30 46B07)},
  MRNUMBER = {3485402},
MRREVIEWER = {Ricardo\ F.\ Vila Freyer},
       DOI = {10.3934/dcds.2016.36.2377},
}

@article {BlumenthalMorris2015,
    AUTHOR = {Blumenthal, Alex and Morris, Ian D.},
     TITLE = {Characterization of dominated splittings for operator cocycles
              acting on {B}anach spaces},
   JOURNAL = {J. Differential Equations},
  FJOURNAL = {Journal of Differential Equations},
    VOLUME = {267},
      YEAR = {2019},
    NUMBER = {7},
     PAGES = {3977--4013},
   MRCLASS = {37D30 (37H15 46B20)},
  MRNUMBER = {3959484},
MRREVIEWER = {Mauricio\ Poletti},
       DOI = {10.1016/j.jde.2019.04.033},
}

@article {BlumenthalYoung2017,
    AUTHOR = {Blumenthal, Alex and Young, Lai-Sang},
     TITLE = {Entropy, volume growth and {SRB} measures for {B}anach space
              mappings},
   JOURNAL = {Invent. Math.},
  FJOURNAL = {Inventiones Mathematicae},
    VOLUME = {207},
      YEAR = {2017},
    NUMBER = {2},
     PAGES = {833--893},
   MRCLASS = {37H15 (37C40)},
  MRNUMBER = {3595937},
MRREVIEWER = {Zhiming\ Li},
       DOI = {10.1007/s00222-016-0678-0},
}

@article {BochiGourmelon2009,
    AUTHOR = {Bochi, Jairo and Gourmelon, Nicolas},
     TITLE = {Some characterizations of domination},
   JOURNAL = {Math. Z.},
  FJOURNAL = {Mathematische Zeitschrift},
    VOLUME = {263},
      YEAR = {2009},
    NUMBER = {1},
     PAGES = {221--231},
   MRCLASS = {37D30},
  MRNUMBER = {2529495},
MRREVIEWER = {Alexander\ Arbieto},
       DOI = {10.1007/s00209-009-0494-y},
}

@book {Bollobas1999,
    AUTHOR = {Bollob\'as, B\'ela},
     TITLE = {Linear analysis},
    SERIES = {Cambridge Mathematical Textbooks},
      NOTE = {An introductory course},
 PUBLISHER = {Cambridge University Press, Cambridge},
      YEAR = {1990},
     PAGES = {xii+240},
      ISBN = {0-521-65577-3},
   MRCLASS = {46-01 (47-01)},
  MRNUMBER = {1087297},
MRREVIEWER = {Ioana\ Cior\u anescu},
}

@book {CastaingValadier2006,
    AUTHOR = {Castaing, C. and Valadier, M.},
     TITLE = {Convex analysis and measurable multifunctions},
    SERIES = {Lecture Notes in Mathematics},
    VOLUME = {580},
 PUBLISHER = {Springer-Verlag, Berlin-New York},
      YEAR = {1977},
     PAGES = {vii+278},
   MRCLASS = {46G99 (26A51 28A05 49A50 54C60)},
  MRNUMBER = {467310},
MRREVIEWER = {Vladimir\ L.\ Levin},
}

@article {FroylandLloydQuas2010,
    AUTHOR = {Froyland, Gary and Lloyd, Simon and Quas, Anthony},
     TITLE = {Coherent structures and isolated spectrum for
              {P}erron-{F}robenius cocycles},
   JOURNAL = {Ergodic Theory Dynam. Systems},
  FJOURNAL = {Ergodic Theory and Dynamical Systems},
    VOLUME = {30},
      YEAR = {2010},
    NUMBER = {3},
     PAGES = {729--756},
   MRCLASS = {37H15 (37C30 37D20 37M25)},
  MRNUMBER = {2643709},
MRREVIEWER = {Hans\ Henrik\ Rugh},
       DOI = {10.1017/S0143385709000339},
}

@article {FroylandLloydQuas2013,
    AUTHOR = {Froyland, Gary and Lloyd, Simon and Quas, Anthony},
     TITLE = {A semi-invertible {O}seledets theorem with applications to
              transfer operator cocycles},
   JOURNAL = {Discrete Contin. Dyn. Syst.},
  FJOURNAL = {Discrete and Continuous Dynamical Systems},
    VOLUME = {33},
      YEAR = {2013},
    NUMBER = {9},
     PAGES = {3835--3860},
   MRCLASS = {37H15 (37A30 37L55)},
  MRNUMBER = {3038042},
MRREVIEWER = {Luciana\ A.\ Alves},
       DOI = {10.3934/dcds.2013.33.3835},
}

@article {QuasThieullenZarrabi2019,
    AUTHOR = {Quas, Anthony and Thieullen, Philippe and Zarrabi, Mohamed},
     TITLE = {Explicit bounds for separation between {O}seledets subspaces},
   JOURNAL = {Dyn. Syst.},
  FJOURNAL = {Dynamical Systems. An International Journal},
    VOLUME = {34},
      YEAR = {2019},
    NUMBER = {3},
     PAGES = {517--560},
   MRCLASS = {37L30 (37H15 47A35)},
  MRNUMBER = {3980636},
MRREVIEWER = {Rui\ Zou},
       DOI = {10.1080/14689367.2019.1571562},
}

@article {GonzalezTokmanQuas2014,
    AUTHOR = {Gonz\'alez-Tokman, Cecilia and Quas, Anthony},
     TITLE = {A semi-invertible operator {O}seledets theorem},
   JOURNAL = {Ergodic Theory Dynam. Systems},
  FJOURNAL = {Ergodic Theory and Dynamical Systems},
    VOLUME = {34},
      YEAR = {2014},
    NUMBER = {4},
     PAGES = {1230--1272},
   MRCLASS = {47A35 (37A30 37H15)},
  MRNUMBER = {3227155},
MRREVIEWER = {Wojciech\ Bartoszek},
       DOI = {10.1017/etds.2012.189},
}

@article {GonzalezTokmanQuas2015,
    AUTHOR = {Gonz\'alez-Tokman, Cecilia and Quas, Anthony},
     TITLE = {A concise proof of the multiplicative ergodic theorem on
              {B}anach spaces},
   JOURNAL = {J. Mod. Dyn.},
  FJOURNAL = {Journal of Modern Dynamics},
    VOLUME = {9},
      YEAR = {2015},
     PAGES = {237--255},
   MRCLASS = {37A30 (37D30 37H15 37L55)},
  MRNUMBER = {3400385},
MRREVIEWER = {Ricardo\ F.\ Vila Freyer},
       DOI = {10.3934/jmd.2015.9.237},
}

@book {Kato1995,
    AUTHOR = {Kato, Tosio},
     TITLE = {Perturbation theory for linear operators},
    SERIES = {Die Grundlehren der mathematischen Wissenschaften},
    VOLUME = {Band 132},
 PUBLISHER = {Springer-Verlag New York, Inc., New York},
      YEAR = {1966},
     PAGES = {xix+592},
   MRCLASS = {47.00 (47.48)},
  MRNUMBER = {203473},
MRREVIEWER = {L.\ de Branges},
}

@article {LianLu2010,
    AUTHOR = {Lian, Zeng and Lu, Kening},
     TITLE = {Lyapunov exponents and invariant manifolds for random
              dynamical systems in a {B}anach space},
   JOURNAL = {Mem. Amer. Math. Soc.},
  FJOURNAL = {Memoirs of the American Mathematical Society},
    VOLUME = {206},
      YEAR = {2010},
    NUMBER = {967},
     PAGES = {vi+106},
   MRCLASS = {37H15 (34D08 37A55 37C45 37D10 37H10 37L25)},
  MRNUMBER = {2674952},
MRREVIEWER = {Bj\"orn\ Schmalfuss},
       DOI = {10.1090/S0065-9266-10-00574-0},
}

@incollection {Mane1983,
    AUTHOR = {Ma\~n\'e, Ricardo},
     TITLE = {Lyapounov exponents and stable manifolds for compact
              transformations},
 BOOKTITLE = {Geometric dynamics ({R}io de {J}aneiro, 1981)},
    SERIES = {Lecture Notes in Math.},
    VOLUME = {1007},
     PAGES = {522--577},
 PUBLISHER = {Springer, Berlin},
      YEAR = {1983},
   MRCLASS = {58F15},
  MRNUMBER = {730286},
MRREVIEWER = {Charles\ C.\ Conley},
       DOI = {10.1007/BFb0061433},
}

@article {Oseledets1968,
    AUTHOR = {Oseledets, V. I.},
     TITLE = {A multiplicative ergodic theorem. {C}haracteristic {L}japunov exponents of dynamical systems},
   JOURNAL = {Trudy Moskov. Mat. Ob\v s\v c.},
  FJOURNAL = {Trudy Moskovskogo Matemati\v ceskogo Ob\v s\v cestva},
    VOLUME = {19},
      YEAR = {1968},
     PAGES = {179--210},
      ISSN = {0134-8663},
   MRCLASS = {28.70 (34.00)},
  MRNUMBER = {240280},
MRREVIEWER = {J\'ozsef\ Sz\H ucs},
}

@book {Pietsch1987,
    AUTHOR = {Pietsch, A.},
     TITLE = {Eigenvalues and {$s$}-numbers},
    SERIES = {Mathematik und ihre Anwendungen in Physik und Technik
              [Mathematics and its Applications in Physics and Technology]},
    VOLUME = {43},
 PUBLISHER = {Akademische Verlagsgesellschaft Geest \& Portig K.-G.,
              Leipzig},
      YEAR = {1987},
     PAGES = {360},
      ISBN = {3-321-00012-1},
   MRCLASS = {47B10 (45-02 46-02 47-02)},
  MRNUMBER = {917067},
MRREVIEWER = {Hermann\ K\"onig},
}

@article {Ullrich2024,
    AUTHOR = {Ullrich, Mario},
     TITLE = {Inequalities between s-numbers},
   JOURNAL = {Adv. Oper. Theory},
  FJOURNAL = {Advances in Operator Theory},
    VOLUME = {9},
      YEAR = {2024},
    NUMBER = {4},
     PAGES = {Paper No. 82, 7},
   MRCLASS = {47B06 (46B50 47B01)},
  MRNUMBER = {4804930},
MRREVIEWER = {Hermann\ K\"onig},
       DOI = {10.1007/s43036-024-00386-x},
}

@article {Raghunathan1979,
    AUTHOR = {Raghunathan, M. S.},
     TITLE = {A proof of {O}seledec's multiplicative ergodic theorem},
   JOURNAL = {Israel J. Math.},
  FJOURNAL = {Israel Journal of Mathematics},
    VOLUME = {32},
      YEAR = {1979},
    NUMBER = {4},
     PAGES = {356--362},
   MRCLASS = {60B15 (22D40 28D99 58F11)},
  MRNUMBER = {571089},
MRREVIEWER = {Yves\ Guivarc'h},
       DOI = {10.1007/BF02760464},
}

@article {Ruelle1979,
    AUTHOR = {Ruelle, David},
     TITLE = {Ergodic theory of differentiable dynamical systems},
   JOURNAL = {Inst. Hautes \'Etudes Sci. Publ. Math.},
  FJOURNAL = {Institut des Hautes \'Etudes Scientifiques. Publications
              Math\'ematiques},
    NUMBER = {50},
      YEAR = {1979},
     PAGES = {27--58},
      ISSN = {0073-8301,1618-1913},
   MRCLASS = {58F18 (58F15 60G99 82A05)},
  MRNUMBER = {556581},
MRREVIEWER = {L.\ A.\ Bunimovich},
}

@article {Ruelle1982,
    AUTHOR = {Ruelle, David},
     TITLE = {Characteristic exponents and invariant manifolds in {H}ilbert
              space},
   JOURNAL = {Ann. of Math. (2)},
  FJOURNAL = {Annals of Mathematics. Second Series},
    VOLUME = {115},
      YEAR = {1982},
    NUMBER = {2},
     PAGES = {243--290},
   MRCLASS = {58F15 (58F11)},
  MRNUMBER = {647807},
MRREVIEWER = {D.\ Newton},
       DOI = {10.2307/1971392},
}

@article {Thieullen1987,
    AUTHOR = {Thieullen, P.},
     TITLE = {Fibr\'es dynamiques asymptotiquement compacts. {E}xposants de
              {L}yapounov. {E}ntropie. {D}imension},
   JOURNAL = {Ann. Inst. H. Poincar\'e{} Anal. Non Lin\'eaire},
  FJOURNAL = {Annales de l'Institut Henri Poincar\'e. Analyse Non
              Lin\'eaire},
    VOLUME = {4},
      YEAR = {1987},
    NUMBER = {1},
     PAGES = {49--97},
      ISSN = {0294-1449},
   MRCLASS = {58F11 (28D20 34D05 58F15)},
  MRNUMBER = {877991},
MRREVIEWER = {Massimo\ Campanino},
}

@article {Walters1993,
    AUTHOR = {Walters, Peter},
     TITLE = {A dynamical proof of the multiplicative ergodic theorem},
   JOURNAL = {Trans. Amer. Math. Soc.},
  FJOURNAL = {Transactions of the American Mathematical Society},
    VOLUME = {335},
      YEAR = {1993},
    NUMBER = {1},
     PAGES = {245--257},
   MRCLASS = {28D05 (58F11)},
  MRNUMBER = {1073779},
MRREVIEWER = {U.\ Krengel},
       DOI = {10.2307/2154267},
}

@book {Wojtaszczyk1991,
    AUTHOR = {Wojtaszczyk, Przemysław},
     TITLE = {Banach spaces for analysts},
    SERIES = {Cambridge Studies in Advanced Mathematics},
    VOLUME = {25},
 PUBLISHER = {Cambridge University Press, Cambridge},
      YEAR = {1991},
     PAGES = {xiv+382},
   MRCLASS = {46-01 (30-01 43-01 46Bxx 46Exx 47Bxx)},
  MRNUMBER = {1144277},
MRREVIEWER = {Raymond\ H.\ Cox},
       DOI = {10.1017/CBO9780511608735},
}

@article {kingman1968ergodic,
    AUTHOR = {Kingman, J. F. C.},
     TITLE = {The ergodic theory of subadditive stochastic processes},
   JOURNAL = {J. Roy. Statist. Soc. Ser. B},
  FJOURNAL = {Journal of the Royal Statistical Society. Series B.
              Methodological},
    VOLUME = {30},
      YEAR = {1968},
     PAGES = {499--510},
      ISSN = {0035-9246},
   MRCLASS = {60.30},
  MRNUMBER = {254907},
MRREVIEWER = {W.\ F.\ Stout},
}

@article {kingman1973subadditive,
    AUTHOR = {Kingman, J. F. C.},
     TITLE = {Subadditive ergodic theory},
   JOURNAL = {Ann. Probability},
  FJOURNAL = {The Annals of Probability},
    VOLUME = {1},
      YEAR = {1973},
     PAGES = {883--909},
   MRCLASS = {60F15 (28A65)},
  MRNUMBER = {356192},
MRREVIEWER = {K.\ Jacobs},
       DOI = {10.1214/aop/1176996798},
}

@article {Lee2024,
    AUTHOR = {Lee, George},
     TITLE = {Constructing the {O}seledets decomposition with subspace
              growth estimates},
   JOURNAL = {Trans. Amer. Math. Soc. Ser. B},
  FJOURNAL = {Transactions of the American Mathematical Society. Series B},
    VOLUME = {11},
      YEAR = {2024},
     PAGES = {396--419},
   MRCLASS = {37H15},
  MRNUMBER = {4706374},
MRREVIEWER = {Luciana\ A.\ Alves},
       DOI = {10.1090/btran/146},
}

@phdthesis{horan2020spectral,
  title={Spectral gap and asymptotics for a family of cocycles of Perron-Frobenius operators},
  author={Horan, Joseph A.},
  year={2020},
  school={University of Victoria}
}

@phdthesis{Lee2023,
  author  = {Lee, George},
  title   = {Some results on linear dynamical systems},
  school  = {University of Victoria},
  year    = {2023},
  type    = {Ph.D. thesis}
}

@article {Pettis1938,
    AUTHOR = {Pettis, B. J.},
     TITLE = {On integration in vector spaces},
   JOURNAL = {Trans. Amer. Math. Soc.},
  FJOURNAL = {Transactions of the American Mathematical Society},
    VOLUME = {44},
      YEAR = {1938},
    NUMBER = {2},
     PAGES = {277--304},
   MRCLASS = {28B05 (46G10)},
  MRNUMBER = {1501970},
       DOI = {10.2307/1989973},
}

@article {GouezelStoyanov2017,
    AUTHOR = {Gou\"ezel, S\'ebastien and Stoyanov, Luchezar},
     TITLE = {Quantitative {P}esin theory for {A}nosov diffeomorphisms and
              flows},
   JOURNAL = {Ergodic Theory Dynam. Systems},
  FJOURNAL = {Ergodic Theory and Dynamical Systems},
    VOLUME = {39},
      YEAR = {2019},
    NUMBER = {1},
     PAGES = {159--200},
   MRCLASS = {37D30 (37E30 37H15)},
  MRNUMBER = {3881129},
MRREVIEWER = {Rafael\ Potrie},
       DOI = {10.1017/etds.2017.25},
}

@article {GonzalezTokmanQuas2022,
    AUTHOR = {Gonz\'alez-Tokman, C. and Quas, A.},
     TITLE = {Lyapunov exponents for transfer operator cocycles of
              metastable maps: a quarantine approach},
   JOURNAL = {Trans. Moscow Math. Soc.},
  FJOURNAL = {Transactions of the Moscow Mathematical Society},
    VOLUME = {82},
      YEAR = {2021},
     PAGES = {65--76},
   MRCLASS = {37H15},
  MRNUMBER = {4397152},
       DOI = {10.1090/mosc/313},
}

@article {Buzzi1999,
    AUTHOR = {Buzzi, J\'er\^ome},
     TITLE = {Exponential decay of correlations for random {L}asota-{Y}orke
              maps},
   JOURNAL = {Comm. Math. Phys.},
  FJOURNAL = {Communications in Mathematical Physics},
    VOLUME = {208},
      YEAR = {1999},
    NUMBER = {1},
     PAGES = {25--54},
   MRCLASS = {37H15 (37H05)},
  MRNUMBER = {1729876},
MRREVIEWER = {Bernard\ Schmitt},
       DOI = {10.1007/s002200050746},
}

@book {ambrosio2000bv,
    AUTHOR = {Ambrosio, Luigi and Fusco, Nicola and Pallara, Diego},
     TITLE = {Functions of bounded variation and free discontinuity
              problems},
    SERIES = {Oxford Mathematical Monographs},
 PUBLISHER = {The Clarendon Press, Oxford University Press, New York},
      YEAR = {2000},
     PAGES = {xviii+434},
      ISBN = {0-19-850245-1},
   MRCLASS = {49-02 (49J45 49K10 49Qxx)},
  MRNUMBER = {1857292},
MRREVIEWER = {J.\ E.\ Brothers},
}

@book {ConwayFunctionalAnalysis,
    AUTHOR = {Conway, John B.},
     TITLE = {A course in functional analysis},
    SERIES = {Graduate Texts in Mathematics},
    VOLUME = {96},
 PUBLISHER = {Springer-Verlag, New York},
      YEAR = {1985},
     PAGES = {xiv+404},
   MRCLASS = {46-01 (47-01)},
  MRNUMBER = {768926},
MRREVIEWER = {Greg\ Robel},
       DOI = {10.1007/978-1-4757-3828-5},
}

@misc{DS21quenchedLimThms,
  author        = {Davor Dragi{\v{c}}evi{\'c} and Julien Sedro},
  title         = {Quenched limit theorems for expanding on average cocycles},
  year          = {2021},
  eprint        = {2105.00548v2},
  archivePrefix = {arXiv},
  primaryClass  = {math.DS}
}

@incollection {BarreiraPesin2001,
    AUTHOR = {Barreira, L. and Pesin, Ya.},
     TITLE = {Lectures on {L}yapunov exponents and smooth ergodic theory},
 BOOKTITLE = {Smooth ergodic theory and its applications ({S}eattle, {WA},
              1999)},
    SERIES = {Proc. Sympos. Pure Math.},
    VOLUME = {69},
     PAGES = {3--106},
 PUBLISHER = {Amer. Math. Soc., Providence, RI},
      YEAR = {2001},
   MRCLASS = {37D25 (34D08 37C40)},
  MRNUMBER = {1858534},
MRREVIEWER = {Miaohua\ Jiang},
       DOI = {10.1090/pspum/069/1858534},
}

@article {LasotaYorke1973,
    AUTHOR = {Lasota, A. and Yorke, James A.},
     TITLE = {On the existence of invariant measures for piecewise monotonic
              transformations},
   JOURNAL = {Trans. Amer. Math. Soc.},
  FJOURNAL = {Transactions of the American Mathematical Society},
    VOLUME = {186},
      YEAR = {1973},
     PAGES = {481--488 (1974)},
   MRCLASS = {28A70},
  MRNUMBER = {335758},
MRREVIEWER = {K.\ Krzy\.zewski},
       DOI = {10.2307/1996575},
}

@book {Baladi2000,
    AUTHOR = {Baladi, Viviane},
     TITLE = {Positive transfer operators and decay of correlations},
    SERIES = {Advanced Series in Nonlinear Dynamics},
    VOLUME = {16},
 PUBLISHER = {World Scientific Publishing Co., Inc., River Edge, NJ},
      YEAR = {2000},
     PAGES = {x+314},
   MRCLASS = {37C30 (37-02 37A25 37D20 47B38)},
  MRNUMBER = {1793194},
MRREVIEWER = {J\'er\^ome\ Buzzi},
       DOI = {10.1142/9789812813633},
}

@article {SchaumloffelFlandoli1991,
    AUTHOR = {Schauml\"offel, Kay-Uwe and Flandoli, Franco},
     TITLE = {A multiplicative ergodic theorem with applications to a first
              order stochastic hyperbolic equation in a bounded domain},
   JOURNAL = {Stochastics Stochastics Rep.},
  FJOURNAL = {Stochastics and Stochastics Reports},
    VOLUME = {34},
      YEAR = {1991},
    NUMBER = {3-4},
     PAGES = {241--255},
   MRCLASS = {60H15 (28D15 54H20 60E15)},
  MRNUMBER = {1124837},
MRREVIEWER = {Wolfgang\ Kliemann},
       DOI = {10.1080/17442509108833684},
}
\end{document}